\documentclass[11pt]{article}

\usepackage[toc,page,header]{appendix}
\usepackage[normalem]{ulem}
\usepackage{amsfonts, latexsym, amssymb, amsthm, amsmath, mathrsfs, esint, color, authblk, fullpage}
\usepackage {amsmath,refcount,makerobust}
\usepackage{amscd}
\usepackage{enumerate}
\usepackage{pstricks}
\usepackage{pst-node}
\usepackage{mathtools}
\usepackage{graphicx}
\usepackage{relsize}
\usepackage{hyperref}
\usepackage{comment}
\usepackage{scrlfile}
\usepackage{srcltx}
\usepackage{nicefrac}

\def\reg{{\rm reg}}

\def\per{\mbox{-}{\rm per}}

\def\ae{{\rm a.e. }}

\renewcommand{\b}{\beta}

\def\e{\varepsilon}

\def\vpot{\mathcal V_{\rm pot}^2}
\def\lpot{L_{\rm pot}^2(\Omega)}
\def\vsol{\mathcal V_{\rm sol}^2}
\def\lsol{L_{\rm sol}^2(\Omega)}

\newcommand{\R}{\mathbf{R}}
\def\Q{\mathbf{Q}}

\def\C{\mathcal{C}}

\def\AA{\mathcal{A}}

\def\DD{\mathcal{D}}

\def\OO{\mathcal{O}}
\def\o{\omega}
\def\l{\lambda}
\def\hom{{\rm hom}}
\def\dom{{\rm dom}}
\def\pot{{\rm pot}}

\newcommand{\Y}{\mathcal Y}

\renewcommand{\O}{\Omega}

\DeclareMathOperator{\Sp}{Sp}

\newcommand\B{B}

\newcommand\N{\mathbf{N}}
\newcommand\Z{\mathbf{Z}}

\newcommand\supp{\mathop{\operatorname{supp}}}
\newcommand\dist{\operatorname{dist}}

\DeclareMathOperator{\diam}{diam}

\newcommand\loc{{\rm loc}}

\def\b{\beta}

\def\XXint#1#2#3{{\setbox0=\hbox{$#1{#2#3}{\int}$}
     \vcenter{\hbox{$#2#3$}}\kern-.5\wd0}}

\newcounter{bei}

\newtheorem{theorem}{Theorem}[section]
\newtheorem{lemma}[theorem]{Lemma}
\newtheorem{proposition}[theorem]{Proposition}
\newtheorem{corollary}[theorem]{Corollary}
\newtheorem{assumption}[theorem]{Assumption}
\newtheorem{definition}[theorem]{Definition}
\newtheorem{remark}[theorem]{Remark}
\theoremstyle{remark}

\usepackage{mathtools} 
\usepackage[normalem]{ulem}
\newcommand{\randomspace}{\Omega}
\newcommand{\pspace}{S}
\newcommand{\randomelement}{\omega}

\newcommand{\randommeasure}{P}
\newcommand{\td}{\mathrm{d}}
\newcommand{\drandommeasure}{\mathrm{d}\randommeasure(\randomelement)}

\newcommand{\wtwoscale}{\xrightharpoonup{2}}

\newcommand{\stwoscale}{\xrightarrow{2}}
\newcommand{\weakly}{\rightharpoonup}

\newcommand{\eps}{\varepsilon}

\DeclarePairedDelimiter\norm{\lVert}{\rVert}

\DeclarePairedDelimiter\set{\{}{\}}

\newcommand\restrict[2]{{
  \left.\kern-\nulldelimiterspace 
  #1 
  \vphantom{\big|} 
  \right|_{#2} 
  }}

\usepackage{color}
\usepackage{ cancel }
\usepackage[color]{changebar}
\cbcolor{red}
\begin{document}
	
\title{{\sc High-contrast random composites: 
	homogenisation framework and new spectral phenomena}}


\author[1]{Mikhail Cherdantsev}
\author[2]{Kirill Cherednichenko}
\author[3]{Igor Vel\v{c}i\'{c}\,}
\affil[1]{School of Mathematics, Cardiff University, Senghennydd Road, Cardiff, CF24 4AG, United Kingdom, CherdantsevM@cardiff.ac.uk, ORCID 0000-0002-5175-5767}
\affil[2]{Department of Mathematical Sciences, University of Bath, Claverton Down, Bath, BA2 7AY, United Kingdom, k.cherednichenko@bath.ac.uk, ORCID 0000-0002-0998-7820}
\affil[3]{Faculty of Electrical Engineering and Computing, University of Zagreb, Unska 3, 10000 Zagreb, Croatia, Igor.Velcic@fer.hr, ORCID 0000-0003-2494-2230}
\date{}
\maketitle

\vspace{-6mm}

\noindent {\bf Acknowledgements:} The research of  Mikhail Cherdantsev was supported by the Leverhulme Trust Research Project Grant RPG-2019-240. Kirill Cherednichenko acknowledges the support of Engineering and Physical Sciences Research Council, Grants EP/L018802/2, EP/V013025/1. Igor Vel\v{c}i\'c is grateful for the support of Croatian Science Foundation under Grant Agreement no.~IP-2018-01-8904 (Homdirestroptcm) and Grant agreement IP-2022-10-5181 (HOMeOS). The authors thank Prof. Andrey Piatnitski for useful comments.

\vskip 0.15cm

\begin{abstract}
	 We develop a framework for multiscale analysis of elliptic operators with high-contrast random coefficients.  For a general class of such operators, we provide a detailed spectral analysis of the corresponding homogenised limit operator.  Under some lenient assumptions on the configuration of the random inclusions, we fully characterise the limit of the spectra of the high-contrast operators in question, which unlike in the periodic setting is shown to be different to the spectrum of the homogenised operator.   Introducing a new notion of the {\it relevant limiting spectrum}, we describe the connection between these two sets.

\vskip 0.15cm


\vskip 0.15cm


\vskip 0.4cm

\end{abstract}

\tableofcontents

\section{Introduction}


A recent drive towards the development of mathematical tools for understanding the behaviour of realistic inhomogeneous media has led to a renewal of interest to stochastic homogenisation and to an equally explosive activity in the analysis of composite media, in particular, of those with contrasting material components. In those fields the new step change is characterised by moving from qualitative analysis to quantitative results: from the general multiscale frameworks by De Giorgi and Spagnolo \cite{DS}, Bakhvalov \cite{Bakhvalov}, Murat--Tartar \cite{Tartar}, \cite{MurTar}, Allaire \cite{Allaire},  Kamotski--Smyshlyaev \cite{KamSm19} (in the periodic setting),  Yurinskii \cite{Yurinskii},  Papanicolaou-Varadhan \cite{PapVar}, Kozlov \cite{Kozlov}, Zhikov--Piatnitski \cite{zhikov1} (in the stochastic setting), Zhikov \cite{Zhikov2000},\cite{zhikov2005} (in the high-contrast periodic setting)  to error bounds in appropriate functional topologies by Griso \cite{Griso}, Zhikov--Pastukhova \cite{ZhPast}, Birman--Suslina \cite{BirmanSuslina},  Kenig--Lin--Shen \cite{Kenig} (periodic),  Gloria--Otto \cite{GloriaOtto}, Armstrong--Smart \cite{ArmstrongSmart}, Armstrong--Kuusi--Mourrat \cite{Armstrong_book} (stochastic),  Cherednichenko--Cooper \cite{ChC}, Cherednichenko--Ershova--Kiselev \cite{ChErK}, Cooper--Kamotski--Smyshlyaev \cite{CKS2023} (high-contrast periodic).

The high-contrast setting occupies a special place in homogenisation theory, for it enables resonant phenomena leading to micro-to-macro scale interactions,  thereby bringing about new, often not naturally occurring, material properties. While the periodic high-contrast setting has now been analysed in good detail, its stochastic  counterpart, which can be argued to be even more relevant to applications in material science, has not yet enjoyed the deserved attention: apart from our tentative study \cite{ChChV} in the bounded domain setting, we are aware of only three more works in  this area  ---  \cite{BoMiPi}, \cite{BBM} and \cite{APZ}. Note that in \cite{BoMiPi} the authors use the notion of stochastic two-scale convergence {\it  in mean}, introduced in \cite{BoMiWr}, which is not convenient for the spectral analysis presented here (or  in \cite{ChChV}). Indeed,  the convergence in mean does not imply the convergence almost surely, and it is not clear what are the implications of the resolvent convergence in mean in relation to the limiting spectrum. On the other hand,  the notion of stochastic two-scale convergence introduced in \cite{zhikov1} provides necessary tools for the analysis of the limiting spectrum in the case of a bounded domain, see \cite{ChChV} and Theorem \ref{th:3.3} below.   The goal of the present paper is to  provide a general framework and a comprehensive toolbox for multiscale analysis of random high-contrast media, which would lay a foundation for the related research avenue.   

We study the problem of homogenisation of operators of the form  $\AA^\e = -\nabla \cdot A^\e \nabla$ with  high-contrast random (stochastic) coefficients, represented by the matrix $A^\e$, which models a two-component material with randomly distributed and randomly shaped ``soft'' inclusions, whose typical size and spacing are both of order $\e \ll 1$, embedded in a ``stiff'' component.  It is assumed that the ellipticity constants of $A^\e$ is of order $1$ in the stiff component and of order $\e^2$ in the inclusions. This scaling regime between the size of the microstructure and the ratio between the coefficients is often referred as the double porosity model.

Our interest in high-contrast homogenisation problems is motivated by the band-gap structure of the spectra of the associated  operators. We emphasise that this spectral phenomenon has only been observed in the homogenisation limit under the double porosity scaling. Therefore, while a more general non-uniformly elliptic setting can be very challenging, see e.g. \cite{BFO2018}, in order to obtain the specific spectral behaviour one has to work with  particular geometric constraints and scaling regimes.  Composites exhibiting spectral gaps are widely used for manipulating acoustic and electromagnetic waves, see e.g. \cite{Joannopoulos}, \cite{Khelif}. These were first analysed from the mathematically rigorous perspective in \cite{Zhikov2000,zhikov2005} in the periodic setting. It was shown that the spectra of  $\AA^\e$ converge in the sense of Hausdorff to the spectrum of a limit homogenised operator $\AA^\hom $. The latter has a two-scale structure that captures the macro- and microscopic behaviour of the operator $\AA^\e$ for small values of $\varepsilon.$ The spectrum of 
$\AA^\hom $ has a band-gap structure (in the whole-space setting), characterised with respect to the spectral parameter $\lambda$ by a  function $\beta(\l),$ that is explicitly determined by the microscopic part of $\AA^\hom$ and quantifies the resonant (or anti-resonant) effect of the soft inclusions. In the case of a bounded domain the analysis of the problem and the results are very similar, in particular, the point spectrum of $\AA^\hom$ ``populates'' the bands (accumulating at their right ends) corresponding to the whole-space case. 

In \cite{ChChV} we considered the high-contrast stochastic homogenisation problem in a bounded domain. To a certain extent, our findings as well as the basic techniques (modulo replacing the reference periodicity cell with the probability space, and the standard two-scale convergence with its stochastic counterpart) were similar to those of \cite{Zhikov2000} in the periodic case. Namely, we  showed that the homogenised operator has a similar two-scale structure with the ``macroscopic'' component $-\nabla \cdot A_1^\hom \nabla$ acting in the physical space and the ``microscopic'' one $-\Delta_\OO$ acting in the probability space on a prototype ``inclusion'' $\OO$. We  proved an appropriate (i.e. stochastic two-scale) version of the resolvent convergence of $\AA^\e$ to  $\AA^\hom $ and the Hausdorff convergence of their spectra. However, due to the technical challenges of the stochastic setting our understanding of the homogenised operator $\AA^\hom$ was limited. In particular, we were able to describe its spectrum only for a range of explicit examples. 
Note that the stochastic two-scale resolvent convergence of the operators (implying  $\lim_{\e\to 0} \Sp(\AA^\e) \supset \Sp(\AA^\hom)$), proven in \cite{ChChV}, is valid both for a bounded domain and for the whole space without any changes to the proof. 

Unlike for the periodic high-contrast operators, whose spectra are described by similar multiscale arguments for a bounded domain and for the whole space (leading to closely related features in the two cases, albeit resulting in spectra of different types from the operator-theoretic perspective),  in the stochastic setting the situation is fundamentally different.
 In the present paper we show that in the case of the whole space  the spectrum $\Sp(\AA^\hom )$ is, in general, a proper subset of the  limit of $\Sp(\AA^\e)$. In fact, the situation when the spectrum of $\AA^\e$ occupies the whole positive half-line is not uncommon. This additional part of the limit spectrum, which does not appear in the case of a bounded domain,
is attributed to the stochastic nature of the problem. At the same time,  at least in the context of the applications that we have in mind, this part of the limit spectrum could (arguably) be  deemed ``physically irrelevant", whereas the part of the spectrum of $\AA^\e$ that converges to $\Sp(\AA^\hom)$ may be thought of as ``physically relevant". In terms of  the spectral analysis our main objectives are to a) provide a comprehensive analysis of the homogenised operator $\AA^\hom$ and its spectrum; b) characterise the limit of  $\Sp(\AA^\e)$; c) understand the relation between  the  limit spectrum and  $\Sp(\AA^\hom)$.

We next compare our results with those of  \cite{APZ}, which also considers a random high-contrast medium in the whole space. Under the assumption that the soft inclusions are copies (up to rotations) of a finite number of $C^2$-smooth geometric shapes, the authors of  \cite{APZ} construct the limit operator and prove the strong convergence of associated semigroups for any finite time. As a consequence, they show that the spectrum of the limit operator is a subset of the limit of the spectra of the original operators. Our assumptions on the types of the inclusions and their regularity are somewhat more general; in particular, we allow uncountably many inclusion shapes. As we already mentioned, we provide a full characterisation of the limiting spectrum in the case of  finite-range correlations and we use different techniques (which are closer to Zhikov's approach).  The overall focus of   \cite{APZ} is also different. Namely,  each second-order divergence-form operators is the generator of a Markov semigroup. It is well known that in the high-contrast setting the  evolution of the effective limit operator exhibits memory effects, which means  that corresponding process is not Markovian. The goal of their work is to equip the coordinate process with additional components so that the dynamics of the enlarged process remains Markovian in the limit. 

In \cite{BBM} the authors consider a scattering problem for Maxwell equations with an obstacle comprising of parallel random  rods with high random permittivity. In particular, in the homogenisation limit, they obtain a formula for the {\it effective} permeability, which, in the context of electromagnetism, is an analogue of Zhikov's $\beta$-function. It should be noted that their results are obtained for a specific random model, rather than for a general class of high-contrast random composites, as in the present paper. Similar to \cite{APZ}, they only analyse the spectrum of the two-scale limit operator, yet do not describe the limiting spectrum, which is strictly larger in their setting.

We next outline the structure of the paper and discuss our results in more detail. In Section \ref{stochastic} we recall the basic facts and definitions of the probability framework, state the main assumptions, describe the problem, and give a brief overview of the results of \cite{ChChV}. We have simplified the main assumption of  \cite{ChChV}, requiring only uniform Lipschitz property (more precisely, uniform minimal smoothness) and boundedness of the inclusions. In particular, these assumptions guarantee an extension property (Theorem \ref{th:extension}) and a density result (Lemma \ref{lemmagloria}). 

In Section \ref{lim_eq} we provide a complete characterisation of the spectrum of $\AA^\hom$ via the spectra of the operators $-\Delta_\OO$ and $-\nabla \cdot A_1^\hom \nabla$ and the stochastic analogue of Zhikov's $\beta$-function, defined via the solution to a resolvent problem for $-\Delta_\OO$, see Theorem \ref{th:3.3}. We then study the properties of $\beta = \beta(\l)$ and provide a formula for its recovery from one {\it typical} realisation (in accordance with the ergodicity framework),  construct and analyse the resolution of identity for the operator $-\Delta_\OO$, as well as characterise the point spectrum of $\AA^\hom$. In the same section we improve and generalise results of \cite{ChChV} by providing a description of the spectra of $-\Delta_\OO$ and $\AA^\hom$ under general assumptions rather than for specific  examples. In particular, this description implies  (again, under  general assumptions) the convergence of spectra $\lim_{\e  \to 0} \Sp(\AA^\e) = \Sp (\AA^\hom)$ in the case of a bounded domain. (In what follows we will simply write $\lim \Sp(\AA^\e)$ instead of $\lim_{\e  \to 0} \Sp(\AA^\e)$ for brevity.) The main results of this section  are a stochastic analogue of those in the periodic case. (In the case of systems the function $\beta$ is matrix valued, allowing for even richer spectral behaviour, see \cite{CaChVsyst}.)

Section \ref{s:conv of spectrum} is concerned with the study of the limit of $\Sp(\AA^\e)$. To that end, we introduce a function $\beta_\infty(\l,\o)$ (where $\o$ is an element of the probability space). We use $\beta_\infty(\l,\o)$ to define a set $\mathcal G$ which, as we prove later, contains the limit spectrum. While $\beta(\l)$ in the present setting is the stochastic analogue of its counterpart in the periodic case, $\beta_\infty(\l,\o)$ is a new object which has no analogues in the periodic setting.
An intuitive explanation of the difference between the two  is as follows: the values of $\beta(\l)$ are determined via the ergodic limit (in other words, from the global average distribution of inclusions), whereas the values of $\beta_\infty(\l)$ are determined, loosely speaking, by the areas with the least dense distribution of the inclusions ({\it non-typical areas}). The first of the two main results of the section (Theorem \ref{th4.1}) is that the  limit of  $\Sp(\AA^\e)$ is a subset of $\mathcal G$. We do not know whether $\mathcal G$ is the actual limit in general. However, under an additional assumption of finite range of dependence of the spacial distribution of inclusions, we establish that $\lim \Sp(\AA^\e) = \mathcal G$  (Theorem \ref{th4.4}), which is the second main result of the section.   We give the proof of Theorem \ref{th4.1} in Section \ref{ss:4.3}, while in  Section \ref{ss:4.4} we provide auxiliary statements on the existence of cubes with almost periodic arrangements of inclusions. Finally, in Section \ref{ss:4.5} we prove Theorem \ref{th4.4}. In Section \ref{cubesfilled} we provide several examples illustrating our results: namely, we consider a (non-periodic) example where $\beta_\infty = \b$ (and, thus, $\lim \Sp(\AA^\e) = \Sp(\AA^\hom)$) and those where this is not the case, i.e. $\lim \Sp(\AA^\e) \neq \Sp(\AA^\hom)$.

In Section \ref{s:relevant} we explore the connection between the spectrum of $\AA^\hom$ and the limiting behaviour (as $\e\to 0$) of the spectra of $\AA^\e$. For the family   $\AA^\e$, we introduce a notion of the \textit{relevant limiting spectrum}, denoted by ${\mathcal R}\mbox{-}\lim \Sp(\AA^\e)$ --- namely, a subset of  $\lim \Sp(\AA^\e)$ whose points are characterised by the existence of  approximate eigenfunctions with significant part of their energy remaining inside a (large) fixed neighbourhood of the origin as $\e\to 0$.  On the other hand, those $\l$ that have  (approximate) eigenfunctions with energy mainly concentrated in regions with  non-typical (in the same sense as above) inclusion distributions, and thus located far  from the origin for small $\e$, constitute the irrelevant limiting spectrum. We show that  ${\mathcal R}\mbox{-}\lim \Sp(\AA^\e)$ coincides with $\Sp(\AA^\hom)$, see Theorem \ref{th:9.2}. We then discuss the implications of this result for the parabolic and hyperbolic evolution semigroups, cf. Corollary \ref{cor: 6.6}: namely, we show that on any finite domain one can neglect,  as $\e\to 0$, the part of the initial conditions corresponding to the irrelevant limiting spectrum.

A number of technical preliminaries, constructions and auxiliary statements that we use in the paper are presented in the appendices.  In particular, we prove the following results: a higher regularity of the periodic homogenisation corrector for perforated domains, see Theorem \ref{th:6.18}, which is discussed in Appendix \ref{ap:regul}; the extension property for potential vector fields in the probability space, see Proposition \ref{p:extensionm}; and a corollary of the latter, see Lemma \ref{l7.6}, which is important for the stochastic homogenisation corrector for  perforated domains.

  \section{Notation index}
  In order to help the reader navigate the text, we provide a short description of some important notation and indicate where it is introduced.

 		 	\begin{itemize} \setlength\itemsep{0em}
 		 		
\item[]	 		 {\bf General notation:}
 			 		 
\item[--]  	For a   Lebesgue measurable set $U\in \R^d$, we denote by $|U|$ its measure;
 		
\item[--]  	$U^c$ is the complement of a set $U$;

\item[--]  	$\overline U$ is the closure of a set $U$ in a relevant topology;

\item[--]  ${\mathbf 1}_U$ is the indicator function of a set $U$;

 \item[--]  	$B_R(x)$ is the ball of radius $R$ centred at $x$, and $B_R:=B_R(0)$;
 		
\item[--]  	$x_j$ is the   $j$-th component of a vector $x\in \R^d$;

\item[--]  	${\mathcal Re}(z)$ is the real part of $z\in \mathbb C$; 		

 \item[--]  	${\rm Dom} (\cdot)$ id the domain of an operator or a function, depending on the context;

 \item[--]  $\rm Sp(\cdot)$ stands for the spectrum of an operator.

  {\bf Section \ref{stochastic}:}
   
    \item[--]  $(\Omega, \mathcal{F},P)$ is a complete probability space;
      
   \item[--]  $T_x$ is a dynamical system on $(\Omega, \mathcal{F},P)$;
   
   \item[--]  $\OO\in\Omega$ is a reference set for the set of inclusions;

 \item[--]    $\OO_\o$, $\OO_\o^k$, $\mathcal B_\o^k$ are: the set of inclusions, an individual inclusion, and its extension set, respectively;
 	
\item[--]  $\rho,\mathcal N,\gamma$ are the constants of minimal smoothness;
 	
\item[--]  $D_\o^k$ is the ``left-bottom'' vertex of the minimal cube containing $\OO_\o^k$;
 	
\item[--]  $\square$ is the unit cube centred at the origin;
 	
\item[--]  $S_0^\e(\o)$,  $S_1^\e(\o)$, $\chi_1^\e(\omega), \chi_0^\e(\omega)$ are the set of ($\e$-scaled) inclusions, its complement, and their characteristic functions;

\item[--]  $A_1^\hom$ is the matrix of homogenised coefficients for the stiff component; 

\item[--]  $\AA^\e(\o)$ and $\AA^\hom$ are the random high-contrast operator and the corresponding two-scale homogenised operator;

\item[--]  $E^\e_{(-\infty,\l]}$ and $ 	E^\hom_{(-\infty,\l]}$ are the spectral projections of $\AA^\e$ and $\AA^\hom$, respectively;

\item[--]  $\Delta_\OO$ is the probabilistic ``Dirichlet'' Laplace operator on $\OO$;

\item[--]  for an open set $U\subset \R^d$ we denote by $\Delta_{U}$   the Dirichlet Laplace operator on $U$;
 	
\item[--]  $H$ and $V$ are functions spaces in which the two-scale homogenised operator acts;

\item[--]  $\langle f \rangle: =\int_\Omega f$.

{\bf Section \ref{lim_eq}:}

\item[--]  $\beta(\l)$ is the stochastic version of Zhikov's $\beta$-function;

\item[--]  $P_\o$ is an appropriately shifted inclusion which contains the origin;

\item[--]  $\Lambda_s$ and $\Psi_s^p$ are eigenvalues and eigenfunctions of $-\Delta_{P_\o}$;

\item[--]  $d_\l : = \dist(\l, \Sp(-\Delta_\OO))$;

\item[--]  $E_{[0,t]}$ is the resolution of identity for  $-\Delta_{\OO}$.

{\bf Section \ref{s:conv of spectrum}:}

\item[--]  $\beta_\infty(\l)$ is a ``local-global'' analogue of $\beta(\l)$;

\item[--]  $\ell(x,M,\l,\o)$ is the local spectral average;

 \item[--]  	$\square_x^M$ is the cube of edge length $M$ centred at $x$, $\square^M$ is the cube of edge length $M$ centred at the origin;

\item[--]  $\mathcal G$ is the ``upper bound'' for the limit of the spectra of $\AA^\e$;

\item[--]  $\hat N_j$, $N_j$, $j=1,\ldots,d$, are the periodic homogenisation correctors on a (large) cube and their shift;

\item[--]   $\fint$ denotes the average value;

\item[--]  $a^\e$ is the bilinear form associated with $\AA^\e + 1$;

 {\bf	Section \ref{s:relevant}:}
 	
 	\item[--]   $\hat N_j^\e$,  $j=1,\ldots,d$, are the  homogenisation correctors;
 	
 	\item[--]  $g_j, g_j^\e, G_j^\e, j=1,\dots,d,$ are the difference of fluxes in the probability space, its realisation in physical space, and the flux corrector, respectively.
 	
 {\bf	Appendix \ref{ap:probability}:}
 	 	
 	\item[--]  $\nabla_\o$ is the gradient in the probability space;
 	 	
 	\item[--]  $\mathcal{C}^{\infty} (\Omega)$, $\mathcal{C}_0^{\infty} (\mathcal{O})$, $W^{k,2} (\Omega),  W^{\infty,2} (\Omega)$ and $W_0^{1,2}(\mathcal{O})$ are the standard spaces of functions on $\Omega$;
 	 	
 	\item[--]  $\lpot$, $\lsol$, $\vpot$ and $\vsol$ are the spaces of potential and solenoidal vector fields on $\Omega$ and their zero-mean subspaces.
 	
 	\end{itemize}

\section{Preliminaries and problem setting} \label{stochastic}
The setting of the problem is similar to \cite{ChChV}; however, we  simplify the main assumption of \cite{ChChV} and deal mostly with the whole space $\R^d$ rather than  a bounded domain. We provide some standard definitions, such as those of Sobolev spaces of functions defined on a probability space and the stochastic two-scale convergence, as well as technical statements, such as measurability properties of various mappings in Appendixes \ref{ap:probability} and \ref{s:auxiliary}. We will often use these tacitly throughout the text.
\subsection{Probability framework} \label{s:3.1}
Let $(\Omega, \mathcal{F},P)$ be a complete probability space. We assume that the $\sigma$-algebra $\mathcal{F}$ is countably generated, which implies that the spaces $L^p(\Omega)$, $p \in [1,\infty),$ are separable. 
\begin{definition}\label{defgroup}
	A family $\{T_x\}_{x \in \R^d}$ of measurable bijective mappings $T_x:\Omega \to \Omega$ on a probability space $(\Omega, \mathcal{F}, P)$ is called a dynamical system  if:
	\begin{enumerate}
		\item $T_x \circ T_y=T_{x+y}\ $ $\forall\,x,y \in \R^d$;
		\item $P(T_{x} F )=P(F)\ $ $\forall x \in \R^d$, $F \in \mathcal{F}$;
		\item $\mathcal{T}: \R^d \times\Omega \to \Omega,\ $ $(x, \omega)\to T_x (\omega)$ is measurable (for the standard $\sigma$-algebra on the product space, where on $\R^d$ we take the Borel $\sigma$-algebra).
	\end{enumerate}
\end{definition}

\begin{definition}\label{defergodic}
A dynamical system is called ergodic if one of the following equivalent conditions is fulfilled:
\begin{enumerate}
	\item $f$ measurable, $f(\omega)=f(T_x \omega) \ \forall x \in \R^d, \textrm{ \ae } \omega \in \Omega \implies f \textrm{\ is\ constant\ $P$-\ae\ } \omega \in \Omega$.
        \item\label{def:ergodicityb} $P\bigl((T_x B \cup B) \backslash (T_x B \cap B)\bigr)=0\ \forall x \in \R^d \implies P(B) \in \{0,1\}$.  
\end{enumerate}	
\end{definition} 
\begin{remark}\label{rem:ergo-weak}
	Note that if the symmetric difference of $T_x B$ and $B$ is empty, then the condition \ref{def:ergodicityb} in the above definition immediately implies $P(B) \in \set { 0,1}$. It can be shown (e.g., \cite{cornfeld}) that ergodicity is also equivalent to the following (a priori weaker) implication:
	\[
	T_x B = B\ \ \forall x \in \R^d \implies P(B) \in \set{0,1}.
	\]
\end{remark}

Henceforth we assume that the probability space $(\Omega, \mathcal{F},P)$ is equipped with an  ergodic dynamical system $T_x$. The dynamical system is used to establish a connection between the probability space $(\Omega, \mathcal{F},P)$  and the physical space $\R^d$. Namely, for $f \in L^p(\Omega)$ and \ae fixed $\o\in\O$ the expression  $f(T_x \omega), x\in\R^d,$ defines an element of $L^p_{\rm loc}(\R^d)$, which is called the {\it realisation} of $f$ for the given $\o$. Note that $f(T_x \omega)$ considered as a function of $x$ and $\o$ is an element of $L^p_{\rm loc}(\R^d, L^p(\Omega))$.

 The differentiation operations in $L^2(\O)$
 as well as the associated spaces $\mathcal{C}^{\infty} (\Omega)$, $W^{k,2} (\Omega), $ $ W^{\infty,2} (\Omega)$ are introduced in a standard way. We also use the spaces of potential and solenoidal vector fields $\lpot$ and $\lsol$ and their zero-mean subspaces $\vpot$ and $\vsol$, see \cite{zhikov2}. We provide the relevant definitions in Appendix  \ref{ap:probability}.

The following theorem is the key tool in setting the basis for our analysis.
\begin{theorem}[``Ergodic Theorem"] \label{thmergodic}
	Consider a probability space  $(\Omega,\mathcal{F}, P)$ and an ergodic dynamical system $\{T_x\}_{x \in \R^d}$ on $\Omega$. Suppose that $f \in L^1(\Omega)$ and ${S}\subset \R^d$ is a bounded open set. Then for $P$-\ae $\omega \in \Omega$ one has
        \begin{equation}\begin{aligned}\label{eq:birkhoff}
        \lim_{\e \to 0} \int_{{S}}f(T_{x/\e}\omega) \td x= |S| \int_{\Omega} f \td  P.
        \end{aligned}\end{equation}
	Furthermore, for all $f \in L^p(\Omega)$, $1 \leq p \leq \infty,$ and \ae $\omega \in \Omega$, the function $f (x, \omega)= f(T_x \omega)$ satisfies $f(\cdot, \omega) \in L^p_{\rm loc} (\R^d)$. For $p < \infty$ one has $f(\cdot/ \e, \omega)=f(T_{ \cdot/ \e} \omega ) \rightharpoonup \int_{\Omega} f \td P$ weakly in $L^p_{\rm loc}(\R^d )$ as $\e \to 0$. 
\end{theorem}
Note that the Ergodic Theorem  implies only that the identity \eqref{eq:birkhoff} holds $P$-\ae for countably many functions simultaneously. This is why in the definition of the two-scale convergence (Definition \ref{definicija1}) one has to restrict to a countable (dense) subset of test functions, see also e.g. \cite{HNV}.  In general,  statements that only require a simple convergence argument, such as passing to the limit in the resolvent problem, are valid a set of full probability measure. However, throughout the paper we have a number of results, notably Theorem \ref{th4.4} and Proposition \ref{propivan1}, where one does not invoke the ergodic theorem but uses the law of large numbers instead for the analysis of non-typical areas (as discussed in Introduction). Such results hold on a full-measure subset of $\Omega$ that in general is different from the one provided by the ergodic theorem. Thus in what follows all results can be viewed as valid on a set $\Omega_t$  of full measure, for whose elements a) the Ergodic theorem holds on a {\it suitable} dense subset of $L^1(\Omega)$, b) the application of the law of large numbers to a {\it suitable} collection of independent random variables (such as Theorem \ref{th:6.10}) is valid. We henceforth always assume that $\o\in \Omega_t$, and refer to it as `typical'.


 \subsection{The main assumption and the extension property} \label{s:3.2}
 We proceed with defining the set of inclusions $\mathcal{O}_{\omega}=\bigcup_{k=1}^\infty \mathcal{O}_{\omega}^k$.  As we have already mentioned, we  simplify the  assumptions of our previous paper \cite{ChChV}. It turns out that, in order to guarantee the validity of  the main results, it suffices to require the uniform minimal smoothness of the inclusions, their boundedness, and uniform separation.  In particular, Assumption \ref{kirill100} implies the key extension property, Theorem \ref{th:extension} (simply postulated in our previous paper \cite{ChChV}), and a density result, Lemma \ref{lemmagloria}. 
 
     We first recall the definition of a minimally smooth set, see \cite{Stein1970}. 
 \begin{definition}
 	An open set $U \subset \R^d$ is said to be minimally smooth with constants	$\rho, \mathcal N, \gamma$ (or $(\rho, \mathcal N, \gamma)$  minimally smooth), there is a countable sequence of open sets $\{U_i\}_{i=1}^\infty$  covering the boundary $\partial U$ such 	that
 	\begin{itemize}
 		\item  Each $x\in\R^d$ is contained in at most $\mathcal N$ of the  sets $U_i$.
 		\item For any $x\in\partial U$ the ball $B_\rho(x)$ is contained in at least one $U_i$.
 		\item For any $i\in\N$ the portion of the boundary $\partial U$ inside $U_i$ is, in some Cartesian frame,  the graph of a Lipschitz function whose Lipschitz	semi-norm is at most $\gamma$.
 	\end{itemize}
 \end{definition}

 Let a set $\mathcal{O}\subseteq \Omega$ be such that $0<P(\mathcal{O})<1$ and for each $\omega\in\Omega$ consider
 its {\it realisation}
 $$ \mathcal{O}_\omega:=\{x \in \R^d: T_x \omega \in\mathcal{O}\}. $$
 Our main assumption is as follows. 
\begin{assumption} \label{kirill100}
There exist constants $\rho, \mathcal N, \gamma,$ such that for \ae  $\omega \in \Omega$ the set $\R^d\setminus \overline{\OO_\o}$ is connected and
\begin{equation*}
\mathcal{O}_{\omega}=\bigcup_{k=1}^\infty \mathcal{O}_{\omega}^k,
\end{equation*}
where: 

\begin{enumerate}[1)]
\item For every  $k\in{\mathbf N}$ the set $\mathcal{O}_{\omega}^k$ is open, connected, and    $(\rho, \mathcal N, \gamma)$ minimally smooth;

\item For every $k \in \mathbf{N}$ one has $\diam \mathcal{O}_{\omega}^k< 1/2$. 

\item There exists $r_0 > 0$ (independent of $\o$) such that for all  $k, l \in \mathbf{N}$, $k\neq l$, one has $\dist(\mathcal{O}_{\omega}^k,\mathcal{O}_{\omega}^l) > r_0$. 
\end{enumerate}
 \end{assumption}

For the extension result below we need to make the following observation: under Assumption \ref{kirill100}, for every inclusion $ \mathcal{O}_{\omega}^k$ there exists an `extension domain'  $\mathcal{B}_\o^k $ with the following  properties: $\mathcal{B}_\o^k $ is open, bounded, $\mathcal{B}^k_{\omega}\supset \overline{\OO_\o^k}$, $\mathcal{B}^k_{\omega}\cap \OO_\o = \OO_\o^k$, and the set $\mathcal{B}^k_{\omega}\setminus \overline{\OO_\o^k}$ is   $(\rho, \mathcal N, \gamma)$ minimally smooth with possibly different constants $(\rho, \mathcal N, \gamma)$ uniform in $\omega$ and $k$. In order to make the notation simple, we choose the constants so that all the sets $\OO_\o^k$, $\mathcal{B}^k_{\omega}$, and $\mathcal{B}^k_{\omega}\setminus \overline{\OO_\o^k}$ are $(\rho, \mathcal N, \gamma)$ minimally smooth uniformly in $\omega$ and $k$. Moreover, we can assume that 
appropriately translated, the domains $\mathcal B_\o^k$  fits into the unit cube 
$\square:=[-\nicefrac{1}{2},\nicefrac{1}{2})^d.$
More precisely, let the vector $D_\o^k\in \R^d$ be defined by 
\begin{equation*}
	(D_\o^k)_j : = \inf\{x_j: x\in \OO_\o^k\},\, j=1,\dots,d,
\end{equation*}
and denote $d_{1/4}:=(1/4,\dots,1/4)^\top$, then $\OO_\o^k - D_\o^k - d_{1/4} \subset \mathcal{B}_\o^k - D_\o^k - d_{1/4} \subset \square $. 

We provide a sketch of the proof of the above claim for the $2$d case, which can be easily extended to any dimension. For sufficiently small $\tau >0$ we cover $\R^2$ with  closed squares with  sides $\tau$ and mutually disjoint interiors. For a given $\OO_\o^k$ we choose the squares that have non-empty intersection with $\OO_\o^k$ together with their ``neighbours'', i.e. with those squares that share with them a common side or vertex. The union of these squares, denoted by $U$, is obviously connected but may fail to be Lipschitz if it contains two squares that share a single vertex and the adjacent sides form part of the boundary of $U$. In order to rectify this problem we add to $U$ smaller squares with  sides $\frac{\tau}{4}$ with centres at all vertices contained in $U$.
 Choosing $\tau$ small enough we obtain an extension domain  $\mathcal{B}^k_{\omega}$ which is minimally smooth with constants $(\rho',\mathcal{N}', \gamma')$ depending only on $\tau$. The choice of $\tau$, in turn, is independent of $k$ and $\o$ and only depends on $(\rho,\mathcal{N}, \gamma)$  due to  Assumption \ref{kirill100}.

\begin{remark}
The requirement of separation of inclusions (perforations) is standard when dealing with problems in perforated domains, see e.g. \cite{GuillenKim}, \cite{PSZ}, as it allows one to use extension techniques. This assumption is even more important for  spectral analysis, as we already mentioned in the Introduction. Indeed, allowing inclusions to touch each other may have unpredictable consequences for the associated spectra, leading, in particular, to uncontrollable ``pollution'' of the gaps in the limiting spectrum by the spectrum arising from clusters of inclusions stuck together.  
\end{remark}

Next we formulate the mentioned extension property, which is valid under Assumption  \ref{kirill100}.
\begin{theorem}\label{th:extension}
	There exist bounded linear extension operators $E_k : W^{1,p}(\mathcal{B}^k_{\omega}\backslash {\mathcal{O}^k_{\omega}})  \to  W^{1,p}(\mathcal{B}^k_{\omega})$, $p\geq 1$, such that  the extension $\widetilde u:= E_k u$, $u \in W^{1,p}(\mathcal{B}^k_{\omega}\backslash {\mathcal{O}^k_{\omega}})$,   satisfies the bounds
	\begin{align}\label{3a}
			\|\nabla \widetilde u\|_{L^p(\mathcal{B}^k_{\omega})}\leq &C_{\rm ext}  \|\nabla u\|_{L^p(\mathcal{B}^k_{\omega}\backslash \mathcal{O}^k_{\omega})}, \quad\quad\quad\\[0.3em]
		\| \widetilde{u} \|_{L^p(\mathcal{B}^k_{\omega})}\leq &C_{\rm ext} \left(\|  u\|_{L^p(\mathcal{B}^k_{\omega}\backslash {\mathcal{O}^k_{\omega}})}+\| \nabla u\|_{L^p(\mathcal{B}^k_{\omega}\backslash {\mathcal{O}^k_{\omega}})} \right), \label{3b}
	\end{align}
where the constant $C_{\rm ext}$ depends only on $\rho,\mathcal N, \gamma, p$, and  is independent of $\o$ and $k$. Additionally, in the case $p=2$ the extension can be chosen to be harmonic in $\mathcal{O}^k_{\omega}$,
	\begin{align*}
	\Delta \widetilde u=0 \textrm{ on } \mathcal{O}^k_{\omega}.
\end{align*}
\end{theorem}
\begin{proof}
	Inequalities  \eqref{3a} and \eqref{3b} are direct  corollaries of a classical result due to Calder\'on and Stein on the existence of uniformly  bounded extension operators \cite[Chapter IV, Theorem 5]{Stein1970} and a uniform Poincar\'e inequality  \cite[Proposition 3.2]{GuillenKim} for minimally smooth domains. We make two observations. First, though it is not stated  explicitly in  \cite{Stein1970}, the norm of the extension operator depends only on the constants of the minimal smoothness assumption, which can be seen from the proof below. Second, the proof of the uniform Poincar\'e inequality in \cite[Proposition 3.2]{GuillenKim} works for any $p\geq 1$ without  amendments. 
	
	Now let $p=2$ and fix $\OO_\o^k$. We look for the harmonic extension $\widetilde{u}$ in the form $\widetilde{u} = E_k u + \widehat u$, where $\widehat u \in W^{1,2}_0(\OO_\o^k)$ (here $E_k$ denotes the extension operator from \cite{Stein1970}). If such $\widehat u$ exists, it  satisfies the identity
	\begin{equation}\label{4}
	  \int_{\OO_\o^k} \nabla \widehat u \cdot \nabla \varphi = -  \int_{\OO_\o^k} \nabla (E_k u) \cdot \nabla \varphi \quad \forall \varphi \in W^{1,2}_0(\OO_\o^k).
	\end{equation}
As per discussion above, we have
\[
\left|\int_{\OO_\o^k} \nabla (E_k u) \cdot \nabla \varphi  \right| \leq C_{\rm ext} \|\nabla u\|_{L^2(\mathcal{B}^k_{\omega}\backslash \mathcal{O}^k_{\omega})}\|\nabla \varphi\|_{L^2(\mathcal{O}^k_{\omega})},
\] 
uniformly in $\omega$ and $k$. Thus the right-hand side in \eqref{4} is a bounded linear functional on $W^{1,2}_0(\OO_\o^k)$ equipped with the norm $\|\nabla \varphi\|_{L^2(\mathcal{O}^k_{\omega})}$. Therefore, by the Riesz representation theorem the solution $\widehat u$ of \eqref{4} exists and satisfies the bound
$
\|\nabla \widehat u\|_{L^2(\mathcal{O}^k_{\omega})} \leq C_{\rm ext}  \|\nabla u\|_{L^2(\mathcal{B}^k_{\omega}\backslash \mathcal{O}^k_{\omega})}.
$
The bounds  \eqref{3a} and \eqref{3b} now follow easily.
\end{proof}

\begin{remark}\label{r:extension}
The above theorem can be reformulated in an obvious way for any family of domains (i.e. not related to the stochastic setting) satisfying the same minimal smoothness condition. 	
\end{remark}

\subsection{Problem setting and an overview of the  existing results}\label{s:3.3}

For $\e>0$ we define 
 \begin{equation*}
 S_0^\e(\omega):=\e \mathcal{O}_{\omega} = \bigcup_{k} \e \mathcal{O}^k_{\omega}, \quad S_1^\e(\omega) : = \R^d\setminus  \overline{S_0^\e(\omega)}.
\end{equation*}
The corresponding set indicator functions are denoted by $\chi_0^\e(\omega)$ and $\chi_1^\e(\omega)$, respectively. We consider the self-adjoint operator  $\mathcal{A}^{\e}(\omega)$  in $L^2(S)$, where $S \subset \R^d$ denotes either a bounded domain with Lipschitz boundary or the whole space $\R^d$, generated by the bilinear form 
\begin{equation*}
	\int_{S}A^{\e} (\cdot,\omega)\nabla u\cdot\nabla v,\qquad u, v\in W^{1,2}_0(S),
\end{equation*}
where
\begin{equation}\label{defAe} 
	A^\e(\cdot,\omega)= \chi_1^\e(\omega)A_1
+\e^2\chi_0^\e(\omega)I,
\end{equation}
and $A_1$ is a symmetric  positive-definite matrix. Note that $W_0^{1,2}(\R^d)=W^{1,2}(\R^d)$  when $S=\R^d.$ In what follows we assume that $\o$ is fixed  and will drop it from the notation of the operator, writing simply $\AA^\e$.

We recall the definition of the limit homogenised operator $\AA^\hom $ given in \cite{ChChV}. The matrix of homogenised coefficients $A_1^{\rm hom}$  arising from the stiff component is defined by 
\begin{equation}\label{Ahom}
A_1^{\rm hom}\xi \cdot \xi:=
\inf_{p\in \vpot}
\int_{\Omega \backslash \mathcal{O}} A_1(\xi+p) \cdot (\xi+p),\qquad\xi\in\R^d.
\end{equation}	 
It is well known, see \cite[Chapter 8]{zhikov2}, that in order for the matrix $A_1^{\rm hom}$ to be positive definite it suffices to have a certain kind of extension property. It is not difficult to see that a slight adaptation of the argument in \cite[Lemma 8.8]{zhikov2} ensures that $A_1^{\rm hom}$ is indeed positive definite under Assumption \ref{kirill100}. 

The variational problem \eqref{Ahom} has a unique solution $p_\xi$  in ${\mathcal X}$, where $\mathcal X$ denotes the completion in $L^2(\Omega \backslash \mathcal{O})$ of 
$\mathcal V_{\rm pot}^2$. This solution satisfies the identity
\begin{equation}\label{10}
	\int_{\O\setminus\OO} A_1(\xi +p_\xi) \cdot \varphi = 0 \quad \forall \varphi\in \vpot.
\end{equation}
In Appendix  \ref{ss:9.1} we prove that under our assumptions on the regularity of the inclusions the space $\mathcal X$ can be viewed merely as the restriction of $\vpot$ to $\O\setminus\OO$ (see Lemma \ref{l7.6}). 
 We define the space 
\[
H:=L^2(S)+L^2(S\times\OO),
\]
which is naturally embedded in $L^2(S\times \Omega)$, and its dense (by Lemma \ref{lemmagloria}) subspace 
$$V:=W_0^{1,2}(S)+L^2(S,W_0^{1,2}(\mathcal{O})),$$
where $W_0^{1,2}(\mathcal{O})$ is the Sobolev space of functions on $\Omega$ vanishing outside $\OO$, see Appendix \ref{ap:probability} for the precise definition.
  
 For $f\in H$, we consider the following resolvent problem: find $u_0+u_1 \in V$ such that 
\begin{equation}	
\begin{aligned} 
\int_S A_1^{\rm hom} \nabla u_0\cdot\nabla \varphi_0+\int_{S \times \Omega} \nabla_{\omega} u_1 \cdot\nabla_{\omega} \varphi_1&+\lambda \int_{S \times \Omega} (u_0+ u_1) (\varphi_0+\varphi_1)
\\[0.1em]
&
=\int_{S \times \Omega}f(\varphi_0+\varphi_1) \quad \forall \varphi_0 + \varphi_1 \in V,
\end{aligned}
\label{zadnjerrr} 
\end{equation}
where $\nabla_\o$ denotes the gradient in the probability space (see  Appendix  \ref{ap:probability}).
This problem gives rise to a positive definite operator $\AA^\hom $ in $H$, see \cite{ChChV} for details. 
Note that one can take  an arbitrary $f\in L^2(S\times\O)$ in (\ref{zadnjerrr}). In this case 
the solution 
coincides with the solution for the right-hand side $\mathcal{P}f$, where $\mathcal P$ is the orthogonal projection onto $H$, namely, we have $u_0 +u_1 = (\AA^\hom + \lambda)^{-1} \mathcal P f$. Furthermore, the identity (\ref{zadnjerrr}) can be written as a coupled system
\begin{align}
&\int_S A_1^{\rm hom} \nabla u_0\cdot\nabla \varphi_0 +\lambda \int_S\bigl(u_0+\langle u_1 \rangle\bigr) \varphi_0 = \int_S \langle f \rangle\varphi_0
\quad \forall \varphi_0 \in W_0^{1,2}(S), \label{mik1r} \\[0.3em]
&\int_\mathcal{O} \nabla_{\omega} u_1 (x, \cdot)\cdot\nabla_{\omega} \varphi_1 +\lambda  \int_\mathcal{O}\bigl(u_0(x)+u_1(x,\cdot)\bigr)  \varphi_1 = \int_\mathcal{O} f(x,\cdot)\varphi_1
\,\,\,\forall  \varphi_1 \in W_0^{1,2}(\mathcal{O}), \, \mbox{a.e. } x\in S. \label{mik2r}
\end{align}
Here and in what follows, unless specified otherwise,  we often use the notation $\langle f \rangle:=\int_\Omega f$, tacitly assuming that functions defined only on $\OO$ are extended by zero to the whole $\O$.

The following statement of  stochastic two-scale resolvent convergence (which we will refer to simply as ``resolvent convergence'' for brevity), which was proved in \cite{ChChV} in the case 
of bounded $S$, remains valid for $S=\R^d$ with exactly the same proof. (For the notion of weak (strong) stochastic two-scale convergence, denoted by $\wtwoscale (\stwoscale)$, we refer the reader to Appendix  \ref{ap:probability}.)
\begin{theorem} \label{misha10}
	Under Assumption \ref{kirill100}, let $\lambda>0$ and suppose that $f^{\eps}$ is a bounded sequence in $L^2(S)$ such that 	$f^{\eps} \wtwoscale (\stwoscale) f\in L^2(S \times \Omega).$ Consider the sequence of solutions $u^\e$ to the resolvent problem
	\begin{equation*}
	\AA^\e u^\e +\l u^\e = f^\e.
	\end{equation*}
	 Then for \ae $\omega \in \Omega$ one has $u^\e\wtwoscale (\stwoscale)  u_0+u_1 \in V$, where $u_0+u_1$ is the solution to \eqref{zadnjerrr}. 
\end{theorem} 

This resolvent convergence can be put in a more  abstract framework of resolvent convergence in variable spaces, see \cite{Past}. It is shown that the weak resolvent convergence is equivalent to the strong one.  Moreover, it entails convergence of hyperbolic semigroups and certain convergence properties for the spectra. 
 In particular, one has ``half'' of the Hausdorff convergence of the spectra:
\begin{equation}\label{12}
	\Sp( \AA^\hom) \subset \lim \Sp (\AA^\e),
\end{equation}
 where by $\lim \Sp (\AA^\e)$ we understand the set of all $\l$ such that  $\exists \{\l_\e\}$,  $\l_\e \in \Sp (\AA^\e)$, $\l_\e\to \l$.  Finally, the following (strong stochastic two-scale) convergence for the spectral projections of $\AA^\e$ and $\AA^\hom$ holds:
 \begin{equation*}
 	E^\e_{(-\infty,\l]} \to 	E^\hom_{(-\infty,\l]}
 \end{equation*}
 unless $\l$ is an eigenvalue of $\AA^\hom$.

This work pursues three main goals warranted by the property \eqref{12}: understanding the spectrum of $\AA^\hom$, obtaining a practically useful characterisation of the limiting spectrum $\lim \Sp (\AA^\e)$, and investigating the nature of the difference between the two sets --- $\lim \Sp (\AA^\e)\setminus \Sp(\AA^\hom)$.

We next report (partial) results on $\Sp (\AA^\hom)$ obtained in \cite{ChChV}. 

Denote by $-\Delta_\OO$  the positive definite self-adjoint operator generated by the bilinear form 
\begin{equation}
\int_\mathcal{O} \nabla_{\omega} u\cdot\nabla_{\omega}v,
\quad\quad u, v\in W_0^{1,2}(\mathcal{O}). 
\label{delta_omega_form}
\end{equation}

\begin{itemize}
	\item One has

	\begin{equation}\label{20a}
	\Sp(-\Delta_\OO) =	\overline{\bigcup_{k \in \mathbf{N} } \Sp\bigl(-\Delta_{\mathcal{O}^k_{\omega}}\bigr)} \mbox{ for \ae } \,\omega \in \Omega,
	\end{equation}
	where $-\Delta_{\mathcal{O}^k_{\omega}}$ is the Dirichlet Laplace operator on ${\mathcal{O}^k_{\omega}}$ for each $\omega, k.$ 
	\item 
	In the case of bounded $S$ and under the condition that 
	$\Sp(-\Delta_\OO) \subset \Sp(\AA^\hom )$,
	the strong stochastic two-scale convergence of eigenfunctions of $\AA^\e$ holds, i.e. if 
	$$ \AA^\varepsilon u^{\e}=\lambda^{\e} u^{\e},\quad \int_S|u^{\e}|^2=1, 
	$$
	and if $\lambda^{\e}\to \lambda$, then $u^\eps \stwoscale u$, 
	where 
	$ \AA^\hom u=\lambda u.$
	\item In the case of bounded $S,$ for two explicit examples of random distribution of inclusions, it was shown that $\Sp(-\Delta_\OO) \subset \Sp(\AA^\hom ),$ Zhikov's function 
	$\beta(\l)$ was calculated, and a complete characterisation of $\Sp(\AA^\hom )$ was given. 
\end{itemize}

In the next section we will extend and improve these results to the general setting of the present work. 

\begin{remark}\label{r:2.10}
	From \eqref{20a} and property 2 of Assumption \ref{kirill100} it follows, via the Poincar\'e inequality and  the min-max theorem, that $\inf \Sp(-\Delta_\OO) >0$.
\end{remark}

\begin{remark}
	The identity \eqref{20a} was proved in \cite{ChChV} under an additional assumption  that the normalised eigenfunctions of the operators $-\Delta_{\mathcal{O}^k_{\omega}}, \o\in \Omega, k\in \N,$ are bounded in the $L^\infty$-norm on any bounded spectral interval uniformly in $\o$ and $k$. It is not difficult to see that this property holds  under Assumption \ref{kirill100}: the proof of this fact follows closely the standard argument of the De Giorgi–Nash–Moser regularity theory, see, e.g. \cite[Chapter 2.3]{Bensoussan}, cf. a similar result in \cite[Lemma 4.3]{CaChV}, for more details. 
In fact,  one does not necessarily need a uniform bound of $L^\infty$-norms of the eigenfunctions of $-\Delta_{\mathcal{O}^k_{\omega}}$ in order to prove \eqref{20a}. Instead, it is sufficient to note that the random variable $\psi_{a,b}$ utilised in the proof of Theorem 5.1 in \cite{ChChV} belongs to $L^2(\Omega)$, which follows by a standard argument (cf. Proposition \ref{sasha100}) via integrating its realisation in the physical space first, then integrating over $\Omega$ and using the Fubini  theorem. 
\end{remark}

\begin{remark}
	Another, perhaps more common, approach to describing random media consists in taking  $\Omega$ to be a set of locally bounded Borel measures on $\mathbf{R}^d$, which is chosen to be a separable metric space (so the corresponding $\sigma$-algebra consists of Borel sets). Thus, the medium comprising a matrix medium and inclusions can be seen as a random measure. 
	The dynamical system is then defined using  the stationarity property of this measure --- however, the assumption of ergodicity  does not follow from this construction. In fact,  ergodicity is not necessary for a suitable version of the ergodic theorem to hold. Indeed, in its absence the expectation on the right-hand side of \eqref{eq:birkhoff} has to be replaced by a conditional expectation. The ergodicity property in the classical moderate contrast setting guarantees that the coefficients of the limit equation are deterministic. 
	In the case when ergodicity assumption is dropped, the coefficients of the limit equation are measurable with respect the $\sigma$-algebra of the invariant sets of the dynamical system. (Under the assumption of ergodicity this $\sigma$-algebra consists of sets of measure zero and one.)   
	
	Point processes such as the Poisson process and the random parking model, used in the present work in the construction of examples of a random medium, in probability theory  are usually looked at as random measures. A key benefit of this approach is that one can use Palm's theory (see \cite[Section 2.10, Section 2.11]{heid1}), which is especially useful in the analysis of random structures (see \cite{zhikov1}). In the present work we do not use  Palm's theory, as we find the adopted framework sufficient for our purposes. 
\end{remark}



\section{Homogenised operator and its spectrum}
\label{lim_eq}

In  this section we study the operator $\AA^\hom $ and its spectrum. In particular, our analysis  allows us to improve  a number of results of \cite{ChChV}. 

We begin with two general statements about  the spectrum of $\AA^\hom$, Proposition \ref{p:3.1} and Theorem \ref{th:3.3}, namely, which were proved in \cite{ChChV} only for certain examples. 

\begin{proposition} \label{p:3.1}
For the spectra of the operators ${\mathcal A}_{\rm hom}$, defined by the problem \eqref{mik1r}-\eqref{mik2r},  and $-\Delta_\OO$, defined by the form \eqref{delta_omega_form}, one has the inclusion $\Sp(-\Delta_\OO) \subset \Sp (\AA^\hom ).$ 
\end{proposition} 		
\begin{proof}
	Suppose that $\l\in \R$ is in the resolvent set of $\AA^\hom $, so that \eqref{zadnjerrr} has a solution $u=u_0+u_1$  for any $f\in H$. { First we take a non-trivial $f \in L^2(S)\setminus W^{1,2}(S)$. Note that in this case $\l u_0 +f$ does not vanish. For the corresponding solution  we have  
		$$ -\Delta_\OO u_1-\lambda u_1=(\lambda u_0+ f){\mathbf 1}_{\mathcal{O}},$$
		where ${\mathbf 1}_\OO$ is the indicator function of the set $\OO$.		Note also that for two arbitrary functions $w\in L^2(S\times \OO)$ and $h\in L^2(S)$ one has $\int_S w h \in L^2(\OO)$.} Then, multiplying the above identity by ${(\l u_0 +f)}\|\l u_0 +f\|^{-2}_{L^2(S)}$  and integrating over $S,$ we conclude that the function
	\begin{equation*}
	 \phi:=\int_S \frac{u_1 {(\l u_0 +f)}}{\|\l u_0 +f\|^2_{L^2(S)}} \in L^2(\OO)
	\end{equation*}
 solves the equation 
	$$ -\Delta_\OO \phi-\lambda \phi={\mathbf 1}_{\mathcal{O}}.$$ 
	 Next, we take $f = g \psi$ with arbitrary non-trivial $g\in L^2(S)$ and $\psi\in L^2(\mathcal O)$. Then for the corresponding solution of \eqref{zadnjerrr}, which we denote by $\tilde u_0 + \tilde u_1$, we have 
	$$ -\Delta_\OO \tilde u_1-\lambda \tilde u_1=\lambda \tilde u_0 {\mathbf 1}_{\mathcal{O}}+g\psi. $$ 
	The difference between $\tilde u_1$ and $\hat{u}_1:=\lambda\tilde u_0\phi$ 
	solves
	$$ -\Delta_\OO (\tilde u_1-\hat{u}_1)-\lambda (\tilde u_1-\hat{u}_1)=g\psi.$$
	Multiplying the last equation by ${g}\|g\|^{-2}_{L^2(S)}$ and integrating over $S,$ we see that	
	\begin{equation*}
	\breve u_1: = \int_S \frac{(\tilde u_1-\hat{u}_1) {g}}{\|g\|^2_{L^2(S)}}
	\end{equation*}
	is a solution of 
		\begin{equation*}
		-\Delta_\OO \breve u_1 -\lambda \breve u_1=\psi.
		\end{equation*} 
We have shown that $-\Delta_\OO - \l I$ acts onto, therefore, by the bounded inverse theorem one concludes that $(-\Delta_\OO-\l I)^{-1}$ is bounded.{ \footnote{ The argument is classical and goes as follows: Since  $-\Delta_\OO - \l I$ acts onto, one can immediately show that the kernel of  $-\Delta_\OO - \l I$ is trivial utilising its self-adjointness. Thus $-\Delta_\OO - \l I$ is a bijection from its domain onto $L^2(\mathcal O)$. The domain of $-\Delta_\OO - \l I$ endowed with the graph norm is a Banach space. Denote it by $\mathcal B$. The operator $-\Delta_\OO - \l I: \mathcal B \to L^2(\mathcal O)$ is bounded. Applying the the bounded inverse theorem one concludes that $(-\Delta_\OO-\l I)^{-1}: L^2(\mathcal O) \to \mathcal B$ is bounded, hence $(-\Delta_\OO-\l I)^{-1}: L^2(\mathcal O) \to L^2(\mathcal O)$ is bounded.}}
\end{proof}
Now we are in a position to define an analogue of the  Zhikov $\b$-function. For $\lambda \notin \Sp(-\Delta_\OO)$ we set
\begin{equation}\label{d:beta1}
	\beta(\lambda):=\lambda+\lambda^2 \int_{\Omega} (-\Delta_\OO-\lambda I)^{-1} {\mathbf 1}_{\mathcal{O}}dP(\omega)=\lambda+\lambda^2 \langle b\rangle,
\end{equation}	
where $b = b(\o,\l)$ is the solution to the problem
\begin{equation}\label{66}
-\Delta_\OO b  =  \lambda b + {\mathbf 1}_{\mathcal{O}}.
\end{equation}
\begin{remark}The function $\beta(\l)$ in \eqref{d:beta1} is an analogue of Zhikov's $\beta$-function introduced in \cite{Zhikov2000}, \cite{zhikov2005} for the periodic setting. In fact, if the probability space is such that the set of inclusions is periodic, then the expression on the right-hand side of \eqref{d:beta1} coincides with the known formula for the periodic setting. As we have mentioned in the Introduction, $\beta(\l)$ (to be more precise, its second term) quantifies ``resonant, anti-resonant'' effect of the soft inclusions. In order to clarify this connection, we analyse the properties of the physical realisations of the function $b(\o,\l)$ in Section \ref{s:4.1}, see Remark \ref{r:3.6} for further discussion.
\end{remark}

The next theorem provides a characterisation of the spectrum of $\AA_\hom$. 
\begin{theorem}\label{th:3.3}$\left.\right.$
	
	\begin{enumerate}
		\item The spectrum of the homogenised operator is fully characterised by $\b(\l)$ and the spectra of the ``microscopic'' $-\Delta_\OO$ and the ``macroscopic'' $-\nabla \cdot A_1^{\rm hom} \nabla$ operators as follows: 
		\begin{equation}
			\Sp(\AA^\hom ) = \Sp(-\Delta_\OO) \cup \overline{\{\lambda: \beta(\lambda) \in \Sp(-\nabla \cdot A_1^{\rm hom} \nabla)  \}}.
			\label{Ahom_char}
		\end{equation} 
	\item	 If $S$ is bounded, then for \ae $\omega \in \Omega$ we have the spectral convergence $\Sp(\AA^\e) \to\Sp(\AA^\hom )$ in the sense of Hausdorff.
	\end{enumerate}

\end{theorem}		
\begin{proof}
	Suppose that   $\lambda \notin \Sp(-\Delta_\OO)$. In order to solve \eqref{mik1r}--\eqref{mik2r}  we first solve \eqref{mik2r} for $u_1$ in terms of an arbitrary fixed $u_0 \in W_0^{1,2}(S)$:
	$$ u_1(x, \cdot)=\lambda u_0(x)(-\Delta_\OO-\lambda I)^{-1}{\mathbf 1}_{\mathcal{O}}+ (-\Delta_\OO-\lambda I)^{-1}f(x, \cdot). $$
	We then substitute the obtained expression into \eqref{mik1r},
	\begin{equation} \label{ivan100}
	-\nabla \cdot A_1^{\textrm{hom}}\nabla u_0 -\beta(\lambda) u_0=\bigl\langle \l (-\Delta_\OO-\lambda I)^{-1} f(x, \cdot) +  f(x, \cdot)\bigr\rangle.
	\end{equation}	
	 Taking $f(x,\omega)=g(x)\psi(\omega)$ with  arbitrary $g \in L^2(S)$ and $\psi\in \dom(-\Delta_\OO)$, we see that the solvability of \eqref{ivan100} is equivalent to $\beta(\lambda) \notin \Sp(-\nabla \cdot A_1^{\textrm{hom}} \nabla),$ and (\ref{Ahom_char}) follows.
	
	The second part of the statement is proved in \cite{ChChV}: the resolvent converges implies that the spectrum of the limit operator is contained in the limit of $\Sp(\AA^\e)$ as $\varepsilon\to0,$ while strong two-scale convergence of eigenfunctions of $\e$ problem implies that any limiting point of $\Sp(\AA^\e)$ is contained in the spectrum of the limiting operator.  
\end{proof}

\begin{remark}
In the case $S=\R^d$ one has
	 \begin{equation*}
\Sp(\AA^\hom ) = \Sp(-\Delta_\OO) \cup {\{\lambda: \beta(\lambda) \geq 0  \}}.
\end{equation*} 
\end{remark}

\begin{remark}
	While in the periodic case $\beta(\l)$ blows up to $\pm \infty$ as $\l$ approaches $\Sp(-\Delta_\OO) $, in the general stochastic setting this is not the case. It is not difficult to construct an example where $\beta(\l)$ does not blow up near $\Sp(-\Delta_\OO) $: for instance, one can consider the example of randomly scaled inclusions presented in Section \ref{sss:randomscaled}. In short,  consider a model for which the inclusions are randomly scaled copies of a single set, where the random scaling parameter $r$ takes values in some finite interval  $[r_1,r_2]$. Then choosing the corresponding probability density  so that it is positive on  $(r_1,r_2)$ and converges to zero sufficiently quickly near the ends of the interval, one gets that $\beta(\l)$ converges to finite values as  $\l$ approaches $\Sp(-\Delta_\OO)$, cf. also  \eqref{betascaling} below.
\end{remark}

\subsection{Properties of  Zhikov's $\beta$-function} \label{s:4.1}

	We begin by showing that the solution to \eqref{66}, defined on the probability space $(\Omega, \mathcal{F},P)$, can be reconstructed from its physical realisations. 
	
	Let $\omega \in \mathcal{O}$. Then by the definition of $\OO_\o$ there exist an inclusion containing the origin, $\OO_\o^{k_0}\ni 0$. We  reserve the notation $k_0$ for the index of such inclusion and denote
$P_\o:=\OO_\o^{k_0} - D_\o^{k_0} - d_{1/4}$.
	For the Dirichlet Laplacian operator $-\Delta_{P_{\omega}}$ on $P_\o$ we denote  by $\Lambda_s = \Lambda_s(\o), s\in\N,$ its eigenvalues arranged in the increasing order, $\Lambda_s <\Lambda_{s+1}, s\in \N$, and by $\{\Psi_s^p\}_{s\in\N,p=1,\dots,N_s}$ the corresponding system of the orthonormalised eigenfunctions, $\Psi_s^p = \Psi_s^p(\cdot,\o)$, where $N_s$ is the multiplicity of $\Lambda_s$.  See Lemmata \ref{l:B.9} and \ref{l:b.11} in Appendix  \ref{s:auxiliary} for the measurability properties of the eigenvalues and their eigenfunctions.
	
		For $\omega \in \mathcal{O}$ and $\lambda \notin \Sp (-\Delta_\OO)$,  we define $\tilde b(\omega,\lambda; \cdot) \in W^{1,2}_0(P_\o)$ as the solution to the problem
	\begin{equation}\label{18a}
	(-\Delta_{P_{\omega}} - \l) \tilde b = {\mathbf 1}_{P_\o}.
	\end{equation}
We assume that $\tilde b$ is extended by zero outside $P_\o$. Note that by applying the spectral decomposition, we can write it in terms of the eigenfunctions of $-\Delta_{P_{\omega}}$, as follows:
	\begin{equation}\label{26}
	 \tilde b(\omega,\lambda; x)= \sum_{s=1}^{\infty }\sum_{p=1}^{N_s(\omega)} \frac{\left(\int_{\R^d} \Psi_s^p(\cdot,\o)\right) \Psi_s^p(x,\o) }{\Lambda_s-\lambda}, \quad x\in \R^d.
	\end{equation}
		The results of Appendix  \ref{s:auxiliary} imply that for a fixed $\l$ the mapping $\omega \mapsto \tilde b(\omega, \lambda; \cdot)$,  taking values in $L^2(\R^d)$, is measurable (on the target space $L^2(\R^d)$ we take Borel $\sigma$-algebra).

For any locally integrable function $f\in L^1_{\rm loc}(\R^d)$ we denote by $f_\reg$ its {\it regularisation}, defined at $x\in\R^d$ by
\begin{equation*}
f_\reg(x):= \begin{cases} \lim_{r \to 0} {|B_r(x)|^{-1}} \int\limits_{B_r(x)} f(y)dy,  & \textrm{ if the limit exists,}\\ 0, & \textrm{ otherwise.}     \end{cases}.  
\end{equation*} 
Note that the Lebesgue differentiation theorem states that $f_\reg = f$ \ae

 Consider the  function
	\begin{equation}
	\label{limit_cases}
	 \breve{b}(\omega,\lambda):=\begin{cases}  \tilde b_\reg(\omega, \lambda;D(\o)), &  \o\in \OO,\\ 0, & \textrm{ otherwise.}     \end{cases}. 
	\end{equation}
We make the following observations. First,  using the regularised representative of $\tilde b$ we ensure that the right-hand side of \eqref{limit_cases} is measurable. Second, for fixed $\o$ and $k$ one has  $\tilde b_\reg(T_x \omega, \lambda;\cdot) = \tilde b_\reg(T_y \omega, \lambda;\cdot) $ for all $x,y\in \OO_\o^k$; therefore, fixing  $y\in \OO_\o^k$, we have  
\begin{equation}\label{27}
\breve b(T_x \omega,\l) = \tilde b_\reg(T_x \omega, \lambda;D(T_x \o)) = \tilde b_\reg(T_y \omega, \lambda;D(T_y \o) + x-y) \quad \forall x\in \OO_\o^k.
\end{equation}
Third, the mapping $x \mapsto \breve b(T_x \omega,\l)$ is an element of $ W^{1,2}_\loc(\R^d)$ for a.e. $\o\in \O$.  Finally,   for \ae $\omega \in \OO$, $x=D(\o)$ is a Lebesgue point of $\tilde b_\reg(\omega, \lambda;\cdot)$, i.e. 
\begin{equation}\label{24}
	\tilde b_\reg(\omega, \lambda;D(\o)) = \lim_{r \to 0} |B_r(D(\o))|^{-1} \int_{B_r(D(\o))} \tilde b_\reg(\omega, \lambda;y)dy.
\end{equation}
 Indeed, denote by $N$ the subset of $\OO$ consisting of $\omega$ for which \eqref{24} does not hold. By construction, we have 
$$ 
\int_{B_R(0)}  {\mathbf 1}_{\{T_x \omega \in N\}}dx=0\quad {\ae}\ \omega \in \Omega,
$$
for a positive $R$. Applying  Fubini Theorem yields
$$  
0 = \int_\Omega \int_{B_R(0)}  {\mathbf 1}_{\{T_x \omega \in N\}}dxdP=\int_{B_R(0)} \int_\Omega {\mathbf 1}_{\{T_x \omega \in N\}} dP dx=|B_R| |N|,$$
and the claim follows.

Throughout the paper, we will use the notation
\begin{equation}\label{d_l}
	d_\l : = \dist(\l, \Sp(-\Delta_\OO)).
\end{equation} 
	\begin{proposition}\label{sasha100} 
		For $\lambda \notin \Sp(-\Delta_\OO)$ the function $\breve b(\cdot, \lambda)$ defined by \eqref{limit_cases} is the solution to \eqref{66}, i.e. $\breve b$ coincides with $b$.  
	\end{proposition}
	\begin{proof}  
		First we show that $\breve b(\cdot, \lambda) \in L^2(\mathcal{O})$. Since $\tilde b(\omega, \lambda,\cdot)$ is the solution to \eqref{18a}, taking into account \eqref{20a} and  by the assumption $|\mathcal{O}^{k_0}_{\omega}|\leq 2^{-d}< 1$, we have
		\begin{equation} \label{pauza1} 
		\bigl\| \tilde b(\omega, \lambda,\cdot)\bigr\|_{L^2(\R^d)}^2 < d_\l^{-2}, \quad \o\in\OO.
		\end{equation}
		 It follows from Assumption \ref{kirill100} that the ball $B_{\rho}$ is contained in $\mathcal{B}_\o^{k_0}$; in particular,  it has no intersections with   inclusions $\OO_\o^k, k\neq k_0$. Invoking \eqref{27} and \eqref{pauza1}, it follows that
		$$ 
		\int_{B_{\rho}} \breve b^2 (T_x \omega, \lambda) dx <  d_\l^{-2}. 
		$$
		Integrating over $\Omega$ and using the  Fubini theorem, we obtain
		$$
		|B_{\rho}|\bigl\| \breve b(\cdot, \lambda)\bigr\|^2_{L^2({\mathcal O})} < d_\l^{-2}.
		$$

		Next we argue that $\breve b(\cdot, \lambda) \in W^{1,2}(\mathcal{O})$. To show this, we define the function $g : \Omega \to \R^d$ as follows:  
		\[ 
		g(\o,\l): = \begin{cases}   (\nabla \tilde b)_\reg(\o,\l;D(\o)), &  \o\in \OO,\\ 0, & \o\notin\OO.     \end{cases}
		\]
		For a fixed $\o\in \OO$ we clearly have $(\nabla \tilde b)_\reg(\o,\l;x) = \nabla \tilde b_\reg(\o,\l;x) = \nabla \tilde b(\o,\l;x)$ for \ae $x\in \R^d$. By observing that 
		$$\bigl\| \nabla \tilde b(\omega, \lambda,\cdot)\bigr\|_{L^2(\R^d)}^2 \leq \l \bigl\| \tilde b(\omega, \lambda,\cdot)\bigr\|_{L^2(\R^d)}^2 + \bigl\| \tilde b(\omega, \lambda,\cdot)\bigr\|_{L^2(\R^d)}$$
	(which follows from \eqref{18a}),	it is not difficult to show that $g \in L^2(\Omega; \R^d)$, similarly to the claim that $\breve b(\cdot,\l)\in L^2(\OO)$. It remains to show that $\nabla_{\omega} \breve b=g$, which is equivalent to
		\begin{equation} \label{zvonko1} 
		\int_{\Omega}  g_j\, v=-\int_{\Omega}\breve  b\, \DD_j v \quad \forall v \in W^{1,2} (\Omega), \, j=1, \dots,d. 
		\end{equation} 
		
		Denote by $\mathring B_R(\omega)$ the union of the inclusions contained in the ball $B_R$,
		$$ \mathring B_R(\omega):=\bigcup_{ \mathcal{O}^k_{\omega} \subset B_R} \mathcal{O}_{\omega}^k. $$  
		We need the following simple assertion.
		\begin{lemma}\label{l:3.6}
			Suppose $f\in L^1(\Omega)$, let $U\subset \R^d$ be an open bounded star-shaped set, and for a fixed $\delta>0$ let  $(V_R)$  be a  sequence of measurable sets satisfying  $(R-\delta)U \subset V_R \subset RU$. Then almost surely  one has 
			\[
			 \int_\O f = \lim_{R \to \infty} \frac{1}{|RU|} \int_{V_R} f(T_x \o) dx.
			\]
		\end{lemma}
	\begin{proof}
	The proof is straightforward and follows from the ergodic theorem and the observation that
	\begin{equation*}
	\lim_{R \to \infty} \frac{1}{|RU|}\int_{RU \setminus (R-\delta)U} \left|f(T_x \o) \right| dx=0.
	\end{equation*}
	\end{proof}
		
		Applying the above lemma, we get
		\begin{equation}	
		\begin{aligned} 
		\label{zvonko10} \int_{\Omega}  g_j v = \lim_{R \to \infty} \frac{1}{|B_R|} \int_{\mathring B_R(\omega)}  g_j(T_x \omega) v(T_x \omega) dx, \, j=1, \dots,d,
\\ 	
	\int_{\Omega} \breve b \, \DD_j v= \lim_{R \to \infty} \frac{1}{|B_R|} \int_{\mathring B_R(\omega)}  \breve b (T_x \omega,\lambda) (\DD_j v)(T_x \omega)  dx,\, j=1, \dots,d.  
		\end{aligned}
	\end{equation} 
		
		Similarly to \eqref{27}, for a fixed $\mathcal{O}^k_{\omega} \subset B_R$  and a fixed $y\in \mathcal{O}^k_{\omega}$ we have 
			\[
		g(T_x \omega) = (\nabla \tilde b)_\reg(T_x \o,\l; D(T_x \o)) = \nabla_x \tilde b_\reg(T_y \o,\l; D(T_y \o) + x - y) \mbox{  \ae } x\in \mathcal{O}^k_{\omega}.
		\]
		Integrating by parts in each of the inclusions, we obtain
		$$ \int_{\mathring B_R(\omega)} g_j(T_x \omega) v(T_x \omega) dx=-\int_{\mathring B_R(\omega)} \breve b(T_x \omega, \lambda)   \DD_j v (T_x \omega),\, j=1, \dots,d, $$
		which in conjunction with   \eqref{zvonko10} implies \eqref{zvonko1}. 
		
		Writing \eqref{18a} in the weak form,
		\begin{equation*}
		\int_{P_{\omega}} \nabla \tilde b(\omega, \lambda; x) \cdot \nabla v(T_x \o)dx-\lambda \int_{P_{\omega} } \tilde b(\omega, \lambda; x) v(T_x \o) dx=\int_{P_{\omega} } v(T_x \o) dx \quad \forall v \in W^{1,2}_0(\OO),
		\end{equation*}
		and arguing in a similar way to the above via Lemma \ref{l:3.6}, one  easily arrives at
			\begin{equation*} 
			\int_{\mathcal{O}}\nabla_{\omega} \breve b(\cdot, \lambda) \cdot \nabla_{\omega} v-\lambda \int_{\mathcal{O}} \breve b(\cdot, \lambda) v= \int_{\mathcal{O}} v \qquad \forall v \in W_0^{1,2} (\mathcal{O}),    
		\end{equation*}
	which concludes the proof.
	\end{proof} 
	
\begin{remark}\label{r:3.6}The construction at the start of the section, in particular the formula \eqref{27}, and Proposition \ref{sasha100} establish an explicit  connection between the function $b(\o,\l)$ and its realisation in the physical space. Namely, for a typical $\o$ the restriction of $b(T_x\o,\l)$ to an individual inclusion $\OO^k_\o$ is the solution to the problem
	\[
	( - \Delta - \l)b(T_x\o,\l) = 1,\quad x\in{\OO^k_\o},
	\]
	subject to the Dirichlet boundary condition on $\partial\OO^k_\o$. Writing the spectral decomposition of $b(T_x\o,\l)|_{\OO^k_\o}$,  cf. \eqref{26} and  integrating it over $\OO^k_\o$, one obtains an expression similar to \eqref{34} below, which may be interpreted, depending on whether it is positive or negative, as the resonant  or anti-resonant contribution of the individual inclusion.
	
	Applying Lemma \ref{l:3.6} to the definition \eqref{d:beta1}, we obtain the  representation
	\begin{equation} \label{sasha103} 
		\b(\l)=	\lim_{R\to \infty}\left(\l+\l^2\frac{1}{|B_R|}  \int_{\mathring B_R(\omega)} b (T_x \omega,\lambda)dx\right),
	\end{equation} 
where the second term may be interpreted as the averaged resonant/anti-resonant contribution of the individual inclusions.
	We  emphasise the importance of \eqref{sasha103}: in order to compute $\b(\l)$ one does not need to know $b(\cdot,\l)$ on $\O$, but only its realisation $b (T_x \omega,\lambda)$ for a single typical $\o$.
\end{remark}

	\begin{proposition}
	The function $\b$ is differentiable and its derivative is uniformly positive. More precisely, $\b'(\l)~\geq~1~-~P(\OO)$.
	\end{proposition}
	\begin{proof} 
		For $\omega \in \mathcal{O}$ we have (cf. \eqref{26}) 
		\begin{equation} \label{34} 
		\int_{\mathbf{R^d}} \tilde b(\omega,\lambda, \cdot)=\sum_{s=1}^{\infty }\sum_{p=1}^{N_s(\omega)}\frac{(\int_{\R^d} \Psi_s^p)^2}{\Lambda_s-\lambda}.  
		\end{equation} 
Observing that 
	\begin{equation} \label{29} 
 \sum_{s=1}^{\infty }\sum_{p=1}^{N_s(\omega)}{\left(\int_{\R^d} \Psi_s^p\right)^2} = \vert P_\o\vert,
	\end{equation} 
a direct calculation yields 
\begin{align}\label{sasha102} 
	\frac{\partial}{\partial \lambda} \left(\l^2 \int_{\mathbf{R^d}} \tilde b(\omega,\lambda, \cdot)\right) = -\sum_{s=1}^{\infty }\sum_{p=1}^{N_s(\omega)}{\left(\int_{\R^d} \Psi_s^p\right)^2}+\sum_{s=1}^{\infty }\sum_{p=1}^{N_s(\omega)}\frac{\Lambda_s^2(\int_{\R^d} \Psi_s^p)^2}{(\Lambda_s-\lambda)^2} \ge -\vert P_\o\vert,
\end{align}
Furthermore, by  construction, one has
	\begin{equation*} 
		  \int_{\OO_\o^k} b (T_x \omega,\lambda)dx=  \int_{\mathbf{R^d}} \tilde b(T_y\omega,\lambda, \cdot) \quad \forall  y\in \OO_\o^k.
		\end{equation*}
		It is not difficult to see  that on any closed interval $[\l_1,\l_2]\subset 	{\rm Dom} (\b) = \R\setminus \Sp(-\Delta_\OO)$ the right-hand side of \eqref{sasha103} and its derivative  
		\begin{equation*}
\frac{\partial}{\partial \lambda}\left(\l+\l^2\frac{1}{|B_R|}  \int_{\mathring B_R(\omega)} b (T_x \omega,\lambda)dx\right)
		\end{equation*} 
		converge uniformly  as $R\to\infty$ (cf. \eqref{29} and \eqref{sasha102} and Lemma \ref{l:3.6}). It follows that $\b(\l)$ is differentiable. Moreover, from \eqref{sasha103} and \eqref{sasha102} we have 
		\begin{equation*}
		 \b'(\l) \geq 1 - \lim_{R\to \infty} \frac{1}{|B_R|}  \int_{\mathring B_R(\omega)}  {\mathbf 1}_{\OO_\o} = 1-P(\OO),
		\end{equation*}
	as claimed
	\end{proof}

\begin{remark}
	It was shown in \cite{ChChV} that 
	$$\Sp(-\Delta_\OO)= \overline{ \bigcup_{x \in \R^d} \bigcup_{s \in \mathbf{N}} \{ \Lambda_s (T_x \omega)\}}$$ 
	almost surely. In particular,  the set on the right-hand side is deterministic almost surely.  For the sake of completeness, we provide a streamlined version of the proof of this statement in Proposition \ref{pppppp1}. 
\end{remark}

\subsection{Resolution of identity for $-\Delta_\OO$}\label{s:4.2}

Similarly to the formula \eqref{sasha103} for $\beta(\l)$, it is often useful (in examples and  practical applications) to be able to characterise or reconstruct abstract probabilistic objects associated with the homogenised operator $\AA^\hom$, such as $-\Delta_\OO$ for example,  from their realisations in the physical space (for a single typical $\o\in \O$).
For a fixed $\o$ we can think of the family of Dirichlet Laplace operators defined on individual inclusions as the realisation of $-\Delta_\OO$ in the physical space. More precisely, we next characterise the resolution of identity for $-\Delta_\OO$ by using the spectral projections of these Dirichlet Laplace operators on the inclusions and pulling them back to the probability space.
 
	For  $0 \leq t_1 \leq t_2$ we define the mappings $\tilde E_{(t_1,t_2]}$ and $E_{(t_1,t_2]}$ as  follows. For all $\varphi \in L^2(\OO)$ let
	\begin{equation} \label{sasha110} 
(\tilde E_{(t_1,t_2]} \varphi)(\omega,x) := \sum_{s=1}^{\infty }\sum_{p=1}^{N_s(\omega)} {\mathbf 1}_{\{t_1 < \Lambda_s (\omega) \leq t_2\}}\left(\int_{\R^d} \Psi_s^p(\cdot, \omega)  \varphi(T_{\cdot-D} \omega)\right) \Psi_s^p(x,\omega), \quad \o\in\OO,
	\end{equation}
	and 
	\begin{equation*}
	(E_{(t_1,t_2]} \varphi)(\omega)= (\tilde E_{(t_1,t_2]}\varphi)_\reg(\omega, D(\o)). 
	\end{equation*} 
	The mapping $E_{[t_1,t_2]}$ is defined analogously. In what follows we extend functions from $L^2(\OO)$ by zero in $\O\setminus\OO$ without mentioning it explicitly. 
	\begin{proposition} \label{josip3} 
		$E_{[0,t]}$ is the resolution of identity for the operator $-\Delta_\OO$. 
	\end{proposition} 
	\begin{proof} 
		Arguing  as in the proof of Proposition \ref{sasha100},  we conclude that there exists $C>0$, independent of $t_1$ and $t_2$, such that 
		$$ \|E_{(t_1,t_2]} \varphi\|_{L^2(\OO)} \leq C\| \varphi\|_{L^2(\OO)}\quad 	\forall \varphi \in L^2(\OO),$$
		and that for \ae $\omega \in \Omega$ one has
		\begin{equation} \label{sasha1000}
		(E_{(t_1,t_2]} \varphi_1, \varphi_2)=\lim_{R \to \infty} \frac{1}{|B_R|} \int_{\mathring B_R(\omega)}  (E_{(t_1,t_2]} \varphi_1)(T_x\omega) \varphi_2 (T_x \omega)  dx \quad \forall \varphi_1, \varphi_2 \in L^2(\OO).
		\end{equation} 
		In order to prove that $E_{[0,t]}$ is the resolution of identity, we next verify the following properties:
		\begin{enumerate}[(a)]
			\item $E_{[0,t]}$ is an orthogonal projection;
			\item $E_{[0,t_1]} \leq  E_{[0,t_2]}$ if $t_1 \leq t_2$; 
			\item $E_{[0,t]}$ is right continuous in the strong topology;
			\item $E_{[0,t]} \to 0$ if $t\to 0$  and $E_{[0,t]} \to I$ if $t \to +\infty $ in the strong topology.
		\end{enumerate}
		Using \eqref{sasha1000} and a representation of $(E_{(t_1,t_2]} \varphi_1)(T_x\omega)$ analogous to \eqref{27}, it is easy to see that
		\begin{align*} 
			(E_{(t_1,t_2]} \varphi_1, \varphi_2)&=( \varphi_1, E_{(t_1,t_2]}\varphi_2), \\[0.3em]
			(E_{(t_1,t_2]} E_{(t_1,t_2]} \varphi_1, \varphi_2) &=(E_{(t_1,t_2]} \varphi_1,\, E_{(t_1,t_2]} \varphi_2)=(E_{(t_1,t_2]} \varphi_1, \varphi_2),
		\end{align*}
		which implies  (a).  In the same way one can see that, for $0 \leq t_1 \leq t_2 \leq t_3 \leq t_4$,
		$$ E_{(t_1,t_2]} E_{(t_3,t_4]}=0, \textrm{ and } E_{(t_1,t_3]}=E_{(t_1,t_2]}+E_{(t_2,t_3]}, $$
		so (b) holds. 
		
		Finally, we prove (c). For a fixed \ae $\omega \in \Omega$ and $\varepsilon>0$ small enough we have 
		$$ \int_{B_{\rho}} \left|E_{[0,t+\varepsilon]} \varphi(T_x \omega)-E_{[0,t]} \varphi(T_x \omega) \right|^2 dx=0$$
		(recall that $B_{\rho}$ has common points with at most one inclusion).		Also, for all $\varepsilon>0$ one has
		$$
		\int_{B_\rho} \left|E_{[0,t+\varepsilon]} \varphi(T_x \omega)-E_{[0,t]} \varphi(T_x \omega) \right|^2 dx \leq \int_{\square} |\varphi^2(T_x \omega)| dx $$
		almost surely.
		Integrating the right-hand side of the last inequality  over $\Omega$ yields
		$$ \int_{\Omega}  \int_{\square} |\varphi^2(T_x \omega)| dxd\o= \|\varphi\|_{L^2(\Omega)}^2.$$ 
		Thus, by the Lebesgue theorem on dominated convergence, we conclude that 
		$$ |B_\rho| \|E_{[0,t+\varepsilon]} \varphi-E_{[0,t]} \varphi\|^2_{L^2(\Omega)} = \int_{\Omega}\int_{B_\rho} \left|E_{[0,t+\varepsilon]} \varphi(T_x \omega)-E_{[0,t]} \varphi(T_x \omega) \right|^2 dxd\o \to 0, $$
		and (c) follows. Similarly we prove (d), cf. also Remark \ref{r:2.10}.
		
		Now, we prove that $E_{[0,t]}$ is the resolution of identity associated with the operator $-\Delta_\OO$, i.e.,
		$$-\Delta_\OO=T: = \int_0^{+\infty} s\,dE_s,\qquad{\rm Dom} (T)=\biggl\{\varphi \in L^2(\OO): \int_0^{+\infty}s^2 dE_s \varphi< \infty\biggr\}.$$
		Consider the space 
		\begin{equation*}
\mathcal Y:=\bigcup_{t>0} \textrm{ Im } E_{[0,t]}.		
		\end{equation*}
	By the property (d), $\Y$ is dense in $L^2(\OO)$. Let $\varphi \in \Y$, then $\varphi = E_{[0,t]} \varphi$ for some $t$ and 
	\begin{equation}\label{40}
	T\varphi = \int_0^t s dE_s \varphi = \lim_{n \to \infty} \sum_{m=1}^n m\frac{ t}{n}E_{\left((m-1)\frac{t}{n},m\frac{t}{n}\right]}\varphi.
	\end{equation}
	Arguing similarly to the proof of Proposition \ref{sasha100} (i.e., using the definition of $E_{[0,t]} \varphi$, passing to realisations, and using the ergodic theorem), one can show that  $\Y \subset {\rm Dom}(-\Delta_\OO)$, and for $\varphi\in \Y$  one has
	\begin{equation}\label{41}
	-\Delta_\OO \varphi(\o) = 	\sum_{s=1}^{\infty }\Lambda_s (\omega) {\mathbf 1}_{\{0 \leq  \Lambda_s (\omega) \leq t\}} \left(\tilde E_{\left[\Lambda_s(\o), \Lambda_s(\o) \right]} \varphi\right)_\reg(\omega, D(\o)).
	\end{equation}	
	Comparing the right-hand sides of \eqref{40} and \eqref{41} on realisations, we see that on each inclusion $\OO_\o^k$ they are linear combinations of the projections of $\varphi(T_x \omega)$ onto the eigenspaces of $-\Delta_{\OO_\o^k}$ and differ only by the coefficients $m\frac{t}{n}$ and $\Lambda_s (T_x \omega)$, respectively. In follows that
\begin{equation*}
\begin{aligned}
 \frac{1}{|B_R|}\int_{\mathring B_R(\omega)}\Bigg| \sum_{m=1}^n m\frac{t}{n} \left(E_{\left((m-1)\frac{t}{n},m\frac{t}{n}\right]}\varphi\right)(T_x \omega) \qquad\qquad & 
\\
-\sum_{s=1}^{\infty }\Lambda_s (T_x \omega) {\mathbf 1}_{\{0 \leq  \Lambda_s (T_x \omega) \leq t\}} \left(\tilde E_{\left[\Lambda_s(T_x \omega), \Lambda_s(T_x \omega) \right]} \varphi\right)_\reg & (T_x\omega, D(T_x \o))  \Bigg|^2 dx 
\\
 &\leq  \frac{t^2}{n^2} \frac{1}{|B_R|}\int_{\mathring B_R(\omega)} \left|  \left(E_{[0,t]}\varphi\right)(T_x \omega) \right|^2 dx.
\end{aligned}
\end{equation*}
Passing to the limit as $R\to \infty$ via Lemma \ref{l:3.6}, we get
 \begin{equation*}
 	\begin{aligned}
\left\| \sum_{m=1}^n m\frac{t}{n} E_{\left((m-1)\frac{t}{n},m\frac{t}{n}\right]}\varphi 		-\sum_{s=1}^{\infty }\Lambda_s  {\mathbf 1}_{\{0 \leq  \Lambda_s  \leq t\}}  E_{\left[\Lambda_s,\Lambda_s \right]} \varphi\right\|_{L^2(\OO)}
 	\leq  \frac{t}{n} \left\|  E_{[0,t]}\varphi\right\|_{L^2(\OO)}.
 	\end{aligned}
 \end{equation*}
Now, passing to the limit as $n\to \infty$,	we infer that $T\varphi = -  \Delta_\OO \varphi$, i.e.   the (symmetric) restrictions of the operators $T$ and $-  \Delta_\OO$ to $\Y$ coincide.

To conclude the proof it remains to note that either of the restrictions is essentially self-adjoint. To this end, it suffices to note that the image of $T|_\Y \pm i$ is dense in $L^2(\OO)$,  since for all  $\varphi \in L^2(\OO)$ one has
		$$ E_{[0,t]} \varphi= \lim_{n \to \infty} \sum_{m=1}^n  \frac{1}{m\frac{t}{n}\pm i}\int_{(m-1)\frac{t}{n}}^{m \frac{t}{n}} (s \pm i) dE_s \varphi, $$  
where the limit is understood in the $L^2(\OO)$ sense.
	\end{proof} 

	\begin{remark}
	Similarly to Remark \ref{r:3.6}, one can recover the resolution of identity from realisations for a single (typical) $\o$ via the formulae \eqref{sasha110}--\eqref{sasha1000}.
\end{remark}
	
	\subsection{Point spectrum of the homogenised operator} 
In this section we first provide a general characterisation of the point spectrum of $\AA^\hom$, arguing that it is determined completely by the point spectrum of $-\Delta_\OO$ (in complete analogy with the periodic setting \cite{zhikov2005}). Combining it with the above construction of the resolution of identity, we then illustrate  it by  two classes of examples. 

	\begin{proposition}\label{p4.13}
		The point spectrum of $\AA^\hom$ (in the case $S=\R^d$) consists of those eigenvalues  of $-\Delta_\OO$ whose eigenfunctions have zero mean:
		\begin{equation*}
		\Sp_p (\AA^\hom) = \{\l\in \Sp_p(-\Delta_\OO): \exists \psi\in W^{1,2}_0(\OO) \mbox{ such that } -\Delta_\OO \psi = \l\psi, \, \langle \psi \rangle = 0 \}.
		\end{equation*}
	\end{proposition}
	\begin{proof} 
Consider the following eigenvalue problem for $\AA^\hom$ (cf. \eqref{mik1r}--\eqref{mik2r}):
		\begin{equation}\label{44}
			-\nabla A_1^{\textrm{hom} } \nabla u_0=\lambda (u_0+\langle u_1\rangle ), \qquad
			-\Delta_\OO u_1 (x, \cdot)=\lambda (u_0(x)+u_1(x,\cdot)), 	
		\end{equation}		
		where $u_0+ u_1 \in V$. Suppose that $\l\in \Sp_p(-\Delta_\OO)$ and the corresponding eigenfunction $\psi$ has zero mean. Then the pair $u_0=0$, $u_1=v \psi$, where $v \in L^2(\R^d)$, clearly satisfies \eqref{44}. Hence $\l \in \Sp_p(\AA^\hom)$.

		Now suppose  $\l \in \Sp_p(\AA^\hom)$, i.e. \eqref{44} holds. We claim that $u_0=0$. Indeed, assuming the contrary, we infer from the second equation in \eqref{44} that the problem
		\begin{equation}\label{45}
		-\Delta_\OO v =\lambda v + {\mathbf 1}_\OO
		\end{equation}
		is solvable. Note that all its solutions have the form $v = v_0 + \psi$, where $v_0$ is a fixed solution and $\psi \in E_{[\l,\l]}(L^2(\OO))$. Moreover, all elements of $E_{[\l,\l]}(L^2(\OO))$ have zero mean by the solvability condition for \eqref{45}. It is easy to see that the solution of the second equation in \eqref{44} has the form $u_1 = \l u_0 v_0 + \varphi$, where $\varphi \in L^2(\R^d;E_{[\l,\l]}(L^2(\OO)))$. Substituting this into the first equation, we arrive at 
		\begin{equation*}
		-\nabla A_1^{\textrm{hom} } \nabla u_0=(\lambda + \l^2 \langle v_0 \rangle )  u_0.
		\end{equation*}	
	But this contradicts to the fact that the spectrum of $-\nabla A_1^{\textrm{hom} } \nabla$ in $\R^d$ has no eigenvalues, therefore $u_0=0$, as required. 
 Then necessarily $\langle u_1 \rangle=0$, and  hence $u_1(x,\cdot)$ is an eigenfunction of $-\Delta_\OO$ for \ae $x \in \R^d$.
	\end{proof} 	

	\subsubsection{Example:  a finite number of inclusion shapes} 
Assume that for \ae $\o$ the inclusions $\OO_\o^k$ are (translated and rotated) copies of a finite number of sets (shapes). By \eqref{20a}, the spectrum of $-\Delta_\OO$ is  the union of the spectra of the Dirichlet Laplacians on  these sets. In particular, $\Sp(-\Delta_\OO)$ is discrete. It is not difficult to see (cf. the construction of the resolution of identity) that the eigenspaces of $-\Delta_\OO$ consist of those functions whose realisations are appropriate eigenfunctions on each inclusion. Thus, the point spectrum of $\AA^\hom$ consists  of the eigenvalues of each shape whose all  eigenfunctions have zero mean. 

	\subsubsection{Example: a continuum family of scaled copies of one shape}

	Assume that for \ae $\o$ the inclusions $\OO_\o^k$ are scaled (translated and rotated) copies of an open set $\mathring{\mathcal \OO}\subset \R^d$ (shape), where the scaling parameter $r$ takes values in some interval $[r_1,r_2]\subset (0,+\infty)$. We will write $\OO_\o^k \sim r\mathring{\mathcal \OO}$ if the inclusion $\OO_\o^k$ is a translation and rotation of $r\mathring{\mathcal \OO}$. Let $\{\nu_j\}_{j=1}^{\infty}$ be the sequence of eigenvalues for the Dirichlet Laplacian on $\mathring{\mathcal \OO}$. Then the spectrum of the Dirichlet Laplacian on  $\OO_\o^k \sim r \mathring{\mathcal \OO}$  is  $\{r^{-2} \nu_j\}_{j=1}^{\infty}$. Clearly, for any $t\in \R$ there are at most finitely many values of the scaling parameter satisfying $t=r^{-2} \nu_j$ for some $j$.
	
	Consider the situation when the probability measure does not concentrate for any $r\in[r_1,r_2]$. We can make this assumption precise using the ergodicity rather than referring to an abstract probability space. Namely,  assume that for \ae $\o$ and all $r\in[r_1,r_2]$ one has
	$$
	\lim_{R\to \infty} \frac{1}{|B_R|}\int_{B_R(\omega)} \sum_{k:\OO_\o^k \sim r \mathring{\mathcal \OO}}  {\mathbf 1}_{\OO_\o^k} = 0,
	$$  
equivalently, that 
	$$
\lim_{R\to \infty} \frac{1}{|B_R|}\int_{B_R(\omega)} \sum_{k:\OO_\o^k \sim r' \mathring{\mathcal \OO}, r'\in [r-\delta,r+\delta]}  {\mathbf 1}_{\OO_\o^k} \stackrel{\delta\to 0}{\longrightarrow} 0.
$$  
It is the easy to see, cf. \eqref{sasha1000}, that   the resolution of identity $E_{[0,t]}, t\geq 0$, has no jumps, i.e., $\Sp_p (-\Delta_\OO) = \emptyset$.


\section{Convergence of the spectra}\label{s:conv of spectrum}

As we have already discussed in the Introduction, the spectrum of $\AA^\e$ and its limit behaviour in the whole space setting is very different from that in the case of a bounded domain. Many ideas we used in \cite{ChChV} were an adaptation of well-known techniques from periodic high-contrast homogenisation. The analysis of the whole space setting, however, calls for methods that go beyond those used in the bounded domains. The rest of the paper is devoted wholly to the former, and in  what follows we will always assume that  $S=\R^d$. 

In the present section we  analyse the limit of the spectra of $\AA^\e$ as $\e\to 0$. In order to understand the basic difference between the bounded domain and the whole space settings, one may think of the following simple example: the set of inclusions $\OO_\o$ is generated by putting balls of equal size with probability $\nicefrac{1}{2}$ at the nodes of a periodic lattice. The spectrum of $\AA^\hom$ then has a band-gap structure very similar to the corresponding periodic case, with the left ends of the bands moved slightly to the left. Yet, $\Sp(\AA^\e) = \R^+_0$ for all $\e$. Indeed, due to the law of large numbers, there almost surely exists a sequence of balls $B_R(x_R)$, $R\in \N$, that is completely void of inclusions from $\e\OO_\o$. Then, for an arbitrary $\l\geq 0$, by taking the real part of a harmonic wave ${\mathcal Re}(e^{ik\cdot x})$ with $|k|^2=\l$, multiplying it by suitable cut-off functions and normalizing, one constructs an explicit, supported on $B_R(x_R)$, $R\in \N$, Weyl sequence for the operator $\AA^\e$, thus asserting that $\Sp(\AA^\e) = \R_+$. In more advanced (and more realistic as far as applications concerned) examples, such as the random parking model, see Section \ref{randomparking} below,  the spectrum of  $\Sp(\AA^\e)$ may exhibit gaps which persist in the limit. However, the general idea is the same: the additional part of the spectrum, not accounted for by $\Sp(\AA^\hom)$ in the limit, is present due to the arbitrarily large areas of space with non-typical distribution of  inclusions.  

In order to account for this part of the spectrum in the limit, one, somewhat surprisingly, may use a  close analogue of the $\beta$-function given by the formula \eqref{sasha103}, which we will denote by $\beta_\infty(\l,\o)$.  Intuitively, one can think of is as the supremum of ``local $\beta$-functions'', i.e.  local averages of the expression $\l + \l^2 b(T_{x/\e} \o,\l)$, whose  intervals of negativity get bigger when the local volume fraction of the inclusions is relatively large, and, conversely, shrink when  the local volume fraction of the inclusions is small (or completely disappear if the volume fraction tends to zero). In order to make this more precise, one needs to look at the behaviour of the distribution of inclusions in large randomly located sets. 

	\subsection{Main results}\label{s4.1}

For $\lambda \notin \Sp(-\Delta_{\mathcal{O}})$, define 
\begin{equation} \label{38} 
	\beta_\infty(\lambda,\omega):=\liminf_{M \to \infty} \sup_{x \in \mathbf{R}^d}  \ell(x,M,\lambda,\omega),
\end{equation}
where
\begin{equation} \label{42}
	\ell(x,M,\lambda,\omega):=\lambda+\lambda^2\frac{1}{M^d}\int_{\square_x^M} b(T_y \omega,\lambda) \,dy, 
\end{equation} 
and $\square_x^M: = x+ M\square$ is the cube of edge length $M$ centred at $x$.
We will use the shorthand notation $\square^M: =  M\square = \square^M_0$.

Note that $\ell(x,M,\lambda,\cdot)$ is measurable, which implies the measurability of $\beta_\infty(\lambda,\cdot)$. 
\begin{remark}\label{r:spec_av}
	One may interpret the term 
	$$\frac{\lambda^2 }{M^{d}}\int_{\square_x^M} b(T_y \omega,\lambda) \,dy$$
	 as the local averaged ``resonant'' (or ``anti-resonant'', if it is negative) contribution from the inclusions (which from the physical point of view play a role of micro-resonators) to the term $\ell(x,M,\lambda,\omega)$, see also Remark \ref{r:3.6}. The latter in turn may be interpreted as a {\it ``local spectral average''}, a precursor for the ``spectral parameter'' $\beta_\infty(\lambda,\omega)$.  
\end{remark}

By Proposition \ref{propivan1} below, the function $\beta_\infty$ is deterministic almost surely, that is $\beta_\infty(\lambda,\omega) = \beta_\infty(\lambda)$  for a.e. $\o$.
The  set 
$$\mathcal G:=  \Sp(-\Delta_\OO)\cup\{\l: \, \beta_\infty(\l) \geq 0\}$$
is an ``upper bound'' for the limit of the spectra of $\mathcal A^\e$, as stated next.

\begin{theorem}\label{th4.1} One has 
	$\lim \Sp(\AA^\e) \subset \mathcal{G}$
almost surely.
\end{theorem}
We provide the proof of this theorem in Section \ref{ss:4.3}.

	In order to prove that $\mathcal G$ is actually the limit of $\Sp(\mathcal A^\e)$, we make an additional assumption of finite-range correlations, which we describe next.
	For a compact set $K \subset \mathbf{R}^d$ we denote by $\mathcal{H}_K$ the set-valued mapping 
	$ \mathcal{H}_K: \, \omega \to \overline{\OO_\o \cap K}.$
	By $\mathcal{F}_{{\rm H},K}$ we denote the $\sigma$-algebra on the set of all compact subsets of $K$ generated by the Hausdorff metric $d_{\rm H}$.  
	The proof of the following lemma is completely analogous to that of Lemma \ref{measp} and we omit it.
	\begin{lemma} 
		The set-valued mapping $\mathcal{H}_K$ is measurable with respect to the $\sigma$-algebra $\mathcal{F}_{{\rm H},K}$.
	\end{lemma}

	For a compact set $K \subset \mathbf{R}^d$, define the $\sigma$-algebra $\mathcal{F}_K$ by 
	$$ \mathcal{F}_K:=\{\mathcal{H}_{K}^{-1}(F_{{\rm H},K}) :F_{{\rm H},K} \in \mathcal{F}_{{\rm H},K} \}. $$ 
	
	The following assumption is crucial for proving the second part of the spectral convergence. In simple words it says that the arrangements of inclusions in sets which are at least at distance $\kappa$ apart are independent.
	\begin{assumption}\label{asumivan1}  
		There exists $\kappa \in \mathbf{R}^{+}$ such that 
		for any two compact sets $K_1,K_2 \subset \mathbf{R}^d$ satisfying  $\dist(K_1,K_2)> \kappa$  the $\sigma$-algebras $\mathcal{F}_{K_1}$ and $\mathcal{F}_{K_2}$ are independent. 	
	\end{assumption} 
\begin{theorem}\label{th4.4}
	Under Assumption \ref{asumivan1} one has  
	$\lim \Sp(\AA^\e) \supset \mathcal{G}$
	almost surely. (Hence, in view of Theorem \ref{th4.1}, the two sets are equal.)
\end{theorem}
The proof of the theorem is given in Section \ref{ss:4.5}.

Assumption \ref{asumivan1} is only needed for the proof of Theorem \ref{th4.4}. It is not required for any other result in this paper. (We believe that one can also  relax it by allowing for sufficiently quickly (e.g., exponentially) decaying rather than finite-range correlation.) Finite range of dependence  guarantees the almost sure existence of arbitrarily large cubes with almost periodic arrangements of inclusions. This allows us to construct approximate eigenfunctions explicitly by starting from the generalised eigenfunctions of an operator with constant coefficients that  represents the macroscopic component of the homogenisation limit for the corresponding periodic operator. In the absence of such (almost) periodic structures, one needs to analyse ``arbitrary'' sequences of  operators $\hat \AA^\e$, each capturing a specific point of the limit spectrum of $\AA^\e$. Such sequences are obtained from $\AA^\e$ by shifting to the origin the areas of non-typical distribution of inclusions on which relevant quasimodes are supported (see the proof of Theorem \ref{th4.1} below). The limits of these sequences, understood in an appropriate sense, correspond to neither stochastic nor  periodic homogenisation, and therefore  in general have unpredictable spectral properties.

\begin{remark}
	It is well known that ergodicity implies that the spectrum of a self-adjoint operator with random coefficients is deterministic \cite[Theorem 1.2.5]{Stollmann2001}. In fact,  this can be proved by an argument  analogous to the first part of the proof of Proposition \ref{pppppp1}.	In particular,  $\Sp (\AA^\e(\o))$  is deterministic almost surely, and so  $\lim \Sp(\AA^\e)$ is also deterministic almost surely. Recalling the preceding discussion, the problem of characterising $\lim \Sp(\AA^\e)$ under the general ergodicity assumption, i.e. without an additional assumption on independence (such as Assumption \ref{asumivan1}), remains open.
\end{remark}	

		\subsection{Properties of $\beta_\infty(\l,\o)$}\label{ss:betainf}
		
		The main purpose of this subsection is to establish the deterministic nature of  $\beta_\infty(\l,\o)$, as well as its continuity and monotonicity. We begin with some obvious observations.
		
\begin{remark}
	
	\begin{enumerate}[(i)]
		
		\item  The use of cubes in the definition of $\ell(x,M,\lambda,\omega)$ is not essential. In fact, one can utilize the scaled and translated versions of any sufficiently regular bounded open set --- this will yield the same function $\b_\infty(\l,\o)$ (cf. also \eqref{sasha103}). The reason we use cubes is that it will be convenient for our constructions in the proofs of both theorems. 
 
 \item Clearly, one has 
 \begin{equation*} 
 	\begin{gathered}
 	\b(\l)=	\lim_{M\to \infty}\left(\l+\l^2\frac{1}{M^d}\int_{\square^M} b (T_y \omega,\lambda)dy\right) \qquad\qquad\qquad
 	\\
 \qquad\qquad\qquad	=  \sup_{x \in \mathbf{R}^d} \lim_{M\to \infty}\left(\l+\l^2\frac{1}{M^d}\int_{\square_x^M} b (T_y \omega,\lambda)dy\right) \leq \beta_\infty(\lambda,\omega). 
 	\end{gathered}
 \end{equation*} 
 	\end{enumerate}
\end{remark}

We first prove two auxiliary results which will be used throughout the remaining part of Section \ref{s:conv of spectrum}. The proof of the following  lemma is straightforward and requires only part 2 of Assumption \ref{kirill100}  and the bound \eqref{pauza1}.
\begin{lemma} \label{remivan2} Let $\nu >0$ be fixed, and denote (cf. \eqref{42})
	$$\ell_\nu(x,M,\lambda,\omega):=\lambda+\lambda^2\frac{1}{M^d}\int_{S_x^{M,\nu}} b(T_y \omega,\lambda) \,dy,$$
	where  $S_x^{M,\nu}$ is a measurable set satisfying $\square_x^{M-\nu} \subset S_x^{M,\nu} \subset \square_x^M$ (clearly, the value of $\ell_\nu(x,M,\lambda,\omega)$ depends on the choice of $S_x^{M,\nu}$, but we do not reflect it in the notation). Denote also 	
	$$\tilde{\ell}(x,M,\lambda,\omega):=\lambda+\lambda^2\frac{1}{M^d}\int_{\mathring\square_x^M(\omega)} b(T_y \omega,\lambda) \,dy,$$
	where $\mathring\square_x^M(\omega): = \cup_{\OO_\o^k\subset \square_x^M} \OO_\o^k$ (i.e. the inclusions touching the boundary of the cube are removed).	
Then
	$$ \beta_\infty(\lambda,\omega)=\liminf_{M \to \infty} \sup_{x \in \mathbf{R}^d}  \ell_\nu(x,M,\lambda,\omega)=\liminf_{M \to \infty} \sup_{x \in \mathbf{R}^d}  \tilde{\ell}(x,M,\lambda,\omega). $$

\end{lemma}
	Note that $\tilde{\ell}(x,M,\lambda,\omega)$ is equal to $\ell_\nu(x,M,\lambda,\omega)$ for an appropriate choice of $S_x^{M,\nu}$. 
\begin{lemma} \label{lemivan13}  
	Let $\{M_n\}_{n \in \mathbf{N}}\subset \R$ be an arbitrary sequence diverging to infinity. Then almost surely one has
	$$\beta_\infty(\lambda,\omega)=\lim_{n \to \infty} \sup_{x \in \mathbf{R}^d}  \ell(x,M_n,\lambda,\omega).$$	
\end{lemma} 
\begin{proof} 
	For a fixed $\omega \in \Omega$, let a sequence $C_n\to \infty$ be such that 
	$$\beta_\infty(\lambda,\omega)=\lim_{n \to \infty}   \sup_{x\in \mathbf{R}^d}\ell(x,C_n,\lambda,\omega),$$
	and suppose that there exists a subsequence of $M_n$ (still indexed by $n$) such that 
	$$ \lim_{n \to \infty} \sup_{x \in \mathbf{R}^d}  \ell(x,M_n,\lambda,\omega) > \beta_\infty(\lambda,\omega). $$
	We take a sequence $x_n$ such that
	\begin{equation*} 
		\lim_{n \to \infty} \sup_{x \in \mathbf{R}^d} \ell(x,M_n,\lambda,\omega)=  \lim_{n \to \infty} \ell(x_n,M_n,\lambda,\omega).
	\end{equation*}
	Clearly, there exists $n_0 \in \mathbf{N}$ and $\delta>0$ such that, for every $n\geq n_0$, 
	\begin{eqnarray*} 
	   \ell(x_n,M_n,\lambda,\omega)&>& \beta_\infty(\lambda,\omega)+\delta, \\
		  \Big|  \sup_{x \in \mathbf{R}^d}\ell(x,C_n,\lambda,\omega)-\beta_\infty(\lambda,\omega)\Big| &<& \frac{\delta}{2}. 
	\end{eqnarray*} 
	Next, we cover each cube  $\square_{x_n}^{M_n}$ with  cubes of edge length $C_{n_0}$. More precisely, for $N_n:= \lfloor M_n/C_{n_0}\rfloor C_{n_0}$ ($\lfloor\,\cdot \,\rfloor $ denotes the integer part) the cube $\square_{x_n}^{N_n} \subset \square_{x_n}^{M_n}$ is a union of $\lfloor M_n/C_{n_0}\rfloor^d$ disjoint cubes  of edge length $C_{n_0}$. On the one hand, observing that $N_n/M_n \to 1$ as $n\to \infty$ and applying Lemma \ref{remivan2} with $\nu = C_{n_0}$ we have 
	\begin{equation} \label{eqivan213} 
	\lim_{n \to \infty}   \ell(x_n, N_n,\lambda,\omega) = 	\lim_{n \to \infty}   \ell(x_n,M_n,\lambda,\omega) \geq \beta_\infty(\lambda,\omega)+\delta.
	\end{equation} 
	On the other hand, since
	\begin{equation*} 
		\ell(x_n, N_n,\lambda,\omega) = \frac{1}{\lfloor M_n/C_{n_0}\rfloor^d} \sum_{i=1}^{\lfloor M_n/C_{n_0}\rfloor^d} \ell(x_i', C_{n_0},\lambda,\omega),
	\end{equation*}
	where $x_i'$ denote the centres of the cubes in the partition described above, we conclude that 
	\begin{equation*}
		\lim_{n \to \infty}    \ell(x_n, N_n,\lambda,\omega) \leq \beta_\infty(\lambda,\omega)+\frac{\delta}{2},
	\end{equation*} 
	which is a contradiction with \eqref{eqivan213}.
\end{proof} 	
Finally, we present the main result of this subsection.
\begin{proposition} \label{propivan1} 
	The function $(\lambda,\omega) \mapsto \beta_\infty(\lambda,\omega)$
	is deterministic almost surely, continuous on  $\mathbf{R}^{+} \backslash \Sp(-\Delta_{\mathcal{O}})$ and strictly increasing on every interval contained in this set.	
\end{proposition} 	
\begin{proof} 
	By Lemmata \ref{remivan2} and  \ref{lemivan13}, we have  
	$$ \beta_\infty(\lambda,\omega)= \lim_{M \to \infty} \sup_{x \in \mathbf{Q}^d}  \tilde{\ell}(x,M,\lambda,\omega). $$ 
	The measurability of $\beta_\infty(\lambda,\cdot)$ now follows directly from the measurability of $b(\cdot,\l)$.
	
	It is easy to see from  \eqref{34}, \eqref{29} and \eqref{sasha102} that almost surely the mapping $\lambda \mapsto \tilde{\ell}(x,M,\lambda,\omega)$ has locally bounded (uniformly in $x,M,\omega$)  and  positive derivative, 
	$$
	\frac{\partial}{\partial \l} \tilde{\ell}(x,M,\lambda,\omega) \geq 1 - \frac{\sum_{\OO_\o^k\subset \square_x^M} |\OO_\o^k|}{M^d}\geq C>0.
	$$  
	Hence $\tilde{\ell}(x,M,\cdot,\omega)$ is locally Lipschitz (uniformly in $x,M,\omega$) and increasing on every open subinterval of its domain. More specifically, for $\l_1<\l_2$ contained in the same open interval from  $\mathbf{R}^{+} \backslash \Sp(-\Delta_{\mathcal{O}})$ one has
	\begin{equation}\label{47}
		\tilde{\ell}(x,M,\lambda_2,\omega) - \tilde{\ell}(x,M,\lambda_1,\omega) \geq C(\l_2 - \l_1).
	\end{equation}
	The bound \eqref{47} implies a similar property for $\beta_\infty$. Indeed, let $x_n$ be a sequence such that $\b_\infty(\l_1,\o) = \lim_{n\to \infty} \tilde{\ell}(x_n, n,\lambda_1,\omega)$. Passing to the limit along $x_n$ in \eqref{47}, we arrive at 
	\begin{equation*}
		\b_\infty(\l_2,\o) - \b_\infty(\l_1,\o) \geq  \limsup_{n\to \infty} \tilde{\ell}(x_n, n,\lambda_2,\omega) -   \lim_{n\to \infty} \tilde{\ell}(x_n, n,\lambda_1,\omega)  \geq C(\l_2 - \l_1).
	\end{equation*}
	Utilising a similar argument, one can prove the reverse inequality, concluding that $\beta_\infty(\lambda,\omega)$ is locally Lipschitz.
	
	Finally,  for a fixed $\l$ the function $\beta_\infty(\lambda,\cdot)$ is  translation invariant and  is thus constant almost surely. To conclude that the function $\beta_\infty$ is deterministic, i.e. $\beta_\infty(\lambda,\omega)=\b_\infty(\l)$ almost surely, it suffices to take the set of probability one such that for every $\l \in \Q$, $\l \notin \Sp(-\Delta_{\mathcal{O}})$,  $\beta(\l,\omega)$ is deterministic and use the almost sure continuity of $\lambda \mapsto \beta_\infty(\lambda,\omega)$.   
\end{proof} 

\begin{remark}
	Note that in general $\beta_\infty(\lambda)$ is not necessarily differentiable, as can  be  seen from examples provided at the end of this section.
\end{remark}

\subsection{Proof of Theorem \ref{th4.1}}\label{ss:4.3}

Let $\l$ be a limit point of the spectra of $\AA^\e$, i.e. $\l = \lim_{\e\to 0} \l^\e$, $\l^\e\in \Sp(\AA^\e)$.
Without loss of generality we may assume that 
\begin{equation}  \label{jdba1} 
\l \notin \Sp(-\Delta_\OO).
\end{equation} 
The task at hand is to show that 
\begin{equation}\label{43}
	\beta_\infty(\l) \geq 0.
\end{equation}

Since $\l^\e\in \Sp \AA^\e$, there exists a sequence $u^\e \in {\rm Dom}(\mathcal A^\e)$, $\|u^\e\|_{L^2(\R^d)}=1$,   such that
\begin{equation}\label{delta}
(\mathcal A^\e -\l^\e)u^\e =  f^\e
\end{equation}
 with  $\|f^\e\|_{L^2(\R^d)} = :\delta^\e \to 0$. (Such a sequence can be extracted by the diagonalisation procedure from  Weyl sequences corresponding to $\l^\e$ for each $\e$.) Multiplying \eqref{delta} by $u^\e$ and integrating by parts, we obtain the estimate
\begin{equation}\label{44a}
\|\e \nabla u^\e\|_{L^2(S_0^\e)} + \| \nabla u^\e\|_{L^2(S_1^\e)} \leq C,
\end{equation}
with some $C>0$. The sequence $(u^\e)_\e$ converges weakly to zero in $L^2(\R^d)$ (otherwise its weak stochastic two-scale limit would be an eigenfunction of $\AA^\hom $, which is impossible due to Proposition \ref{p4.13}).   In what follows we devise a ``compactification'' procedure for $(u^\e)_\e$ modifying it in such a way that the new sequence has a non-vanishing limit, thus allowing us  to retrieve  information about the relation of $\l$ to the set $\mathcal G$.

We fix an arbitrary $L>0$ and cover $\R^d$ by the cubes $\square^L_\xi$, $\xi \in L\Z^d$. A key ingredient of the ``compactification'' is the following obvious assertion.
\begin{lemma}
	Let (possibly finite) sequences $a_\xi, b_\xi, c_\xi, d_\xi, \xi\in \mathbb I$, where $\mathbb I$ is a countable or finite set of indices, be such that
	\begin{equation*}
	\sum_\xi a_\xi = \sum_\xi b_\xi = \sum_\xi c_\xi  = \sum_\xi d_\xi  =  1.
		\end{equation*}
	Then there exists $\zeta\in \mathbb I$ such that $b_\zeta+c_\zeta+d_\zeta	    \leq 3 a_\zeta$.
\end{lemma}
For each $\e$ we apply the above lemma to the  sequences 
\begin{equation*}
{\| u^\e\|_{L^2(\square^{L}_\xi)}^2}, \quad\frac{\|(\mathcal A^\e -\l^\e)u^\e\|_{L^2(\square^{3L}_\xi)}^2}{3^d (\delta^\e)^2}, \quad \frac{1}{3^d}\| u^\e\|_{L^2(\square^{3L}_\xi)}^2,\quad \frac{\|\chi_1^\e \nabla u^\e\|_{L^2(\square^{3L}_\xi)}^2}{3^d \| \chi_1^\e \nabla  u^\e\|_{L^2(\R^d)}^2}, 
\end{equation*} $\xi \in L\Z^d$, $\| u^\e\|_{L^2(\square^{L}_\xi)} \neq 0$. (We only consider those cubes $\square^{L}_\xi$ where at least one of the above terms  does not vanish.) Taking into account \eqref{44a},  we infer that for each $\e$ there exists $\xi^\e$ such that
\begin{equation}\label{50}
	\begin{aligned}
	(\delta^\e)^{-1} 	\|(\mathcal A^\e -\l^\e)u^\e\|_{L^2(\square^{3L}_{\xi^\e})} +	\|u^\e\|_{L^2(\square^{3L}_{\xi^\e})} +\| \chi_1^\e \nabla u^\e\|_{L^2(\square^{3L}_{\xi^\e})} \leq C  \| u^\e\|_{L^2(\square^{L}_{\xi^\e})}.
	\end{aligned}
\end{equation}

Now for every $\e$ we shift the cube $\square^{L}_{\xi^\e}$ to the origin and re-normalise $u^\e$. Namely,  we define
\begin{equation}\label{47b}
	w^\e_L(x):= \frac{u^\e(x+\xi^\e)}{\| u^\e\|_{L^2(\square^{L}_{\xi^\e})}}. 
\end{equation}
Note that 
\begin{equation}\label{42a}
	\| 	w^\e_L\|_{L^2(\square^{L})} =1.
\end{equation}
We denote by $\check{\mathcal{A}}^\e$  the operator obtained from $\mathcal A^\e$ by shifting its coefficients: $\check A^\e(x,\o) :=A^\e(x+\xi^\e,\o)$. Analogously, $\check S_0^\e, \check S_1^\e$, and $\check \chi_0^\e, \check \chi_1^\e$, denote appropriately shifted set of inclusions, its complement, and the corresponding characteristic functions. Furthermore, it is convenient to denote by $\check \OO_\o$ and $\check \OO_\o^k$ appropriately shifted unscaled set of inclusions and individual unscaled inclusions (which depend on $\e$ via $\xi_\e$, but we omit this dependence from the notation for brevity). 

From \eqref{50} we immediately have 
\begin{equation}\label{17}
	\begin{aligned}
		\|f^\e_L\|_{L^2(\square^{3L})}\leq C \delta^\e
	\end{aligned}
\end{equation}
 and 
\begin{equation}\label{45a}
	\begin{aligned}
		\|  w^\e_L\|_{L^2(\square^{3L})} + \|\check \chi_1^\e \nabla   w^\e_L\|_{L^2(\square^{3L})} \leq C,
	\end{aligned}
\end{equation}
where
\begin{equation}\label{17a}
	\begin{aligned}
	f^\e_L:=(\check{\mathcal{A}}^\e -\l^\e)w^\e_L =\frac{f^\e(\cdot + \xi^\e)}{\| u^\e\|_{L^2(\square^{L}_{\xi^\e})}}.
	\end{aligned}
\end{equation}

In order to verify \eqref{43}, we need to analyse the behaviour of the sequence $w^\e_L$ on the soft component. To this end, we consider the  decomposition
\begin{equation}\label{51}
w^\e_L = 	\widetilde w^\e_L + z^\e,
\end{equation}
where $\widetilde w^\e_L$, representing the macroscopic part, denotes the harmonic extension of $w^\e_L|_{\check S_1^\e}$ to the whole $\R^d$, as per Theorem \ref{th:extension} and a standard scaling argument, and the term
\begin{equation}\label{47a}
	z^\e:= w_L^\e - \widetilde w_L^\e \in W^{1,2}_0(\check S_0^\e)
\end{equation}
captures the micro-resonant behaviour of the inclusions. Further on we will use analogous decompositions on several occasions.

Since $\widetilde w_L^\e$ is harmonic in the inclusions, we have 
\begin{equation}\label{eq:v}
	-(\e^2 \Delta + \l^\e)  z^\e = \l^\e \widetilde w_L^\e + f_L^\e,\quad x\in \check S_0^\e.
\end{equation}
The following lemma is trivial and follows  from \eqref{20a} by a rescaling argument.
\begin{lemma}\label{l:res}
	The spectrum of the self-adjoint operator $-\e^2 \Delta$ on $\check S_0^\e$ with Dirichlet boundary conditions  coincides with the spectrum of $-\Delta_\OO$ almost surely.
\end{lemma}
By the above lemma, for sufficiently small $\e$ we have
\begin{equation}\label{54}
	\|z^\e\|_{L^2(\e\check\OO_\o^k)} \leq d_{\l^\e}^{-1}{\| \l^\e \widetilde w_L^\e + f_L^\e\|_{L^2(\e\check\OO_\o^k)}} \quad \forall \check\OO_\o^k \subset\check\OO_\o
\end{equation}
(recall the notation \eqref{d_l}). Furthermore,  using $z^\e$ as a test function in \eqref{eq:v} and taking into account \eqref{54}, it is not difficult to see that 
\begin{equation}\label{49}
	\e\|\nabla z^\e\|_{L^2(\e\check\OO_\o^k)} \leq C(\|\widetilde w_L^\e\|_{L^2(\e\check\OO_\o^k)} +\| f^\e\|_{L^2(\e\check\OO_\o^k)})\quad \forall \check\OO_\o^k \subset\check\OO_\o.
\end{equation}
Utilising Theorem \ref{th:extension} and the bounds \eqref{17}, \eqref{45a}, \eqref{49}, we infer that 
\begin{equation}\label{53}
	\begin{gathered}
		\| \widetilde w^\e_L\|_{L^2(\square^{5L/2})}  + \|\nabla \widetilde w^\e_L\|_{L^2(\square^{5L/2})} \leq C, \quad
\e \|\nabla  w^\e_L\|_{L^2(\square^{2L})} \leq C.
	\end{gathered}
\end{equation}
Next, by  \eqref{17} and \eqref{54} we have
\[
	\| 	w^\e_L\|_{L^2(\square^{L})} \leq 	\| 	\widetilde w^\e_L\|_{L^2(\square^{L})}  + 	\| 	z^\e\|_{L^2(\square^{L})} \leq  (1 + d_{\l^\e}^{-1} \, {\l^\e})	\| 	\widetilde w^\e_L\|_{L^2(\square^{3L/2})} + {C} d_{\l^\e}^{-1}\, \delta^\e.
\]
Combining \eqref{jdba1}, \eqref{delta} and \eqref{42a}, we infer that,  for  small enough $\e$,
\begin{equation}\label{53a}
		0<C\leq \|\widetilde w^\e_L\|_{L^2(\square^{3L/2})}.
\end{equation}
From \eqref{53a} and  the first  bound in \eqref{53}   we have that, up to a subsequence,  
 \begin{equation}\label{57}
 	\widetilde w^\e_L \to w^0_L \neq 0 \mbox{ weakly in } W^{1,2}(\square^{5L/2}) \mbox{ and strongly in }L^2(\square^{5L/2}).
 \end{equation}

 
 Let $\eta_L \in C_0^\infty(\square^{2L}), L>0$, be a family of cut-off functions satisfying  $0\leq\eta_L\leq 1$, $\eta_L\vert_{\square^{3L/2}} = 1$, $|\nabla \eta_{L}|\leq C/L$, with $C$ independent of $L$.  Using  $\widetilde w^\e_L \eta_{L}$ as a test function  in (\ref{17a}) and integrating by parts, we get
 \begin{equation}\label{59}
 	\begin{aligned}
 		I_1^\e + I_2^\e+I_3^\e: =& \int_{ \square^{2L}}\check \chi_1^\e \eta_{L}   A_1 \nabla w^\e_L \cdot \nabla  w^\e_L + \int_{ \square^{2L}} \check \chi_1^\e   w^\e_L  A_1 \nabla w^\e_L \cdot \nabla \eta_{L} 
 		\\
 		 + &\int_{ \square^{2L}} \e^2 \check \chi_0^\e \nabla w^\e_L \cdot  \nabla (\widetilde w^\e_L  \eta_{L})  
 		= \int_{\square^{2L}}\l^\e  w^\e_L  \widetilde w^\e_L \eta_{L} +  \int_{\square^{2L}} f^\e_L \widetilde w^\e_L \eta_{L}.
 	\end{aligned}
 \end{equation}  
We estimate all the terms but one via \eqref{17} and \eqref{53} as follows:
 \begin{equation}\label{60}
 	\begin{aligned}
 		\liminf_{\e\to 0} I_1^\e \geq 0, \quad |I_2^\e| \leq C/L, \quad \lim_{\e\to 0} I_3^\e = 0, \quad \lim_{\e\to 0}\int_{\square^{2L}} f^\e_L \widetilde w^\e_L \eta_{L} =0.
 	\end{aligned}
 \end{equation}

 It remains to analyse the behaviour of the first term on the right-hand side of \eqref{59}. 
We use the family of local averaging operators $P^\e$  on $L^2(\R^d)$ defined in Lemma \ref{solta10}, taking for $X^\e_k$ the inclusions $\e\check\OO^k_\omega \subset \check S_0^\e$ and considering the decomposition 
\begin{equation}	\label{54a} 
	z^\e = \hat z^\e + \mathring z^\e,
\end{equation} 
where $\hat z^\e, \mathring z^\e \in W^{1,2}_0(\check S_0^\e)$ satisfy 
 \begin{equation*}
 	- \e^2 \Delta \hat z^\e -\l \hat z^\e = \l P^\e  w^0_L  \quad \mbox{ in } \check S_0^\e,
 \end{equation*}
 \begin{equation}\label{66a}
 	- \e^2 \Delta \mathring z^\e -\l^\e \mathring z^\e = (\l^\e - \l) \hat z^\e + \l^\e \widetilde w^\e_L - \l P^\e  w^0_L  + f^\e_L \quad \mbox{ in } \check S_0^\e.
 \end{equation}
First, it is easy to see that
 \begin{equation}\label{56} 
 	\hat z^\e(x)=\l   (P^\e w^0_L)(x)\,  b (T_{x/\e+\xi^\e}\o,\l). 
 \end{equation}
 Furthermore, it follows from Lemmata  \ref{l:res} and \ref{solta10} that for sufficiently small $\e$ one has
 \begin{equation} \label{estuniformb} 
 	\|\hat z^\e\|_{L^2(\square^{2L})} \leq C d_\l^{-1} \l\|w^0_L\|_{L^2(\square^{5L/2})} \leq C. 
 \end{equation}
 From the  bound \eqref{estuniformb}, \eqref{57}, Lemma \ref{solta10}, and the bound \eqref{17} we infer  that the right-hand side of \eqref{66a} vanishes in the limit as $\e\to 0$. Then,  applying Lemma \ref{l:res} again, we infer that
 \begin{equation}\label{67}
 	\|\mathring z^\e \|_{L^2(\square^{2L})} \to 0.
 \end{equation}

 Next, denote $	g^{\varepsilon}_{\l}(x): =b (T_{x/\e+\xi^\e}\o,\l)$.
It follows from \eqref{pauza1}  that the sequence  $(g^{\varepsilon}_{\l})_\e$ is bounded in $L^2(\square^{2L})$ and thus converges, up to a subsequence, weakly in $L^2(\square^{2L})$ as $\e\to 0$ to some $g_{\l}$. 
 Integrating $\l+\l^2g_{\l}^\e$ over an arbitrary fixed cube contained in $\square^{2L}$, passing to the limit, and using the definition of $\beta_{\infty}(\l)$,  it is not difficult to see that 
 \begin{equation*}
 	\l+\l^2g_{\l} (x)\leq \beta_{\infty}(\l) \textrm{ for a.e. } x \in \square^{2L}.
 \end{equation*} 
 Using Lemma \ref{solta10}, we obtain from \eqref{54a},\eqref{56} and \eqref{67} that
 $$z^\e \rightharpoonup \l g_{\l} w^0_L, \textrm{ weakly in } L^2(\square^{2L}).$$
 Finally, the strong convergence of $\widetilde w^\e_L$, the relations  \eqref{60}, and a passage to the limit in \eqref{59} as $\e\to 0$ yield
 \begin{equation*}
 - \frac{C}{L}\leq  \lim_{\e\to 0}	\int_{\square^{2L}}\l^\e (\widetilde w^\e_L + z^\e) \widetilde w^\e_L \eta_{L}= \int_{\square^{2L}} (\l+\l^2 g_{\l}) (w^L_0)^2\eta_{L}
 \leq  \beta_{\infty} (\l) \int_{\square^{2L}}(w^L_0)^2\eta_{L}\leq C \beta_{\infty}(\l), 
 \end{equation*} 
 where in the last inequality we used  \eqref{53a} and \eqref{57}. Since $L$ is arbitrary, noting that all constants in the bounds obtained in the proof are independent of $L$, we conclude that $\beta_{\infty}(\l)\geq 0$.

 \subsection{Existence of almost periodic cubes} \label{ss:4.4}

In this subsection we present preparatory constructions and results necessary for the proof of Theorem \ref{th4.4}.
 
By $P_{\#}$ we denote the push-forward of the probability measure given by the map $\mathcal{H}_K$ on the $\sigma$-algebra $\mathcal{F}_{{\rm H},K}$ (see Section \ref{s4.1}). We  use the following notation for a closed Hausdorff ball of radius $r$ around a fixed compact set $K' \subset K$: 
\[
{B_{{\rm H},K}(K',r)}:= \{U\subset K: U \mbox{ is compact, } d_{\rm H}(K',U)< r\}.
\] 
Next lemma is the key assertion that implies the existence of ``almost periodic'' arrangements of inclusions.  
\begin{lemma} \label{l:4.11}  
	Let $K \subset \mathbf{R}^d$ be a  compact set. There exists a subset $\Omega_K \subset \Omega$ of probability one such that for every $\omega \in \Omega_K$ and every $r>0$ one has
	\begin{equation} \label{svoivan1} 
	P_{\#}\left(B_{{\rm H},K}(\mathcal{H}_{K}(\omega),r  )\right)>0. 
	\end{equation}    
\end{lemma} 	
\begin{remark}
	One can rewrite \eqref{svoivan1} as $P\left(\{\tilde \o: d_{\rm H}(\mathcal{H}_K(\o),\mathcal{H}_K(\tilde \o))< r\}\right)>0$.
\end{remark}
\begin{proof} 
It is sufficient to prove that  the set where the inequality \eqref{svoivan1} does not hold for any $r>0$ is a set of probability zero. Recall that the Hausdorff topology on the compact subsets of $K$ is compact and thus separable.  Consider a countable  family $\{K_m\}_{m \in \mathbf{N}}$ of compact subsets of $K$ that is dense in this topology and let 
	\begin{equation}\label{80a}
		r_m:=\sup\left\{r>0: P_{\#}\left(B_{{\rm H},K} (K_m,r)\right)=0\right\}.
	\end{equation} 
	Additionally, we set $r_m = -\infty$ if $K$ is empty. Note that by the continuity of probability measure  we also have  $P_{\#}\left(B_{{\rm H},K} (K_m,r_m)\right)=0$. 
	Let $\o$ be such that there exists $r>0$ with $P_{\#}\left(B_{{\rm H},K}(\mathcal{H}_{K}(\omega),r)\right)=0$. By the density of the family $\{K_m\}_m$, there exists $K_i$ and $r'>0$  satisfying $\mathcal{H}_{K}(\omega) \in B_{{\rm H},K}(K_i,r') \subset B_{{\rm H},K}(\mathcal{H}_{K}(\omega),r)$. As $P_{\#}\left( B_{{\rm H},K}(K_i,r')\right) = 0$, it follows from \eqref{80a} that
\begin{equation*}
	\begin{gathered}
	\left\{\omega \in \Omega: \exists r>0 \textrm{ such that }   	P_{\#}\left(B_{{\rm H},K}(\mathcal{H}_{K}(\omega),r  )\right)=0\right\}\qquad\qquad
	\\
	\qquad\qquad=\bigcup_{m \in \mathbf{N},r_m>0}\left\{\omega \in \Omega:\mathcal{H}_{K}(\omega) \in B_{{\rm H},K }(K_m,r_m)\right\}.
\end{gathered}
\end{equation*}
	The right-hand side of the last equality is clearly a set of probability zero.
\end{proof} 

\begin{corollary} \label{lemivan10}  
	There exists a subset $\Omega_1 \subset \Omega$ of probability one such that every $\omega \in \Omega_1$ one has
	\begin{equation*}
		P_{\#}\left(B_{{\rm H},\square_q^n}(\mathcal{H}_{\square_q^n}(\omega),r  )\right)=	P_{\#}\left(B_{{\rm H},\square^n}(\mathcal{H}_{\square^n}(T_q\omega),r  )\right)>0 \quad \forall n \in \mathbf{N}, q \in \mathbf{Q}^d, r>0.
	\end{equation*}    
\end{corollary}	
\begin{proof}
	The equality follows from a straightforward translation argument. The existence of $\Omega_1$ follows from Lemma \ref{l:4.11}  and a simple observation that the intersection of a countable family of sets of probability one is a set of probability one. 
\end{proof}

The following two theorems contain the main result of the present subsection. Though the notation is somewhat involved, their meaning can be expressed as follows: for a fixed $\l$, there is an arbitrarily large cube that can be divided into smaller sub-cubes with almost periodic arrangements of inclusions (apart from a fixed size boundary layer) such that on each  sub-cube  the associated quantity $\ell(\ldots,\o,\l)$ approximates $\beta_\infty(\l)$.

\begin{theorem}\label{th:6.10}
	 Let $\hat\o$ be a {\it typical} element of $\Omega$ for which the statement of Corollary \ref{lemivan10} is satisfied. Under  Assumptions \ref{asumivan1} there exists a set $\Omega_1 \subset \Omega$ of probability one (which in general depends on the choice of $\hat\o$) such that for any $\omega \in \Omega_1$, $N,M\in \mathbf{N}$, $x\in \R^d$, $\delta >0$, and the cube $\square_x^M$ there exist $N^d$    cubes $\square_{x_i}^M,$ $i=1,\dots,N^d$, such that the cubes $\square_{x_i}^{M+\kappa}$ (discarding the boundary) are mutually disjoint, their union is a cube  $\cup_i \square_{x_i}^{M+\kappa} = :\square_{x^*}^{N(M+\kappa)}$, and the following estimate holds:
	 \begin{equation}\label{67a}
	 	d_{\rm H}\left(\mathcal{H}_{\square^M}(T_x\hat\o),\mathcal{H}_{\square^M}(T_{x_i}\omega)\right)<\delta.
	 \end{equation}
\end{theorem}
\begin{proof}
	It is direct consequence of Corollary \ref{lemivan10},  Assumption \ref{asumivan1}, and the law of large numbers that for any $N,M,n\in \mathbf{N}$, $x\in \Q^d$, for almost every $\omega \in \Omega$ there exist $N^d$ cubes $\square_{x_i}^M,$ $i=1,\dots,N^d$, such that  $\cup_i \square_{x_i}^{M+\kappa} = \square_{x^*}^{N(M+\kappa)}$ for a suitable $x^*\in \R^d$, and
	\begin{equation*}
		d_{\rm H}\left(\mathcal{H}_{\square^M}(T_x\hat\o),\mathcal{H}_{\square^M}(T_{x_i}\omega)\right)<n^{-1}
	\end{equation*}
We denote the set of such $\omega$ by  $\Omega_{N,M,n,x}(\hat\o)$ and set
	\begin{equation*}
		 \Omega_1(\hat\o):=\bigcap_{N,M,n \in \mathbf{N},x\in \Q^d}\Omega_{N,M,n,x}(\hat\o).
	\end{equation*} 
 Clearly, $\O_1(\hat\o)$ has probability one, and the statement of the theorem holds for any $x\in \Q^d$, but then, by the density argument it holds for any $x\in \R^d$.
\end{proof}

\begin{theorem} \label{thmivan1} Under  Assumption \ref{asumivan1}, there exists a set $\Omega_1 \subset \Omega$ of probability one such that for any  $\lambda \in \R\setminus\Sp(-\Delta_{\mathcal{O}})$, $\omega \in \Omega_1$, $\delta>0$, $N,M_0\in \N$ there exist $N^d$  cubes $\square_{x_i}^M,$ $i=1,\dots,N^d$, $\N \ni M>M_0$, as described in Theorem \ref{th:6.10}, such that 
	\begin{equation} \label{eqivan50} 
	 \max_{i}\left|\ell(x_i,M+\kappa,\lambda,\omega)-\beta_\infty(\lambda)\right|<\delta. 
	\end{equation}   	
\end{theorem} 	
\begin{proof} Let us fix $\lambda \in \Q\setminus\Sp(-\Delta_{\mathcal{O}})$. Let us fix a {\it typical} $\hat\o\in \O$ for which the statements of  Lemma \ref{lemivan13}, Proposition \ref{propivan1} and Corollary \ref{lemivan10} hold. Then from these assertions and   Lemma \ref{remivan2} it follows that for all sufficiently large  $M\in \N$ there exist  cubes $\square_{\overline{x}_M}^M$ such that
	\begin{equation}\label{87b}
		|\beta_\infty(\lambda)-  \tilde \ell(\overline{x}_M,M,\lambda,\hat\o)| < {\delta}/3.
	\end{equation} 
Similarly to Lemma \ref{remivan2}, it is not difficult to show, via part 2 of  Assumption \ref{kirill100}  and the bound \eqref{pauza1}, that for for all sufficiently large  $M$ and any $x$ one has
\begin{eqnarray}\label{86a}
	 |\ell(x,M+\kappa,\lambda,\hat\o)-  \tilde \ell(x,M,\lambda,\hat\o)| < {\delta}/3.
\end{eqnarray}
By Theorem \ref{th:6.10}, there exists a set of full measure $\mathring\Omega = \mathring\Omega(\hat\o) \subset \O$ satisfying the following: for any $\omega \in \mathring\Omega$ there is a sequence of cubes $\square_{x^*_n}^{N(M+\kappa)}$, $n\in \N$, such that for each $n$ the cube $\square_{x^*_n}^{N(M+\kappa)}$ is a union of mutually disjoint cubes $\square_{x_i^n}^{M+\kappa}$, $i=1,\dots,N^d$, and 
 \begin{equation*}
\lim_{n\to \infty} \max_i\,	d_{\rm H}\left(	\mathcal{H}_{\square^M}(T_{\overline{x}_M}\hat\o),\mathcal{H}_{\square^M}(T_{x_i^n}\omega)\right) =0.
\end{equation*}
It is clear, by direct inspection, that the difference between $\tilde \ell(\overline{x}_M,M,\lambda,\hat\o)$ and $\tilde \ell(x_i^n,M,\lambda,\omega)$ is controlled (uniformly in $M$) by the maximal difference between the values of $\int_{\R^d} \tilde b(T_x, \l, \cdot)$ corresponding to the ``matching'' inclusions in $\square_{\overline{x}_M}^M$ and $\square_{x_i^n}^M$. Hence, by Lemma \ref{l:bconv}, for sufficiently large (fixed) $n$ one has
 \begin{equation}\label{72a}
 	|\tilde \ell(x_i^n,M,\lambda,\omega)-  \tilde \ell(\overline{x}_M,M,\lambda,\hat\o)| < {\delta}/3.
\end{equation} 

Combining inequalities \eqref{87b}, \eqref{86a} and \eqref{72a}, we obtain  \eqref{eqivan50}  (with $x_i:=x_i^n$) for each $\lambda \in \Q\setminus\Sp(-\Delta_{\mathcal{O}})$ on a set of probability one, which, in general, depends on $\lambda$ (as well as on $\hat\o$, although the latter is of no importance). We define $\O_1$ as the intersection of these sets for all $\lambda \in \Q\setminus\Sp(-\Delta_{\mathcal{O}})$. Then by continuity of $\beta_\infty(\l)$  we conclude that the assertion holds for all $\lambda \in \R\setminus\Sp(-\Delta_{\mathcal{O}})$.

\end{proof}

 \subsection{Proof of Theorem \ref{th4.4}}\label{ss:4.5}
 
 We begin by outlining the main idea of the proof. In  view of \eqref{12} and Theorem \ref{th:3.3}, we only need to consider the case when $\l$ is such that  $\b_\infty(\l)\geq0$. We fix $\hat\o\in\O$ as in Theorem \ref{thmivan1}. Our goal is to construct an approximate solution to the spectral problem for the operator $\AA^\e$. It is sufficient for this approximate ``eigenfunction'' to be from the operator form domain and satisfy the equation in a certain weak sense. Lifting it up to the domain of $\AA^\e$ can then be carried out via a simple abstract argument, see Lemma \ref{l:E.1}. 
 
 One of the key technical ingredients of our construction is the use of a periodic corrector on nearly periodic geometries. More specifically, one can always find a cube  $\square_{\overline{x}_M}^M$ such that the corresponding local spectral average $ \ell(\overline{x}_M,M,\lambda,\hat\o)$, cf. Remark \ref{r:spec_av}, approximates  $\b_\infty(\l)$ with a given error. According to Theorem \ref{thmivan1}, for any $\o\in \O_1(\hat\o)$ one can find an arbitrarily large cube $\square_{x^*}^{N(M+\kappa)}$ tiled by the cubes  $\square_{x_i}^{M},$ $i=1,\dots,N^d$, separated by the correlation distance $\kappa$ from Assumption \ref{asumivan1}, such that $\square_{x_i}^{M}\setminus \OO_\o$ are almost exact copies   of the set $\square_{\overline{x}_M}^M\setminus \OO_{\hat \o}$. Moreover, the local spectral average $ \ell$ on the  large cube  $\square_{x^*}^{N(M+\kappa)}$ is also close to   $\b_\infty(\l)$. This almost periodic structure of the composite in $\square_{x^*}^{N(M+\kappa)}$ is important, for one then can approximate the ``macroscopic'' part of the the operator $\AA^\e$ on  $\square_{\e x^*}^{\e N(M+\kappa)}$ by an operator with constant coefficients (obtained by homogenisation on a perfectly periodic structure) and use this as the basis for the construction of a quasimode, see  \eqref{67b} below.

 We emphasise that the set of full measure $\Omega_1$, for which the statement of Theorem \ref{th4.4} holds, depends in general on the choice of $\hat\o$ (note that $\Omega_1$, in principle, does not have to contain $\hat \o$).  	 In what follows we introduce a number objects derived from $\hat \o$, such as  the periodic correctors $\hat N_j, j=1,\dots,d,$ and  the matrix of homogenised coefficients $\hat A^\hom_1$,  while other quantities, e.g. $b^\e$, correspond to $\o \in \Omega_1$. It is important to keep account of these dependencies.

 Henceforth the parameters $L$ and $M$ are assumed sufficiently large, and $\e$ and $\delta$ are sufficiently small. The constants that appear in the subsequent bounds may depend on $\l$ but are independent of $\e$, $\delta$, $L$ and $M$. We will emphasise this through the notation in two key bounds below. We fix $\delta>0$ and $L>0$ and choose $M$ so that \eqref{86a} holds and \eqref{87b} is satisfied for some cube $\square_{\overline{x}_M}^M$ (by Lemmata \ref{remivan2} and \ref{lemivan13} ).  Let $N=N(\e)$ be the smallest integer such that $\e N(M+\kappa)\geq L$. By Theorems \ref{th:6.10} and \ref{thmivan1}  there exist   cubes $\square_{x_i}^{M+\kappa},$ $i=1,\dots,N^d$, (as described in Theorem \ref{th:6.10}) satisfying \eqref{67a} and \eqref{eqivan50}. In particular, their union is $\square_{x^*}^{N(M+\kappa)}$ and 
  	 $$\square_{\e x^*}^{\e N(M+\kappa)} \supseteq \square_{\e x^*}^L.$$
  	  Clearly, the choice of the cubes (equivalently, their centres $x_i, x^*$) depends on $L,M, \delta$ and $\e$, which we omit in the notation for brevity.  
  	
  	  Denote by $\hat\square^{M+\kappa,1}_{\overline{x}_M}$ the set obtained from $\square^{M+\kappa}_{\overline{x}_M}$ by removing all the sets $\OO_{\hat\o}^k$ whose closures are contained in $\square^M_{\overline{x}_M}$,
  	\begin{equation*}
  		  	  \hat\square^{M+\kappa,1}_{\overline{x}_M}:= \square^{M+\kappa}_{\overline{x}_M}\setminus \bigcup_{ \overline{\mathcal{O}^k_{\hat\o}} \subset \square^M_{\overline{x}_M}} \mathcal{O}^k_{\hat\o}.
  	\end{equation*}
  	   Let $\hat N_j \in W^{1,2}_{\rm per}(\hat\square^{M+\kappa,1}_{\overline{x}_M}), j=1,\dots,d,$ be the solutions to the periodic corrector problems for the perforated cube $\hat\square^{M+\kappa,1}_{\overline{x}_M}$:
  	\begin{equation}\label{88}
  		\int\limits_{\hat\square^{M+\kappa,1}_{\overline{x}_M}}  A_1  (e_j + \nabla \hat N_j) \cdot \nabla\varphi = 0  \quad \forall \varphi \in W^{1,2}_{\rm per}(\hat\square^{M+\kappa,1}_{\overline{x}_M}).
  	\end{equation}
  	We assume that each $\hat N_j$ is extended inside the inclusions in  $\square^{M+\kappa}_{\overline{x}_M}$ according to Theorem \ref{th:extension}, has zero mean over $\square^{M+\kappa}_{\overline{x}_M}$, and extended by periodicity to the whole of $\R^d$. Observe the following estimates with the constant depending only on $A_1$ and the extension constant $C_{\rm ext}$,
  	\begin{equation}\label{95}
  		\begin{gathered}
  			\|\nabla \hat N_j\|_{L^2(\square^{M+\kappa}_{\overline{x}_M})} \leq C M^{d/2}, \qquad
  			\| \hat N_j\|_{L^2(\square^{M+\kappa}_{\overline{x}_M})} \leq C M^{d/2 + 1}.
  		\end{gathered}
  	\end{equation}
  The first easily follows from the identity 
  	\[
  	\int\limits_{\hat\square^{M+\kappa,1}_{\overline{x}_M}} A_1\nabla \hat N_j \cdot \nabla \hat N_j =   - \int\limits_{\hat\square^{M+\kappa,1}_{\overline{x}_M}} A_1 e_j \cdot \nabla \hat N_j,
  	\]
  	and the second from the Poincar\'e inequality (recall that $M$ is sufficiently large, in particular $M\geq\kappa$).

  	An essential component of the construction is the higher (than $L^2$) regularity of the correctors $\hat N_j$. The proof closely follows the argument of \cite{Bensoussan} and is based on the use of special versions of two well-known results: the Poincar\'e-Sobolev inequality and the reverse H\"older's inequality. In particular, the uniform scalable version of Poincar\'e-Sobolev inequality for perforated domains is valid under the minimal smoothness assumption.
  	\begin{theorem}[Higher regularity of the periodic corrector]\label{th:6.18}
  		Under Assumption \ref{kirill100} there exist $p>2$ and $C>0$ such that for a.e. $\o$ one has 
  		\begin{equation}\label{89a}
  			\bigg(\fint\limits_{\square^{M+\kappa}_{\overline{x}_M}} |\nabla \hat N_j |^p \bigg)^{1/p} \leq C  + C  \bigg(\fint\limits_{\square^{M+\kappa}_{\overline{x}_M}}|\nabla \hat N_j |^2 \bigg)^{1/2},
  		\end{equation}	
  		uniformly in $M$. (Here $\fint$ denotes the average value.)
  	\end{theorem}
We provide the  proof of the theorem in Appendix  \ref{ap:regul}. As a corollary of the theorem and the bound \eqref{95}, we have the bound  
  	\begin{equation}\label{93}
	\begin{gathered}
		\|\nabla \hat N_j\|_{L^p(\square^{M+\kappa}_{\overline{x}_M})} \leq C M^{d/p}.
	\end{gathered}
\end{equation}

  	We denote by $\hat A^\hom_1$ the matrix of homogenised coefficients associated with $\hat N_j$, 
  	\begin{equation*}
		\hat A^\hom_1 \xi = {\left|\square^{M+\kappa}_{\overline{x}_M}\right|^{-1}}\int_{\hat\square^{M+\kappa,1}_{\overline{x}_M}} A_1(\xi+\xi_j \nabla \hat N_j) \quad \forall \xi \in \R^d.
  	\end{equation*}  
Take some $k$  such that $\hat A_1^{\rm hom}k \cdot k = \b_\infty(\l)$. Then  $u(x) := {\mathcal Re}(e^{ik\cdot x})$ satisfies 
  \begin{equation}\label{67b}
  	-\nabla \cdot \hat A_1^{\textrm{\rm hom}} \nabla u = \b_\infty(\l) u.
  \end{equation} 	
   	 We define  $N_j \in W^{1,2}_{\rm per}(\square^{M+\kappa}_{x_1})$  as 
   	\begin{equation*}
   		 N_j(x):= \hat N_j(x-x_1+\overline{x}_M),
   	\end{equation*}
   and extend it periodically to $\R^d$. Denote by $\hat \chi_1^\e$  the characteristic function of the set $\hat\square^{M+\kappa,1}_{\overline{x}_M} - \overline{x}_M + x_1$ extended periodically to $\R^d$.
   	
  	Let $\eta \in C_0^\infty(\square)$ be a  cut-off function satisfying $0\leq\eta\leq 1$, $\eta\vert_{\square^{1/2}} = 1$, and for each $L>0$ define  $\eta_L$ by setting
  	\begin{equation*}
  		\eta_L(x):= \eta\left(\frac{x-\e x^*}{L}\right).
  	\end{equation*} 
We have, in particular, that $\eta_L\vert_{\square^{L}} = 1$, $|\nabla \eta_{L}|\leq C/L$.  
 	 Multiplying $u$ by  $\eta_L$ and normalising the resulting expression, 
 	\begin{equation*}
 	u_L:= \frac{\eta_L u}{\|\eta_L u\|_{L^2(\R^d)}},
 	\end{equation*} 
 	we obtain a standard Weyl sequence for the operator $-\nabla \cdot \hat A_1^{\textrm{\rm hom}}\nabla$ (with respect to letting $L\to\infty$):
 	\begin{equation}\label{70a_6}
 	\|(-\nabla \cdot \hat A_1^{\textrm{\rm hom}}\nabla  - \b_\infty(\l))  u_L\|_{L^2(\R^d)} \leq C/L
 	\end{equation}
 	Note also that 
 	\begin{equation}
 	\label{u_L_6}
 	\begin{aligned}
 	\|u_L\|_{L^\infty(\square^L)} + 	\|\nabla u_L\|_{L^\infty(\square^L)}+\|\nabla^2 u_L\|_{L^\infty(\square^L)}\leq C {L^{-d/2}}.
 	\end{aligned}
 	\end{equation}
 	
 	Let $b$ be the solution to (\ref{66}) and denote by $b^\e$  its $\e$-realisation, $b^\e(x):= b(T_{x/\e}\o)$. We define 	$u_L^\e : = (1+\l b^\e)u_L$. It is not difficult to see that the norm of $u_L^\e$ is bounded from below uniformly in $\e$ and $L$,
 	\begin{equation*}
 	\|u_L^\e\|_{L^2(\R^d)} \geq C>0.
 	\end{equation*}
 	We conclude the construction of the approximate solution by introducing the corrector term
 	\begin{equation*}
 	u_{LC}^\e: = u_L^\e + \e \partial_j u_L N_j(\cdot/\e).
 	\end{equation*}
 We  estimate the corrector term via \eqref{95}and \eqref{u_L_6},
 \begin{equation}\label{136}
 	\|u_{LC}^\e - u_L^\e\|_{L^2(\R^d)} = \e\|\partial_j u_L N_j\|_{L^2(\R^d)} \leq C\e M.
 \end{equation}
 In what follows we will often use the bounds \eqref{95} and \eqref{u_L_6} without mentioning.
 
 \begin{remark}
 	The technical details in what follows build up  to the bound \eqref{81_6}.  The general scheme for this argument is adopted from \cite{KamSm1} and is useful in settings where one needs to establish proximity to the spectrum. We will resort to it once again in the proof of Theorem \ref{th:9.2} below.
 \end{remark}

 	Denote by $a^\e$ the bilinear form associated with the operator $\AA^\e + 1$,
 	\begin{equation*}
 	a^\e(u,v) : = \int\limits_{\R^d} (A^\e \nabla u\cdot  {\nabla v} + u  v).
 	\end{equation*}
 	We substitute $u_{LC}^\e$ into the form with an arbitrary test function $v\in W^{1,2}(\R^d)$,
 	\begin{equation}\label{a_e_6}
 	a^\e(u_{LC}^\e,v)  = \int\limits_{S_1^\e} A_1 \nabla u_{LC}^\e\cdot {\nabla v} +  \int\limits_{S_0^\e} \e^2 \nabla u_{LC}^\e\cdot {\nabla v} + \int\limits_{\R^d} u_{LC}^\e  v.
 	\end{equation}
 To proceed with the argument we   use the  decomposition
 \begin{equation}\label{89c}
 	v = \widetilde v^\e + v_0^\e,
 \end{equation}
 	where $\widetilde v^\e\in W^{1,2}(\R^d)$ is the extension of $v\vert_{S_1^\e}$ into the set of inclusions $S_0^\e$  by Theorem \ref{th:extension}, and  $v_0^\e = v - \widetilde v^\e \in W^{1,2}_0(S_0^\e)$. The next bounds are straightforward consequence of Theorem \ref{th:extension} and the Poincar\'e inequality:
 	\begin{equation}\label{135_6}
 		\begin{gathered}
 			 	\|\nabla\widetilde v^\e\|_{L^2(\R^d)} \leq C \|\nabla v\|_{L^2(S_1^\e)},
 			 	\,\,\,\,
 			 	 \| \nabla v_0^\e\|_{L^2(\R^d)} \leq C \| \nabla v\|_{L^2(\R^d)}, 
 			 	 \,\,\,\,
 			 	  \|v_0^\e\|_{L^2(\R^d)} \leq C\e \|\nabla v_0^\e\|_{L^2(\R^d)},	 	
 		\end{gathered}
 	\end{equation}
 	\begin{equation}\label{a_vv_6}
  \|v_0^\e\|_{L^2(\R^d)} +	\|\widetilde v^\e\|_{L^2(\R^d)} +	\|\nabla\widetilde v^\e\|_{L^2(\R^d)}  +\|\e \nabla v\|_{L^2(\R^d)} \leq C \sqrt{a^\e(v,v)}.
 	\end{equation}
 	We begin by analysing the first term on the right-hand side  of \eqref{a_e_6}:
 	\begin{equation*}
 	\int\limits_{S_1^\e} A_1 \nabla u_{LC}^\e\cdot {\nabla v}  = \int\limits_{\R^d} \big( \hat A_1^{\rm hom} \nabla u_{L} + [\chi_1^\e A_1 (e_j + \nabla N_j) - \hat A_1^\hom e_j] \partial_j u_L + \e \chi_1^\e N_j A_1\nabla\partial_j u_L\big) \cdot {\nabla \widetilde v^\e} 
 	\end{equation*}
 	We first estimate the last term on the right-hand side of the latter, as follows:
 	\begin{equation*}
 	\bigg|\e \int\limits_{\R^d}  \chi_1^\e N_j A_1\nabla\partial_j u_L\cdot {\nabla \widetilde v^\e}\bigg|\leq C\e M \|\nabla \widetilde v^\e\|_{L^2(\R^d)}.
 	\end{equation*}
 Next, we estimate the term containing 
 \begin{equation}\label{89b}
 \chi_1^\e A_1 (e_j + \nabla N_j) - \hat A_1^\hom e_j =  (\chi_1^\e  - \hat \chi_1^\e) A_1 (e_j + \nabla N_j)  + \hat \chi_1^\e A_1 (e_j + \nabla N_j) - \hat A_1^\hom e_j.
 \end{equation}

\begin{lemma} \label{l:5.21} For sufficiently large $M$, $\delta<\rho/2$, and any  $ p \in[1,\infty)$, one has
	\begin{equation}\label{105}
		\begin{aligned}
			\int\limits_{\square_{\e x^*}^L} 	\left|\chi_1^\e  - \hat \chi_1^\e\right|^p =	\int\limits_{\square_{\e x^*}^L} 	\left|\chi_1^\e  - \hat \chi_1^\e\right|\leq C \delta L^d.
		\end{aligned}
	\end{equation}	
\end{lemma}
\begin{proof}
The proof consists in establishing the following straightforward facts for sufficiently small $\delta>0$.
\begin{enumerate}
\item For two bounded open connected $(\rho,\mathcal N,\gamma)$ minimally smooth sets $U_1$ and $U_2$, such that their complements are also connected, the relation $d_{\rm H}(U_1, U_2)\leq \delta$ implies $d_{\rm H}(\partial U_1, \partial U_2)\leq \delta$.
\item The previous statement immediately implies that the symmetric difference $(U_1\setminus U_2)\cup (U_2\setminus U_1)$ is a subset of the $\delta$-neighbourhood of the boundary of either of the sets. By covering the sets $U_1$, $U_2$ with a sufficiently fine lattice, one can easily obtain  the bound $|(U_1\setminus U_2)\cup (U_2\setminus U_1)|\leq C \delta$,  where the constant depends only on the minimal smoothness parameters and the diameter of the sets $U_1$, $U_2$. 
\end{enumerate}
These observations allow us to estimate  the difference between $\chi_1^\e$ and $\hat \chi_1^\e$ in every cube $\square_{\e x_i}^{\e M}$ via \eqref{67a}. Finally, the volume of the ``boundary layers'' $\square_{\e x_i}^{\e (M+\kappa)}\setminus \square_{\e x_i}^{\e M}$ relative to the volume of $\square_{\e x^*}^L$  can be made arbitrarily small by choosing $M$ sufficiently large.
\end{proof}
 
Applying Lemma \ref{l:5.21}, we have 
 \begin{equation}\label{106a}
 	\begin{aligned}
 		\bigg|\int\limits_{\R^d}  (\chi_1^\e  - \hat \chi_1^\e) A_1 e_j \partial_j u_L  \cdot {\nabla \widetilde v^\e}\bigg|\leq C \delta^{1/2}  \|\nabla \widetilde v^\e\|_{L^2(\R^d)}.
 	\end{aligned}
 \end{equation}
For the next estimate we invoke Theorem \ref{th:6.18}.  Applying H\"older's inequality twice and taking into account \eqref{93} and \eqref{105}, we obtain
 \begin{multline}\label{107}
		\bigg|\int\limits_{\R^d}  (\chi_1^\e  - \hat \chi_1^\e) A_1  \nabla N_j \partial_j u_L  \cdot {\nabla \widetilde v^\e}\bigg|
		\\
		\leq \frac{C}{ L^{d/2}} \|\chi_1^\e  - \hat \chi_1^\e\|_{L^{2p/(p-2)}(\square_{\e x^*}^L)} \sum_j\| \nabla N_j\|_{L^p(\square_{\e x^*}^L)}  \|\nabla \widetilde v^\e\|_{L^2(\R^d)} 
		\\
		\leq C  \delta^{(p-2)/2p}   \|\nabla \widetilde v^\e\|_{L^2(\R^d)}.
\end{multline}

 Finally, we deal with the last two terms on the right-hand side of \eqref{89b}. The vector field 
 \begin{equation}\label{96}
 	g_j(y) := \hat \chi_1^\e(\e y) A_1 (e_j + \nabla N_j(y)) - \hat A_1^\hom e_j 
 \end{equation}
 is solenoidal, i.e. $ \int_{\square^{M+\kappa}_{x_1}}  g_j \cdot \nabla\varphi = 0$ for all $\varphi \in W^{1,2}_{\rm per}(\square^{M+\kappa}_{x_1})$, and has zero mean, by the definition of the homogenised matrix $\hat A_1^\hom$. It is well known, see an explicit construction via Fourier series in \cite[Section 1.1]{zhikov2},  that such field can be represented as the ``divergence'' of a skew-symmetric zero-mean field $G^j$, $G^j_{ik} =-G^j_{ki}$, $G^j_{ik} \in W^{1,2}_{\rm per}(\square^{M+\kappa}_{x_1})$:
\begin{equation}\label{106}
	 g_j = \nabla\cdot G^j, 
\end{equation}
which is understood in the sense that $(g_j)_k = \partial_i G^j_{ik}$.
  Note that the fields $G^j$ are commonly referred to in homogenisation theory as flux correctors. It is not difficult to see (by estimating the coefficients of the said Fourier series similarly to the proof of Lemma \ref{l13.3})  that 
  \begin{equation}\label{109}
  	  \|G^j\|_{L^2(\square^{M+\kappa}_{\overline{x}_M})} \leq C M \|g_j\|_{L^2(\square^{M+\kappa}_{\overline{x}_M})} \leq C M^{d/2 +1},
  \end{equation}
where the second inequality  is obtained from \eqref{95}.
\begin{remark}
	We would like to warn the reader against confusing the flux correctors $G^j$, which are periodic, with the flux correctors $G_j^\e$ constructed in Corollary \ref{c12.3} and used in Section \ref{s:relevant}, which are stochastic by  nature.
\end{remark}
 	Applying \eqref{106} (note that $g_j(\cdot/\e) = \e \nabla\cdot G^j(\cdot/\e)$) and integrating by parts yields
 	\begin{equation}\label{99}
 	\begin{aligned}
 	&\int\limits_{\R^d}  \partial_j u_L \, g_j  \cdot {\nabla \widetilde v^\e} =  \int\limits_{\R^d} \e\, \partial_j u_L \,(\nabla\cdot G^j) \cdot {\nabla \widetilde v^\e}  = \int\limits_{\R^d} \e  \big[\partial_j u_L\, G^j: {\nabla^2 \widetilde v^\e}  -   (G^j \nabla \partial_j u_L)\cdot {\nabla \widetilde v^\e}\big],
 	\end{aligned}
 	\end{equation}
 where the colon denotes the Frobenius inner product of matrices. 	The first term in the last integral is identically zero due to the anti-symmetry of $G^j$. Therefore, combining \eqref{89b}, \eqref{106a}, \eqref{107}, \eqref{109} and the last identity, we obtain 
 	\begin{equation}\label{72c}
 		\begin{aligned}
 			 	&\bigg|\int\limits_{\R^d}  [\chi_1^\e A_1 (e_j + \nabla N_j) - \hat A_1^\hom e_j] \partial_j u_L \cdot  {\nabla \widetilde v^\e} \bigg| 
 			 	\leq C \left(\delta^{(p-2)/2p} +   M{\e}\right) \|\nabla \widetilde v^\e\|_{L^2(\R^d)}.
 		\end{aligned}
 	\end{equation}

 	We proceed by considering the second integral in \eqref{a_e_6}. Decomposing both $u_{LC}^\e$ and $v$ and integrating by parts one of the resulting terms, yields
 	\begin{equation}\label{73_6}
 	\begin{aligned}
 	\e^2 \int\limits_{S_0^\e} \nabla u_{LC}^\e\cdot  {\nabla v}
 	&= \e^2 \int\limits_{S_0^\e} \left(\l u_L \nabla b^\e\cdot    {\nabla(v^\e_0 +\widetilde v^\e)} +   \big[\nabla  u_L(1+\l b^\e) + 
 	\e\nabla( \partial_ju_L N_j ) \big]\cdot  {\nabla v}\right)
 	\\
 	&=\e^2  \int\limits_{S_0^\e} \big( - \l u_L \Delta b^\e  {  v^\e_0} -\l \nabla u_L\cdot \nabla b^\e \, {  v^\e_0} 
+ \l u_L \nabla b^\e\cdot  {\nabla \widetilde v^\e}
\\
& \qquad\qquad\qquad\qquad\qquad +  \big[\nabla  u_L(1+\l b^\e) + \e\nabla( \partial_ju_L N_j )\big]\cdot  {\nabla v}\big).
 	\end{aligned}
 	\end{equation}
 	Since $b$ is the solution to (\ref{66}), its $\e$-realisation satisfies $- \e^2 \Delta b^\e  = \l b^\e + 1$,
 	thus we have 
 	\begin{equation}\label{75_6}
 	- \e^2\int\limits_{S_0^\e} \l u_L \Delta b^\e  {  v^\e_0}= \int\limits_{S_0^\e} \l u_L (\l b^\e + 1) {  v^\e_0}.
 	\end{equation}
 	The bound for the remaining terms of \eqref{73_6} via \eqref{95}, \eqref{u_L_6}, \eqref{135_6}, \eqref{a_vv_6} and Lemma \ref{l9.9}  is straightforward:  
 	\begin{multline*}
 	\biggl| \e^2  \int\limits_{S_0^\e} \big(  -\l (\nabla u_L\cdot \nabla b^\e) {  v^\e_0} 
 	+ \l u_L \nabla b^\e\cdot  {\nabla \widetilde v^\e} +  \big[\nabla  u_L(1+\l b^\e) + \e\nabla( \partial_ju_L N_j )\big]\cdot  {\nabla v}\big)\biggr|
 	\\\leq C(\e + \e^2 M)  \sqrt{a^\e(v,v)}.
 	\end{multline*}

 	Combining (\ref{a_e_6})--(\ref{89b}), (\ref{106a}), (\ref{107}), (\ref{72c})--(\ref{75_6}) and the last bound yields
 	\begin{equation}\label{76a_6}
 	\begin{aligned}
 	a^\e(u_{LC}^\e,v)  = \int\limits_{\R^d}  \hat A_1^{\rm hom} \nabla u_{L} \cdot  {\nabla \widetilde v^\e} + \int\limits_{S_0^\e} \l u_L (1 + \l b^\e) {  v^\e_0} + \int\limits_{\R^d} u_{LC}^\e   v + \mathcal R^\e,
 	\end{aligned}
 	\end{equation}
 	where the remainder $\mathcal R^\e$ satisfies 
 	\begin{equation}\label{156_6}
 	\begin{aligned}
 	|\mathcal R^\e| \leq  C(\l) \sqrt{a^\e(v,v)} \big({\e}M  + \delta^{(p-2)/2p}\big).
 	\end{aligned}
 	\end{equation}

 	Replacing $v^\e_0$ with $v-\widetilde v^\e$ in \eqref{76a_6} and recalling the definition of $u_L^\e$, we  rewrite \eqref{76a_6} as follows:
 	\begin{multline}\label{91}
 	a^\e(u_{LC}^\e,v)  =\int\limits_{\R^d} \big( (-\nabla \cdot \hat A_1^{\textrm{\rm hom}}\nabla  - \b_\infty(\l))  u_L   { \widetilde v^\e}+ (\b_\infty(\l) - (\l+\l^2 b^\e) ) u_L  {   \widetilde v^\e}
 	\\
 	 +  \left(\l u_{L}^\e + u_{LC}^\e \right)   v \big)+ \mathcal R^\e.
 	\end{multline}
 	It remains to estimate the second term under the integral sign.
 	
 	Consider the piece-wise averaging operator $\mathcal M^\e_{M+\kappa}: L^2(\e\square_{x^*}^{N(M+\kappa)}) \to L^2(\e\square_{x^*}^{N(M+\kappa)})$ defined by 
 	\[
 \mathcal M^\e_{M+\kappa} f(x) = \fint\limits_{\e\square_{x_i}^{M+\kappa}} f, \quad x\in \e\square_{x_i}^{(M+\kappa)}, \quad i=1,\dots,N^d.
 \]
 The following lemma is proved by a straightforward application of the Poincar\'e inequality.
 \begin{lemma}\label{l6.19}
 Let $f\in W^{1,2}(\e\square_{x^*}^{N(M+\kappa)})$, then 
 	\[
 	\| f-\mathcal M^\e_{M+\kappa} f\|_{L^2(\e\square_{x^*}^{N(M+\kappa)})} \leq C \e M \|\nabla f\|_{L^2(\e\square_{x^*}^{N(M+\kappa)})}.
 	\]
 \end{lemma}
 Now note that by \eqref{eqivan50} on every cube $\e\square_{x_i}^{M+\kappa}$ we have 
 \begin{equation*}
 	\biggl|\, \int\limits_{\e\square_{x_i}^{M+\kappa}} (\b_\infty(\l) - (\l+\l^2 b^\e))\biggr| = \left|\e\square_{x_i}^{M+\kappa}\right| |\b_\infty(\l) - \ell(x_i,M+\kappa,\l,\o)| \leq \left|\e\square_{x_i}^{M+\kappa}\right| \delta.
 \end{equation*}
Therefore,
 \begin{equation*}
 	\begin{aligned}
 			\biggl|\, \int\limits_{\R^d} (\b_\infty(\l) - (\l+\l^2 b^\e))\mathcal M^\e_{M+\kappa}\left(u_L  {   \widetilde v^\e}\right) \biggr| &
 			\\
 			\leq \delta \sum_{i=1}^{N^d}\left|\e\square_{x_i}^{M+\kappa}\right| \left(\mathcal M^\e_{M+\kappa}\left|u_L  {   \widetilde v^\e}\right|\right)&(x_i) 
 		=	\delta \int\limits_{\R^d} \left|u_L  {   \widetilde v^\e}\right| \leq C \delta \|  \widetilde v^\e \|_{L^2(\R^d)}.
 	\end{aligned}
\end{equation*}
Decomposing $u_L  {   \widetilde v^\e} = \left( u_L  {   \widetilde v^\e}-\mathcal M^\e_{M+\kappa} (u_L  {   \widetilde v^\e})\right) + \mathcal M^\e_{M+\kappa} (u_L  {   \widetilde v^\e})$, applying Lemma \ref{l6.19} and observing that $\|\l+\l^2 b^\e \|_{L^2(\square_{\e x^*}^L)} \leq C L^{d/2}$, cf. \eqref{pauza1}, we obtain
\begin{equation}\label{117a}
	\biggl|\int\limits_{\R^d} (\b_\infty(\l) - (\l+\l^2 b^\e)) u_L  {   \widetilde v^\e}\biggr| \leq C ({\e M} + \delta)\|  \widetilde v^\e \|_{L^2(\R^d)}  + C {\e M}\|\nabla  \widetilde v^\e \|_{L^2(\R^d)} .
\end{equation}

Combining   \eqref{70a_6}, \eqref{136}, \eqref{a_vv_6}, \eqref{156_6}--\eqref{117a} yields
 	\begin{equation}\label{81_6}
 	\begin{aligned}
 	|a^\e(u_{LC}^\e ,v) - (\l+1) (u_{LC}^\e ,v)|   
 	\leq  \widehat{\mathcal R}(\e,L,\l) \sqrt{a^\e(v,v)},
 	\end{aligned}
 	\end{equation}
 	where  
 	\begin{equation*}
 	\begin{aligned}
 	\widehat{\mathcal R}(\e,L,\l) :=   C(\l) \big({\e}M  +\delta^{(p-2)/2p} + L^{-1}\big).
 	\end{aligned}
 	\end{equation*}

Now, resorting to Lemma \ref{l:E.1} we see that for sufficiently small $\delta$ and large $L$ one has
 \begin{equation*}
 	\limsup_{\e\to 0} \dist(\l, \Sp(\AA^\e)) \leq C(\l)  \big(\delta^{(p-2)/2p} + L^{-1}\big),
 \end{equation*}
Since $\delta$ and $L$ are arbitrary, we conclude that 
  \begin{equation*}
 	\lim_{\e\to 0} \dist(\l, \Sp(\AA^\e)) = 0.
 \end{equation*}	
 
\begin{remark}
	In the present,  in order to simplify the exposition of the main ideas, work we assume that the coefficients of $\AA^\e$ are constant in the stiff component and are a multiple of the identity matrix in the inclusions. The coefficients  $A^\e(\cdot,\omega)$ can be described as follows. Consider a sequence of random variables of the form 
	\begin{equation*}
		A^{\varepsilon} (\omega)= A_1 {\mathbf 1}_{\Omega \backslash \mathcal{O}}+\varepsilon^2 I {\mathbf 1}_{\mathcal{O}},  
	\end{equation*} 
	where $A_1$ is a positive definite matrix. Then the coefficients defined by \eqref{defAe} are simply the $\e$-realisation of $A^{\varepsilon} (\omega)$:
	$$ A^{\varepsilon} (x,\omega)= A^{\varepsilon} (T_{x/\e} \omega). $$ 
	(Note that the operator $\mathcal{A}^{\eps}$ may be formally considered as the $\varepsilon$-realization of the operator $-\nabla_{\omega}\cdot  A^\e(\o) \nabla_{\omega}$.) Similarly, we could analyse a more general  problem where the coefficients in the constitutive parts of the composite are not assumed to be constant, namely, by setting
	$$A^{\varepsilon} (\omega)= A_1(\omega) {\mathbf 1}_{\Omega \backslash \mathcal{O}}+\varepsilon^2 A_0(\omega) {\mathbf 1}_{\mathcal{O}},  $$
	where $A_0, A_1 \in L^{\infty}(\Omega;\mathbf{R}^{d \times d} )$ are uniformly coercive. (One can look at an interesting simpler case where one does not change the geometry of the inclusions, but only the coefficients on the soft component.) We believe that our analysis can be easily adapted to this more general setting. It will manifest, in particular, in the need to analyse operator $-\nabla_{\o} \cdot  A_0 \nabla_{\o}$ with the domain being a suitable subspace of $W^{1,2}_0(\OO)$ instead of the operator $-\Delta_{\mathcal{O}}$. In this case one will get  analogous results; in particular, the spectrum of $\AA^\hom$ will be characterised by a suitable version of the $\beta$-function, and the limit set $\mathcal G$ by a suitable version of $\beta_{\infty}(\l)$ (cf. \eqref{d:beta1} and \eqref{38}). Moreover, the notions introduced and the results presented in Section \ref{s:relevant} can also be easily adapted to this more general setting.
\end{remark}

\subsection{Examples} \label{cubesfilled}

In general, constructing explicitly a probability space that would provide a ``truly random'' distribution of inclusions is a challenging task. We consider two general setups: inclusions randomly placed at the nodes of a periodic lattice and the random parking model. We begin with the former.

\subsubsection{Probability space setup} \label{s:5.6.1}

Let $\{\widetilde{\omega}_j\}_{j\in \mathbf{Z}^d}$ be a sequence of  independent and identically distributed random vectors taking values in $\mathbf{N}_0^l \times [r_1, r_2],$ where $\mathbf{N}_0^l := \{0,1,\ldots, l\}$ and $0<r_1\le r_2\le1$. Let $(\widetilde{\Omega},\widetilde{\mathcal{F}},\widetilde{P} )$ be the canonical probability space associated with $\{\widetilde{\omega}_j\}_{j\in \mathbf{Z}^d}$ (obtained by the Kolmogorov construction). 
Let $Y_k\subset [0,1)^d,$ $k \in \mathbf{N}_0^l,$ be open, connected sets satisfying minimal smoothness assumption and not touching the boundary of $[0,1)^d$. These sets model the shapes of the inclusions, namely, for every $j\in \Z^d$ (which determines the location of the inclusion) the first component of $\widetilde \o_j = (k_j, r_j)$ describes the shape of the inclusion and the second component describes its size. We also set $Y_0= \emptyset$. 
On $\widetilde{\Omega}$, there is a natural shift $\widetilde{T}_z (\widetilde{\omega}_j)=(\widetilde{\omega}_{j-z})$, which is ergodic.

We treat $[0,1)^d$ as a probability space with Lebesgue measure $dy$ and the standard algebra $\mathcal{L}$ of Lebesgue measurable sets, and define 
$$
\Omega=\widetilde{\Omega} \times [0,1)^d,\quad \mathcal{F}= \widetilde{\mathcal{F}} \times \mathcal{L}, \quad P= \widetilde{P} \times dy. 
$$
On $\Omega$ we introduce a dynamical system $T_x (\widetilde \omega, y)=(\widetilde T_{[x+y]} \widetilde \omega, x+y-[x+y])$, and define  
$$\mathcal{O}:=\{(\widetilde \omega, y):\, y\in r_0 Y_{k_0}  \}.$$
It is easy to see that $\mathcal{O}$ is measurable. For a fixed $\o = (\widetilde{\o},y)$ the realisation $\mathcal O_\o$ consists of the inclusions $r_j Y_{k_j} + j -y,\,j\in\Z^d$.

Next we consider three special cases of this general setup.

\subsubsection{One shape randomly placed at a periodic lattice nodes} 
\label{simple_example}

In this example we set $l=1,$ $r_1=r_2=1$. The second component of $\widetilde{\o}$ is redundant in this example, so we disregard it in the notation. The value $0$ or $1$ of  $\widetilde\o_z,$ $z\in\mathbf{Z}^d,$ corresponds to the absence or the presence of the inclusion at the lattice node $z$, respectively. We have
$$
\mathcal{O} =\bigl\{\o=(\widetilde{\o},y)\,:\, \widetilde \o_0 = 1,\, y\in Y_1\bigr\}\subseteq \O,
$$ 
so $\Sp(-\Delta_{\OO})=\Sp(-\Delta_{Y_1})$. For a given  $\omega=(\widetilde{\omega}, y)\in\Omega$ the realisation $\mathcal{O}_\o = \{ x\,:\, T_x\o \in \mathcal{O}\}$ is the union of the sets $Y_1+z-y$ for all $z\in\mathbf{Z}^d$ such that $\widetilde \o_z=1$.  
By the law of large numbers, for a.e. $\o$ and there exist arbitrary large cubes contain no inclusions and arbitrary large cubes containing an inclusion at every lattice node.  
Thus, we have  
\begin{equation}\label{razgovor20}  
	\beta_{\infty} (\lambda) \geq \max\{\lambda, \beta_{1\per,Y_1}(\lambda)\},
\end{equation} 
where
$$ \beta_{t\per,Y} (\lambda) = \lambda+ \frac{\lambda^2}{t^d} \sum_{j=1}^\infty
\frac{\langle \varphi_j \rangle^2}{\nu_j-\lambda}$$
is  Zhikov's $\b$-function corresponding to the  $t$-periodic distribution of an inclusion $Y$. Here  $\nu_j$ and $\varphi_j$ denote the eigenvalues (repeated according to their multiplicity) and orthonormalised eigenfunctions of $-\Delta_{Y}$ (extended by zero outside of $Y$), and we put for short $\langle f \rangle := \int_{\R^d} f$ until the end of this section, keeping in mind that the eigenfunctions $\varphi_j$ are extended by zero outside the inclusion.
(In fact, it is not difficult to see that the equality holds in \eqref{razgovor20}. It follows from an observation that on an arbitrary sequence converging to $\b_\infty(\l)$, the limit $\lim_{M \to \infty} \ell(\overline{x}_M,M,\lambda,\omega)$ is a convex combination of $\lambda$ and $\beta_{1\per,Y_1}(\lambda)$.)
Thus we have $\mathcal{G}=\mathbf{R}^+_0$.  
\begin{remark} 
	It was shown in \cite{ChChV} that 
	\begin{equation*}
		\beta(\lambda)= \lambda+\lambda^2P\bigl(\{\widetilde\o:\,\widetilde\o_0=1\}\bigr) \sum_{j=1}^\infty
		\frac{\langle \varphi_j \rangle^2}{\nu_j-\lambda}.
	\end{equation*}	
Note also that the eigenvalues $\nu_j$ such that    $\langle \varphi_j\rangle=0$  may be excluded from the above formulae.
\end{remark}

\subsubsection{Finite number of shapes at the lattice nodes}
Now we take finite $l>1$, $r_1=r_2=1$ (again, second component of $\widetilde{\o}$ being redundant) and assume that $P\{\widetilde{\o}_0 = (0,1)\} = 0$, which is equivalent to saying that $\{\widetilde{\omega}_j\}_{j\in \mathbf{Z}^d}$ take  values in $\mathbf{N}^l$.    By the law of large numbers, for each $k\in\N^l$ there exist  arbitrarily large cubes containing only the shapes $Y_k$ at the lattice nodes. 
We have  
$$ \Sp(-\Delta_{\OO})=\cup_{k=1}^l \Sp(-\Delta_{Y_k}), \  \beta_{\infty} (\lambda)=\max_{k=1,\dots,l} \beta_{1\per,Y_k}(\lambda). $$
In particular,
$$ \mathcal{G}=\bigcup_{k=1}^l \Sp(-\Delta_{Y_k})\, \cup\, \bigcup_{k=1}^l \{\lambda\geq 0:\beta_{1\per,Y_k}(\lambda)\geq 0 \}.$$

\begin{remark} 
	It was shown in \cite{ChChV} that	
	\begin{equation*}
		\beta(\lambda)= \lambda+\lambda^2\sum_{k=1}^l P\bigl(\{\widetilde\o:\,\widetilde\o_0=k\}\bigr)  \sum_{j=1}^\infty
		\frac{\langle \varphi_j^k \rangle^2}{\nu_j^k-\lambda}.
	\end{equation*}
	Here, $\nu_j^k$ and $\varphi_j^k$ denote the eigenvalues and orthonormalised eigenfunctions of $-\Delta_{Y_k}$. 
\end{remark}	

\subsubsection{Randomly scaled inclusions}\label{sss:randomscaled}
Take $l=1$ and $0<r_1<r_2\leq 1$ and assume that $P\{\widetilde{\o}_0 \in 0\times[r_1,r_2]\} = 0$. Note that the first component of $\widetilde{\omega}_0$ is redundant and so we drop it from the notation.  We denote  by $\mathcal S \subset [r_1,r_2]$ the support of the random variable $\widetilde{\omega}_0$.  Thus, for every value of the scaling parameter $r \in \mathcal S$ and every   $\delta>0$, the set $\{\widetilde \o_0 \in [r-\delta,r+\delta]\}$ has positive probability. 
It follows that for any  $\delta>0$  there exist arbitrarily large cubes containing only inclusions whose  scaling parameter belongs to the interval $[ r-\delta,r+\delta]$.  
Therefore, 
$$ \Sp(-\Delta_\OO)=\bigcup_{r \in \mathcal S} \bigcup_{j \in \mathbf{N}}\{r^{-2} \nu_j\},\qquad \beta_{\infty}(\lambda)=\max_{r\in \mathcal S} \beta_{1\per,rY_1} (\lambda), $$
where   $\nu_j$ are the eigenvalues of  $-\Delta_{Y_1}$. It follows that
$$ \mathcal{G}= \bigcup_{r \in \mathcal S} \bigcup_{j \in \mathbf{N}}\{r^{-2} \nu_j\}\, \cup  \{\lambda>0: \max_{r\in \mathcal S}\beta_{1\per,rY_1}(\lambda)\geq 0 \}.$$

\begin{remark} 
	It was shown in \cite{ChChV} that 
	\begin{equation}\label{betascaling}
		\beta(\lambda)= \,\lambda\,+\, \lambda^2 \int\limits_{\{\widetilde{\o}_0\in[r_1,r_2]\}} \sum_{j=1}^\infty
		\frac{\langle \varphi_{j, \widetilde{\o}_0}\rangle^2}{\nu_{j,\widetilde{\o}_0}-\lambda}.
	\end{equation}
	Here $\nu_{j,r}$  and  $\varphi_{j,r}$ are the eigenvalues and orthonormalised eigenfunctions of $-\Delta_{rY_1}$. They can be obtained from those of $-\Delta_{Y_1}$  by scaling; in particular, $\nu_{j,r} = r^{-2} \nu_j$.
\end{remark}

\subsubsection{Modifications of the above examples}
It is not difficult to construct  explicit modifications of the above examples in which the inclusions are randomly rotated and shifted within the lattice cells (still respecting the minimal distance between them in order for the extension property to hold). As can be seen from the definition of $\beta_\infty(\l)$ and  the formula \eqref{sasha103} for $\b(\l)$, this additional degree of freedom does not affect these functions (but only $A_1^\hom$), hence both the limit set $\mathcal{G}$ and $\Sp(\mathcal A^\hom)$ (in the whole space setting) remain the same. 

Another possible modification is to assume that random variables $\{\widetilde{\omega}_j\}_{j\in \mathbf{Z}^d}$ have finite correlation distance. In this case, in order to be able to recover the above formulae for $\beta_\infty(\l)$, arbitrarily large cubes containing only the inclusions of a specific type (as in the above examples)  should have positive probability. If this is not the case, then only the general formula \eqref{38} is available.

\subsubsection{Random parking model} \label{randomparking} 
In this example we consider the random parking model described in  \cite{penrose} (see also \cite{gloria}). The intuitive description of the model is as follows: copies of a set $V$ arrive sequentially at random without overlapping until jamming occurs. 

We start by briefly recalling the graph construction of the random parking measure (RPM) in $\R^d$. Let $\mathcal{P}$ be a homogeneous Poisson process with intensity one in $\R^d \times \mathbf{R}^{+}$, and let $V \subset \R^d$ be an open bounded set that contains the origin. 
An oriented graph is a special kind of directed graph in which there is no pair of vertices $\{x,y\}$ for which both $(x,y)$ and $(y,x)$ are included as directed edges. We shall say that $x$ is a parent of $y$ and $y$ is an offspring of $x$ if there is an oriented edge from $x$ to $y$. By a root of an oriented graph we mean a vertex with no parent.
The construction of RPM goes as follows.  Make the points of the Poisson process $\mathcal{P}$ on $\R^d \times  \mathbf{R}^+$ into the vertices of an infinite oriented graph, denoted by $\Gamma$, by putting in an oriented edge $(X,T) \to (X',T')$ whenever $(X' +V)\cap(X +V) \neq \emptyset$ and $T < T'$. For completeness, we also put an edge $(X,T) \to (X',T')$ whenever $(X' + V) \cap (X + V) \neq \emptyset$, $T = T'$, and $X$ precedes $X'$ in the lexicographical order, although in practice the probability that $\mathcal{P}$ generates such an edge is zero. 
For $(X,T) \in \mathcal{P}$, let $C(X,T)$ (the ``cluster at $(X,T)$'') be the (random) set of ancestors of $(X,T)$, that is, the set of $(Y,U) \in \mathcal{P}$  such that there is an oriented path in $\Gamma$ from $(Y,U)$ to $(X,T)$. As shown in \cite[Corollary 3.1]{penrose}, the ``cluster'' $C(X,T)$ is finite for $(X,T) \in \mathcal{P}$ with probability one.  Recursively define subsets $F_i$, $G_i$, $H_i$ of $\Gamma$ as follows. Let $F_1$ be the set of roots of the oriented graph $\Gamma$, which is non-empty due to the finiteness of clusters, and let $G_1$ be the set of offspring of the roots. Set $H_1 = F_1 \cup G_1$. For the next step, remove the set $H_1$ from the vertex set, and define $F_2$ and $G_2$ in the same way:  $F_2$ is the set of roots of the restriction of $\Gamma$ to the vertices in $\mathcal{P} \backslash H_1$, and $G_2$ is the set of vertices in $\mathcal{P}\backslash H_1$ that are offspring of those in $F_2$. Set $H_2 =F_2 \cup G_2$, remove the set $H_2$ from $ \mathcal{P} \backslash H_1$, and repeat the process to obtain $F_3, G_3, H_3$. Continuing ad infimum gives us subsets $F_i, G_i$ of $\mathcal{P}$ defined for $i = 1,2,3,\dots$. These sets are disjoint by construction.  As proved in \cite[Lemma 3.2]{penrose}, the sets $F_1$, $G_1$, $F_2$, $G_2 \dots$ form a partition of $\mathcal{P}$ with probability one. Clearly, the points of $F_1\vert_{\R^d}$ (the projection of $F_1$ onto $\R^d$) are ``well separated'' in the sense that for every pair $X,X'\in F_1\vert_{\R^d}$ one has  $(X' +V)\cap(X +V) = \emptyset$. Removing the set $G_1$ we ensure that we discard all points that do not satisfy the said separation condition with respect to the points of $F_1\vert_{\R^d}$. The points of  $F_2\vert_{\R^d}$ are well separated from each other and from the points of $F_1\vert_{\R^d}$, and so forth.

\begin{definition}  The random parking measure in $\R^d$ is given by the counting measure $N(A)$, $A\subset \R^d$, generated by the projection of the union $\cup_{i=1}^{\infty} F_i$  onto $\R^d$. 
\end{definition} 	

It is proven in, e.g., \cite{penrose,gloria} that there exists a probability space  supporting RPM and that RPM is an ergodic process with respect to translations. Furthermore,  there exists $\kappa = \kappa(d,V)$ such that 
\begin{equation}\label{112a}
	\lim_{M\to\infty} \frac{N(\square^M)}{|\square^M|} = \kappa \mbox{ almost surely}.
\end{equation}

In our example we set $V=\square$, so we have a collection of randomly parked non-overlapping cubes $\square_{X_j}^1, j\in \N, \cup_j X_j = \cup_i F_i\vert_{\R^d}$, such that no more unit cubes can be fitted without overlapping. At each cube we place an inclusion $X_j + Y_1$, where $Y_1 \subset \square$ is a reference inclusion observing a positive distance from the boundary of the cube. Note that $\Sp(-\Delta_\OO)=\Sp(-\Delta_{Y_1}).$
From \eqref{sasha103} and \eqref{112a} we easily infer that
	\begin{equation*}
	\beta(\lambda)= \lambda+\lambda^2\, \kappa \sum_{j=1}^\infty \frac{\langle \varphi_j \rangle^2}{\nu_j-\lambda}.
\end{equation*}

In order to derive a formula for $\beta_\infty$, we  analyse  the areas with the smallest and the greatest density of the distribution of inclusions (hence the choice of the set $V$, for which this can be done explicitly). For an arbitrary (small) $\delta>0$, consider the periodic lattice $(2-4\delta)\Z^d$, denoting its points by $\xi_i, i\in \N$, and consider the balls of radius $\delta$ centred at $\xi_i,  i\in \N$. For an arbitrary $T>0$ there is a positive probability of  $\mathcal P$ having exactly one point $(X_i,T_i)$ in each of the sets $B_\delta(\xi_i)\times(0,T)\subset \square^{M+4}\times(0,T)$ and no other points in the set $\square^{M+4}\times(0,T)$. Then, by the above construction, the random parking measure has exactly one point $X_i$ inside each ball $B_\delta(\xi_i)$ with $\xi_i \in \square^{M}$ and no other points inside $\square^{M}$, as no more unit cubes could fit between the cubes $\square_{X_i}^1$. By the law of large numbers, with probability one there exist arbitrarily large cubes $\square_x^{M}$ containing only almost periodically positioned inclusions $Y_1+ X_i$, $X_i \in B_\delta(\xi_i)$, $\xi_i\in(2-4\delta)\Z^d$. Since $M$ and $\delta$ are arbitrary, we conclude that 
\begin{equation}\label{114}
	 \beta_{\infty}(\lambda)\geq \beta_{2\per,Y_1} (\lambda).	
\end{equation}

Making an analogous construction  with the periodic lattice $(1+4\delta)\Z^d$, $\delta>0$, we can see that 
\begin{equation}\label{115a}
	\beta_{\infty}(\lambda)\geq \beta_{1\per,Y_1} (\lambda).
\end{equation}

On the other hand, since $1$-periodic and $2$-periodic distribution  provide the greatest and the smallest possible inclusion density, it follows that equalities in \eqref{114} and \eqref{115a} are attained when $\int_{Y_1} b(\cdot,\lambda) <0$ and $\int_{Y_1} b(\cdot,\lambda) \geq 0$, respectively. Thus  
$$
\beta_{\infty}(\lambda) = \max\{ \beta_{1\per,Y_1} (\lambda), \beta_{2\per,Y_1} (\lambda)\}.
$$

\begin{remark}
	The random parking model does not satisfy  Assumption \ref{asumivan1} (concerning finite range of dependence), and therefore we cannot apply   Theorem \ref{th4.4} to it straight away. Nevertheless, the preceding argument shows that one can ``recover'' the values of $\beta_\infty(\l)$ on arbitrarily large non-typical spacial regions (with almost periodic positioning of inclusions), and then follow the  argument of  the proof of Theorem \ref{th4.4} in order to show that 	$\lim \Sp(\AA^\e) \supset \mathcal{G}$.
\end{remark}

\begin{remark}
	As in the previous examples, one can elaborate the random parking model example further by allowing inclusions of different shapes, sizes, randomly rotated, through the use of the  marked point processes framework, see \cite{heid1}.
\end{remark}

\begin{remark} 
	The random parking model is a more realistic (from the  point of view of applications) model of random distribution of inclusions. On one hand, it does not allow arbitrarily large inclusion-free regions.  (More precisely, there exists a radius $r$ such that any ball $B_r(x), x\in \R^d$, contains at least one inclusion.) On the other hand, the inclusions are not too close to each other, thus satisfying Assumption \ref{kirill100}.  If one  only uses the Poisson point process in the construction, then one  ends up with arbitrarily large inclusion-free regions and at the same time having arbitrarily many overlapping inclusions. 
	
	The problem of overlapping inclusions can also be  dealt with by using a Mat\'ern modification of Poisson process (see e.g. \cite{heid1}) which is constructed
	from a given point process  by simultaneously erasing all points with the distance to the nearest
	neighbour smaller than a given constant. It can be shown  that if the original process is  stationary (ergodic),
	the resulting hardcore process is stationary (ergodic).
	If one starts from a Poisson process and makes  a Mat\'ern modification, this  gives a point process which can be further used to construct random inclusions satisfying  Assumption \ref{kirill100}. However, the limit spectrum in this case would coincide with  $\R^+_0$, due to the existence of arbitrarily large areas inclusion-free regions. (More precisely, in this case  $\beta_{\infty}(\lambda) \geq \lambda$ for all $\lambda\geq 0$.) 
\end{remark} 

\subsubsection{Periodic inclusions randomly shifted and rotated within their cells}
Let $Y_1\subset [0,1)^d$ be a reference inclusion,  as described above. In each cell of the lattice $\Z^d$ we put a randomly rotated and randomly positioned (within the cell) copy of $Y_1$. We do this in such a way that the minimal distance between the inclusions is positive. It is not difficult to describe the corresponding probability space explicitly similarly to how it was done in Section  \ref{s:5.6.1}. Clearly, in this case $\b = \b_\infty,$ and thus $\mathcal G = \lim \Sp(\AA^\e) = \Sp (\AA^\hom)$.

 \section{Relevant and irrelevant limiting spectrum}\label{s:relevant}

In this section we characterise the limiting spectrum of $\AA^\e$  in the gaps of $\Sp(\AA^\hom)$. As we have seen in the previous section, in general the limit set $\mathcal G$ of $\Sp(\AA^\e)$ is larger than $\Sp(\AA^\hom)$. On the other hand, in the bounded domain setting we have the Hausdorff convergence of the spectra, i.e. the gaps of $\Sp(\AA^\hom)$ are free from the spectrum of $\AA^\e$ in the limit. These facts indicate that the elements of the Weyl sequences for $\AA^\e$ corresponding to $\l \notin \Sp(\AA^\hom)$ are supported further and further away from the origin (as $\e\to 0$) and could not be contained in a bounded domain. Another observation that can be drawn from Section \ref{s:conv of spectrum} is that the elements of such Weyl sequences are concentrated in the areas of space with non-typical distribution of inclusions. Both observations may be interpreted in the following way: in  applications, where one deals with finite distances (finite-size medium),  manufactured composites with relatively uniform distribution of inclusions, or only the bulk waves (characterised by $\beta(\lambda)$ in the homogenisation limit) matter, the aforementioned part of the spectrum may not be relevant.
In this lies our motivation to distinguish the two  types of the spectrum and refer to them as  relevant and irrelevant. However, it does not seem feasible to precisely demarcate the border between these two parts for finite $\e$: indeed, the notion of Weyl sequences ``supported further and further away from the origin'' is not easily quantifiable. Instead, we provide such characterisation in the limit as $\e\to 0$. In Definition \ref{d:srel} below we introduce the notions of {\it relevant and irrelevant limiting spectra}. Remarkably, but not surprisingly given the above discussion, ${\mathcal R}\mbox{-}\lim \Sp(\AA^\e)$  coincides  with $\Sp(\AA^\hom)$. This constitutes the main result of this section, formulated in Theorem \ref{th:9.2}.
 A number of  auxiliary constructions and results required in the proof of the theorem \ref{th:9.2} are presented in  Appendix  \ref{ss:9.1}.

Recall that the notation $E^\e_{(-\infty,\l]}$ is used for the spectral projections of the operators $\mathcal{A}^{\varepsilon}$. 
Obviously, one has
$$ \lambda \in \textrm{Sp} (\mathcal{A}^{\varepsilon}) \Leftrightarrow \forall \delta>0, E^\e_{[\lambda-\delta,\lambda+\delta]} \neq 0$$ 
and
\begin{equation*} 
	\lambda \in \lim_{\varepsilon\to 0} \textrm{Sp} (\mathcal{A}^{\varepsilon}) \Leftrightarrow \forall \delta>0 \ \exists \varepsilon_0>0\ \forall \varepsilon< \varepsilon_0 \ 
	E^\e_{[\lambda-\delta,\lambda+\delta]} \neq 0. 
\end{equation*} 

By $P_{\square^L}$ we denote the operator that to each $u \in L^2(\R^d)$ assigns its restriction  to the cube $\square^L$ (extended by zero outside of $\square^L$ whenever necessary).

\begin{definition}\label{d:srel}
We say that $\l$ belongs to the {\bf relevant limiting spectrum} of $\AA^\e$ if
\begin{equation}\label{111}
	\limsup_{\delta \to 0} \limsup_{L\to \infty} \limsup_{\varepsilon \to 0} \|P_{\square^L} E^\e_{[\lambda-\delta,\lambda+\delta]}\|_{L^2 \to L^2}>0
\end{equation}
 and denote the set of such values of $\l$ by
$${\mathcal R}\mbox{-}\lim \Sp(\AA^\e).$$

	We define the {\bf irrelevant limiting spectrum} as 
$$ {\mathcal Irr}\mbox{-}\lim \textrm{Sp}(\mathcal{A}^\varepsilon): = \lim \Sp(\AA^\e) \setminus {\mathcal R}\mbox{-}\lim \Sp(\AA^\e).$$
\end{definition}

\begin{remark}\label{r6.2}${}$	
	\begin{enumerate}
		\item The definition of the relevant/irrelevant limiting spectra can be generalised and applied in other settings, not necessarily stochastic.
		\item In the stochastic case the  relevant and irrelevant limiting spectra are deterministic almost surely due to translation invariance. 
		\item Clearly, ${\mathcal R}\mbox{-}\lim \Sp(\AA^\e)\subset \lim \Sp(\AA^\e)$.
		\item Note that the sequence $\|P_{\square^L} E^\e_{[\lambda-\delta,\lambda+\delta]}\|_{L^2 \to L^2}>0$ is increasing in $L$ and $\delta$. Therefore, we can replace $\limsup_{\delta \to 0}$ and  $\limsup_{L\to \infty}$  in \eqref{111} with $\lim_{\delta \to 0}$ and  $\lim_{L\to \infty}$, respectively. From this it is easy to see that the relevant limiting spectrum ${\mathcal R}\mbox{-}\lim \Sp(\AA^\e)$ is a closed set.
	\end{enumerate}
\end{remark}

\begin{remark}\label{r6.3}
It is not difficult to see that $\l\in {\mathcal R}\mbox{-}\lim \Sp(\AA^\e)$ if and only if	there exists a family of functions  $\psi_L^\varepsilon\in{\rm Dom}({\mathcal A}^\varepsilon)$, $\|  \psi^\e_L\|_{L^2(\R^d)}=1$, $L\in \N$, $\e>0$,   satisfying
	\begin{equation}\label{154}
		\begin{aligned}
			\lim_{L \to \infty} \limsup_{\e  \to 0} {\|(\AA^\e-\l)  \psi^\e_L\|_{L^2(\R^d)}}=0,
		\end{aligned}
	\end{equation}
	and 
	\begin{equation}\label{155}
		\begin{aligned}
			\limsup_{L \to \infty} \liminf_{\e \to 0}{\| \psi^\e_L\|_{L^2(\square^{L})}}> 0.
		\end{aligned}
	\end{equation}
	We will use this observation as an equivalent definition of  relevant limiting spectrum in order to prove the  main result of the section.
\end{remark}

We next discuss the relation of the relevant and irrelevant limiting spectra to the spectrum of the limit operator.  Consider a family of self-adjoint operators $T^\e$ converging to $T$ in the sense of resolvent convergence in variable spaces (such as in Theorem \ref{misha10}), cf. \cite{Past}. Assume that Definition \ref{d:srel} makes sense for $T^\e$, i.e. the projection $P_{\square^L}$ is well defined. One may naturally expect that the inclusion $	{\mathcal R}\mbox{-}\lim \Sp(T^\e) \supset \Sp(T)$ takes place. Indeed,  let $u$, $\|u\|=1$,  be an approximate eigenfunction of the limit operator, i.e. $\|(T-\l)u\|\leq \delta$, where $\delta$ can be made arbitrary small. Then if we can guarantee the existence of a sequence of approximate eigenfunctions  $u^\e$ of  $T^\e$ (for the same $\l$) converging to $u$, then we automatically get that $\l$ is in the relevant limiting spectrum of $T^\e$.  The converse statement, i.e. that ${\mathcal R}\mbox{-}\lim \Sp(T^\e) \subset \Sp(T)$ is not always true, and requires some additional compactness properties of the family $T^\e$, cf. Proposition \ref{propperhc} and Remark \ref{r:6.6} below. In the next subsection we show that in our setting the	 relevant limiting spectrum   coincides with the spectrum of the limit operator $\AA^\hom$.

In the first four examples we provided in Section \ref{cubesfilled},   there exists a non-empty set of $\l$ such that  $\b_\infty(\l) \geq 0 > \beta(\l)$, implying that the limit spectrum is strictly larger than the spectrum of $\AA^\hom$. Together with Theorem \ref{th:9.2} below, this implies that the irrelevant limiting spectrum is non-empty. This (rather unsurprising) observation is in  stark contrast with the periodic case, as the following assertion suggests.

\begin{proposition} \label{propperhc} 
	Let $Y_0 \subset \square$ be a reference soft inclusion and $Y_1:=\square \backslash Y_0$. Denote by $Y_0^{\#}$ and $Y_1^{\#}$ the sets obtained by  extending  $Y_0$ and $Y_1$ periodically to the whole of $\R^d$. Consider the self-adjoint operator in $L^2(\R^d)$  defined by the differential expression
	$$ \mathcal{A}^{\eps}_{\#}= -\nabla \cdot A_1 {\mathbf 1}_{\eps Y_1^{\#}}  \nabla-\varepsilon^2 \nabla \cdot A_0 {\mathbf 1}_{\eps Y_0^{\#}}  \nabla, 	$$
	where  $A_0, A_1 \in L^{\infty}(\R^d;\mathbf{R}^{d \times d} )$ are periodic uniformly positive matrix valued functions.
	 Then   
	\begin{equation*} 
		\mathcal R\mbox{-}\lim  \Sp (\mathcal{A}^{\varepsilon}_{\#})=\lim  \Sp (\mathcal{A}^{\varepsilon}_{\#}). 
		\end{equation*} 
\end{proposition} 
\begin{proof} To prove the result, one can follow the argument of Theorem \ref{th4.1}, leading to the formula \eqref{47b}, and then use the obtained family of approximate eigenfunctions to construct a sequence as in Remark \ref{r6.3}. We, however, will use a more direct approach  utilising directly the periodicity. Consider a sequence $\l^\e \in \Sp(\AA^\e)$ converging to some $\l$. As is well know from the Floquet-Bloch  theory, for every $\l^\e$ there exists a quasi-periodic function $u^\e \in H^1_{\rm loc}(\R^d)$ solving the equation 
	\begin{equation}\label{115c}
		\AA^\e_{\#} u^\e - \l^\e u^\e = 0
	\end{equation}
in the distributional sense. Let us fix some (large) $L>0$. We normalise $u^\e$ by requiring that 
\begin{equation}\label{116a}
	\|u^\e\|_{L^2(\square^L)} = 1.
\end{equation}
  One also has  (again, by the Floquet-Bloch theory) the energy bound
\begin{equation}\label{117}
	\left\|\left(A_1 {\mathbf 1}_{\eps Y_1^{\#}}  + \e^2 A_0 {\mathbf 1}_{\eps Y_0^{\#}}\right)^{1/2} \nabla u^\e\right \|_{L^2(\square^L)} \leq C,
\end{equation}                           
where the constant depends only on $\l$.
	
Let us test the equation \eqref{115c} with $\eta_L v$, where $v\in H^1(\R^d)$ and $\eta_L$ is a standard cut-off functions satisfying  $\eta_L \in C_0^\infty(\square^{2L})$, $0\leq\eta_L\leq 1$, $\eta_L\vert_{\square^{L}} = 1$, $|\nabla \eta_{L}|\leq C/L$. Denoting by $a^\e_{\#}(\cdot,\cdot)$ the bilinear associated with the operator $\AA^\e_{\#} + 1$, we have 
\begin{multline*}
0 =	a^\e_{\#}(u^\e,\eta_{L}  v) -(\l^\e+1) (u^\e,\eta_{L}  v) = 	a^\e_{\#}(\eta_{L}  u^\e, v) -(\l^\e+1) (\eta_{L}  u^\e, v) 
	\\
	+ \int_{\square^{2L}} \left(A_1 {\mathbf 1}_{\eps Y_1^{\#}}  +  \e^2  A_0 {\mathbf 1}_{\eps Y_0^{\#}}\right) (v \nabla u^\e - u^\e\nabla v) \cdot \nabla \eta_{L} .
\end{multline*}
Applying  \eqref{116a} and \eqref{117} (taking into account quasi-periodicity of $u^\e$) to the last term in the above we arrive at
\begin{equation*}
|		a^\e_{\#}(\eta_{L}  u^\e, v) -(\l^\e+1) (\eta_{L}  u^\e, v) |\leq \frac{C}{L} \sqrt{ a^\e_{\#}(v, v)}.
\end{equation*}
Applying Lemma \ref{l:E.1}, we infer the existence of a family of functions $\psi^\e_L$ (namely,  $\psi^\e_L : = (\l^\e+1)(\AA^\e_{\#} + 1)^{-1} \eta_{L}  u^\e $) satisfying the conditions in Remark \ref{r6.3}. Thus, we conclude that $\l\in \mathcal R\mbox{-}\lim  \Sp (\mathcal{A}^{\varepsilon}_{\#})$.
 
\end{proof} 	
\begin{remark}\label{r:6.6}
	Note that in the above theorem we did not assume that the soft component $Y_0^{\#}$ is a collection of disconnected inclusions. In the case of connected $Y_0^{\#}$, or $Y_0^{\#}$ representing a collection of infinite fibres, the limit spectrum of 	$\AA^\e_{\#} $ is larger that the spectrum of the corresponding homogenised operator $\AA^\hom_{\#}$. The difference  $\lim \Sp(\AA^\e_{\#}) \setminus \Sp(\AA^\hom_{\#})$ is attributed to all possible quasi-periodic modes supported on   $Y_0^{\#}$, see e.g. \cite{Cooper2018}.
\end{remark}

\subsection{Relevant limiting spectrum}

\begin{theorem}\label{th:9.2}
	The	 relevant limiting spectrum  of $\AA^\e$ coincides with the spectrum of $\AA^\hom$:
	\begin{equation*}
		{\mathcal R}\mbox{-}\lim \Sp(\AA^\e) = \Sp(\AA^\hom).
	\end{equation*}
\end{theorem}

In order to illustrate the basic idea behind Definition  \ref{d:srel}, we compare it to the argument presented in Theorem \ref{th4.1} (assuming that $\l\notin \Sp(-\Delta_\OO)$) via Remark \ref{r6.3}. The bound \eqref{50} is sufficient for constructing from the restrictions $u^\e|_{\square_{\xi^\e}^L}$  a family of approximate eigenfunctions satisfying the condition \eqref{154}, by using cut-off functions and a resolvent argument (cf. Lemma \eqref{l:E.1}). Most of the energy of these approximate eigenfunctions will be localised in the vicinity of the cubes $\square_{\xi^\e}^L$, whose location in space is random (i.e. depends on $\o$ as well as on $\e$), hence the corresponding local spectral averages $\ell(\e^{-1}\xi^\e, \e^{-1} L, \l,\o)$ (cf. Remark \ref{r:spec_av})) will be controlled from above only by $\beta_\infty(\l)$. (Note that here the  local spectral averages are calculated on the rescaled cubes $\e^{-1}\square_{\xi^\e}^L$.)   However, if we assume that for every $L\in\N$ the cubes are centred at the origin,  i.e. $\xi^\e = 0$, as one has in Definition  \ref{d:srel}, then by the ergodic theorem the local spectral averages will converge to $\beta(\l)$ (cf. Remark \ref{r:3.6}).
	
We continue with the proof of Theorem \ref{th:9.2}.
	
\subsubsection{Proof of the inclusion ${\mathcal R}\mbox{-}\lim \Sp(\AA^\e) \subset \Sp(\AA^\hom)$}\label{ss:5.1}
	Let $\l \in {\mathcal R}\mbox{-}\lim \Sp(\AA^\e)$, and let  $ \psi^\e_L$ be as in Remark \ref{r6.3}. We denote 
	\begin{equation}\label{113}
	\begin{aligned}
    f_L^\e:=	(\AA^\e-\l)  \psi^\e_L.
	\end{aligned}
	\end{equation}
	 For each $L$ the sequences $f_L^\e$ and $\psi^\e_L$ are bounded and, hence, have subsequences that converge weakly stochastically two-scale as $\e\to 0$ to some $f_{L}, \psi_L \in L^2(\R^d\times \Omega)$, respectively. In particular,
	\begin{equation}\label{115b}
	\begin{aligned}
	\| f_{L}\|_{L^2(\R^d\times \O)} \leq \limsup_{\e  \to 0} \| f_{L}^\e\|_{L^2(\R^d)}.
	\end{aligned}
	\end{equation}
	By Theorem \ref{misha10} applied to the equation $(\AA^\e+1)  \psi^\e_L = (\l+1)  \psi^\e_L + f_L^\e,$
	we infer that $ \psi_L = \psi_{L,0} + \psi_{L,1} \in V$   is the solution to
	\begin{equation}\label{116}
	(\AA^\hom - \l)  \psi_L = f_L.
	\end{equation}

	It remains to show that $\psi_L$ does not vanish as $L\to \infty$. Multiplying \eqref{113} by ${\psi^\e_L}$  and integrating by parts, we easily obtain the bound
	\begin{equation*}
	\|\e \nabla \psi^\e_L\|_{L^2(S_0^\e)} + \| \nabla \widetilde \psi^\e_L\|_{L^2(\R^d)} \leq C,
	\end{equation*}
	where $\widetilde \psi^\e_L$ is the extension of $\psi^\e_L|_{S_1^\e}$ onto $\R^d$ by Theorem \ref{th:extension}. Moreover, by the same theorem we have
	\begin{equation*}
 \|  \widetilde \psi^\e_L\|_{L^2(\R^d)} \leq C.
	\end{equation*}
It follows from the last two bounds we have that, up to a subsequence,  $\widetilde \psi^\e_L$ converges  strongly in $L^2(K)$, for any bounded domain $K$, to some function $\widetilde \psi_L\in W^{1,2}(\R^d)$ (which a priori could be zero). Then, by the properties of stochastic two-scale convergence, cf. Proposition \ref{properties}, the convergence  $ \chi_1^\e\,\widetilde \psi^\e_L \stwoscale {\mathbf 1}_{\O\setminus\OO} \,\widetilde \psi_L $ holds. On the other hand, $\chi_1^\e \,\widetilde \psi^\e_L \wtwoscale  {\mathbf 1}_{\O\setminus\OO}\, \psi_L  = {\mathbf 1}_{\O\setminus\OO}\,\psi_{L,0} $. Thus,
	\begin{equation}\label{115}
	\widetilde \psi^\e_L \to \psi_{L,0} \mbox{ in }L^2(K) \mbox{ for any bounded domain } K.
	\end{equation} 

If $\l \in \Sp(-\Delta_\OO)$ then $\l \in \Sp(\AA^\hom)$. Assume  that $\l \notin \Sp(-\Delta_\OO)$. Then, arguing in a similar way as in Section \ref{ss:4.3}, cf. \eqref{47a}, Lemma \ref{l:res} and \eqref{54}, we infer that for small enough $\e$
	\begin{equation*}
\| \psi^\e_L\|_{L^2(\square^L)} \leq C \|\widetilde \psi^\e_L\|_{L^2(\square^{L+1})} + C \|f_L^\e\|_{L^2(\square^{L+1})}.
\end{equation*}
Combining this with \eqref{155}, passing to the limit via \eqref{115}, and taking into account \eqref{154}, we infer that for large enough $L$, one has $	 \| \psi_{L,0}\|_{L^2(\square^{L+1})} \geq C >0.$
It follows that  $\| \psi_L\|_{L^2(\R^d)} \geq C > 0$ and, therefore, since $f_L$ in \eqref{116} vanishes as $L\to \infty$ (cf. \eqref{115b}), we conclude that  $\l\in \Sp(\AA^\hom)$.

\subsubsection{Proof of the inclusion ${\mathcal R}\mbox{-}\lim \Sp(\AA^\e) \supset \Sp(\AA^\hom)$}\label{ss:9.2}

	Starting from a standard prototype $u_L$  of a Weyl sequence for the homogenised operator $\AA^\hom$,  we will construct for each element of the sequence and each $\e$ an approximate solution $u_{LC}^\e$ to the spectral problem for the operator $\AA^\e$ while maintaining control of the error. In this we will follow, to a certain extent, a general scheme developed in \cite{KamSm1}.

\medskip
	  
\noindent\textbf{Case $\l \in  \Sp(\AA^\hom)\setminus\Sp(-\Delta_\OO)$.}

\noindent A part of the argument and the initial construction we present below are similar to those contained in the proof of Theorem \ref{th4.4}. Therefore, we will often refer to the relevant places in Section  \ref{ss:4.5} to fill in some detail and reuse a number of formulae therein so as not be repetitive.  At the same time, since we are no longer in the periodic setting of Section \ref{ss:4.5}, we have no asymptotic bounds on the homogenisation correctors $N^\e_j$ and the terms $G^\e_j$ and $B^\e$ carrying the information about the microscopic structure of the composite. Therefore, we will need to keep track of these quantities in the error bounds explicitly.   

By the assumption we have  $\b(\l)\geq 0$. Similarly to the proof of Theorem  \ref{th4.4}, we consider $u(x) := {\mathcal Re}(e^{ik\cdot x})$ with $k$ such that $A_1^{\rm hom}k \cdot k = \b(\l)$ and denote  
\begin{equation*}
u_L:= \frac{\eta_L u}{\|\eta_L u\|_{L^2(\R^d)}},
\end{equation*} 
where $	\eta_L(\cdot):= \eta(\cdot/L),$
with the  cut-off function  $\eta \in C_0^\infty(\square)$   satisfying $0\leq\eta\leq 1$, $\eta\vert_{\square^{1/2}} = 1$.
 Then $u_L$ satisfies 
\begin{equation}\label{70a}
\|(-\nabla \cdot A_1^{\textrm{\rm hom}}\nabla  - \b(\l))  u_L\|_{L^2(\R^d)} \leq \frac{C}{L}
\end{equation}
and the bound \eqref{u_L_6}. We define $u_L^\e$  by setting
	\begin{equation*}
	u_L^\e : = (1+\l b^\e)u_L.
\end{equation*} 

Further, let  $p_j\in \mathcal X,\,j=1,\ldots,d$, be the solution to the problem \eqref{10} with $\xi=e_j$. By Lemma \ref{l7.6} we can assume that $p_j\in \vpot$. Then for any cube $\square^L$ and any $\e$ there exists a function $N_j^\e \in W^{1,2}(\square^L)$ such that 
\begin{equation*}
	\nabla N_j^\e = p_j^\e  \mbox{ in } \square^L, \quad \int_{\square^L} N_j^\e  = 0,
\end{equation*}
where $p_j^\e$ in the $\e$-realisation of $p_j$. By the ergodic theorem  
\begin{equation}\label{87}
	\begin{aligned}
			\lim_{\e\to 0} \| N_j^\e\|_{L^2( \square^L)}=0, \qquad \qquad
		\lim_{\e\to 0} \|\nabla N_j^\e\|_{L^2(\square^L)} = L^{d/2} \|p_j\|_{L^2(\Omega)}.
	\end{aligned}
\end{equation}
 We define $u_{LC}^\e: = u_L^\e + \partial_j u_L N_j^\e.$
In what follows we will often use the bound \eqref{u_L_6} without referring to it every time. Similarly to  \eqref{136}, we  estimate the corrector:
\begin{equation}\label{136a}
\|u_{LC}^\e - u_L^\e\|_{L^2(\R^d)} = \|\partial_j u_L N_j^\e\|_{L^2(\R^d)} \leq \frac{C}{L^{d/2}} \sum_j \| N_j^\e\big\|_{L^2(\square^L)}.
\end{equation}

As in the proof of Theorem \ref{th4.4}, we substitute $u_{LC}^\e$ and a test function $v\in W^{1,2}(\R^d)$ into the bilinear form associated with the operator $\AA^\e + 1$, namely
\begin{equation}\label{a_e_7}
	a^\e(u_{LC}^\e,v)  = \int\limits_{S_1^\e} A_1 \nabla u_{LC}^\e\cdot {\nabla v} +  \int\limits_{S_0^\e} \e^2 \nabla u_{LC}^\e\cdot {\nabla v} + \int\limits_{\R^d} u_{LC}^\e  v,
\end{equation}
 and begin by analysing the first term on the right-hand side of \eqref{a_e_7} (recall the decomposition \eqref{89c} and the bounds \eqref{135_6}, \eqref{a_vv_6}): 
\begin{equation}\label{70}
\int\limits_{S_1^\e} A_1 \nabla u_{LC}^\e\cdot  {\nabla v}  = \int\limits_{\R^d} \big( A_1^{\rm hom} \nabla u_{L} + [\chi_1^\e A_1 (e_j + \nabla N_j^\e) - A_1^\hom e_j] \partial_j u_L + \chi_1^\e N_j^\e A_1\nabla\partial_j u_L\big) \cdot  {\nabla \widetilde v^\e}. 
\end{equation}
The following bound is straightforward:
\begin{equation}\label{71}
\left|\int\limits_{\R^d}  \chi_1^\e N_j^\e A_1\nabla\partial_j u_L\cdot  {\nabla \widetilde v^\e}\right| \leq \frac{C}{L^{d/2}}  \|\nabla \widetilde v^\e\|_{L^2(\R^d)} \sum_j \| N_j^\e\big\|_{L^2(\square^L)}.
\end{equation}
In order to estimate the second term on the right-hand side of \eqref{70}, we employ a similar approach as in \eqref{96}--\eqref{72c}.  Denote by
\begin{equation*}
	g_j:={\mathbf 1}_{\O\setminus \OO} A_1(e_j+p_j) - A_1^\hom e_j\in \vsol, \quad j=1,\ldots,d,
\end{equation*}
the difference of fluxes, and consider its $\e$-realisations 
\begin{equation}\label{128a}
	g_j^\e(x):= g_j(T_{x/\e}\o)= \chi_1^\e A_1 (e_j + \nabla N_j^\e) - A_1^\hom e_j.
\end{equation} 
By  Corollary \ref{c12.3}, there exist skew-symmetric tensor fields $G^\e_j\in W^{1,2}(\square^L; \R^{d\times d})$ (often referred to as  flux correctors)  such that $g^\e_j = \nabla\cdot G_j^\e$ and 
\begin{equation}\label{130}
	\lim_{\e\to 0} \|G^\e_j\|_{L^2(\square^L)} = 0.
\end{equation} 
Proceeding as in \eqref{99}, we obtain
\begin{equation}\label{72}
\left|\int\limits_{\R^d}  [\chi_1^\e A_1 (e_j + \nabla N_j^\e) - A_1^\hom e_j] \partial_j u_L \cdot  {\nabla \widetilde v^\e} \right| \leq\frac{C}{L^{d/2}}\|\nabla \widetilde v^\e\|_{L^2(\R^d)}  \sum_j\|G_j^\e\|_{L^2(\square^L)}.
\end{equation}

Now we address the second integral on the right-hand side of \eqref{a_e_7}. Analogously to \eqref{73_6} and by utilising the identity \eqref{75_6}, we have
\begin{multline}\label{73}
	\e^2 \int\limits_{S_0^\e} \nabla u_{LC}^\e\cdot  {\nabla v}
	= \int\limits_{S_0^\e} \l u_L (\l b^\e + 1) {  v^\e_0}
	\\
	+ \e^2  \int\limits_{S_0^\e} \big(  -\l \nabla u_L\cdot \nabla b^\e \, {  v^\e_0} 
+ \l u_L \nabla b^\e\cdot  {\nabla \widetilde v^\e} +  \big[\nabla  u_L(1+\l b^\e) + \nabla( \partial_ju_L N_j^\e )\big]\cdot  {\nabla v}\big).
\end{multline}
We estimate the second integral on the right-hand side of (\ref{73}) by employing Lemma \ref{l9.9} and the bounds \eqref{135_6}, as follows:  
\begin{equation}\label{132}
	\begin{aligned}
&\Bigg|\e^2  \int\limits_{S_0^\e} \big(  -\l \nabla u_L\cdot \nabla b^\e \, {  v^\e_0} 
			+ \l u_L \nabla b^\e\cdot  {\nabla \widetilde v^\e} +  \big[\nabla  u_L(1+\l b^\e) + \nabla( \partial_ju_L N_j^\e )\big]\cdot  {\nabla v}\big)\Bigg|
		\\
&\quad		\leq C\e \bigg( \|\nabla \widetilde v^\e\|_{L^2(\R^d)} + \|\e \nabla v\|_{L^2(\R^d)} \bigg(1 + \frac{1}{L^{d/2}} \sum_j  \left(  \|N_j^\e\|_{L^2( \square^L)} + \|\nabla N_j^\e\|_{L^2( \square^L)} \right)\bigg) \bigg).
	\end{aligned}
\end{equation}

Combining (\ref{70}), (\ref{71}), (\ref{72})--(\ref{132}), and (\ref{a_vv_6}), we can rewrite the left-hand side of \eqref{a_e_7} as
\begin{equation*}
\begin{aligned}
a^\e(u_{LC}^\e,v)  = \int\limits_{\R^d}  A_1^{\rm hom} \nabla u_{L} \cdot  {\nabla \widetilde v^\e} + \int\limits_{S_0^\e} \l u_L (\l b^\e + 1) {  v^\e_0} + \int\limits_{\R^d} u_{LC}^\e   v + \mathcal R(\e,L,\lambda)\sqrt{a^\e(v,v)} ,
\end{aligned}
\end{equation*}
where the remainder term $\mathcal R(\e,L,\lambda)$ satisfies  (cf. \eqref{87}, \eqref{130})
\begin{equation}\label{156}
\begin{aligned}
\lim_{\e\to 0}\mathcal R(\e,L,\lambda)= 0 \qquad \forall L,\l.
\end{aligned}
\end{equation}

Proceeding as in the proof of Theorem \ref{th4.4}, cf. \eqref{91}, and recalling the definition of $\beta(\l)$, see \eqref{d:beta1},  yields
\begin{multline}\label{76}
a^\e(u_{LC}^\e ,v)  - (\l+1) (u_{LC}^\e ,v) 
\\
= \int\limits_{\R^d}  \left((-\nabla \cdot A_1^{\textrm{\rm hom}}\nabla  - \b(\l))  u_L  { \widetilde v^\e} +  \l^2 u_L  (\langle b \rangle - b^\e)  {   \widetilde v^\e} -\l \,\partial_j u_L N_j^\e   v\right) 
+  \mathcal R(\e,L,\lambda)\sqrt{a^\e(v,v)}.
\end{multline}
It remains to estimate the second term under the integral on the right-hand side of \eqref{76}. Since $\langle b \rangle - b^\e$ converges weakly to zero in $L^2_\loc(\R^d)$, by Lemma \ref{l13.3} there exists a sequence of vector fields $B^\e\in W^{1,2}(\square^L;\R^d)$ such that $ \nabla\cdot B^\e = \langle b \rangle - b^\e$ and 
\begin{equation}\label{137}
	\lim_{\e\to 0}\|B^\e\|_{L^2(\square^L)} \to 0.
\end{equation} 
Integrating by parts, we obtain
\begin{equation}\label{80}
\begin{aligned}
\bigg| \int\limits_{\R^d}  u_L  (\langle b \rangle - b^\e)  {   \widetilde v^\e}  \bigg| = \bigg|\int\limits_{\R^d}  \nabla(u_L   {   \widetilde v^\e})  \cdot B^\e\bigg| \leq \frac{C}{L^{d/2}} \left( \| \widetilde v^\e\|_{L^2(\R^d)}+\|\nabla \widetilde v^\e\|_{L^2(\R^d)}\right) \| B^\e\|_{L^2(\square^L)}.
\end{aligned}
\end{equation}
Combining  \eqref{70a}, \eqref{136a},  \eqref{156},  \eqref{137}, and \eqref{80} yields
\begin{equation*}
\begin{aligned}
|a^\e(u_{LC}^\e ,v)  - (\l+1) (u_{LC}^\e ,v) |   
\leq  \widehat{\mathcal R}(\e,L,\l)  \sqrt{a^\e(v,v)},
\end{aligned}
\end{equation*}
where (for fixed $L$ and $\lambda$)
\begin{equation*}
\begin{aligned}
\limsup_{\e\to 0} \widehat{\mathcal R}(\e,L,\l) \leq {C(\l)}/{L}.
\end{aligned}
\end{equation*}
Now we invoke Lemma \ref{l:E.1}: for all large enough $L$ and small enough $\e$ there exists  $\psi_L^\e \in \dom(\AA^\e)$ such that
\begin{equation*}
	\begin{aligned}
		\|\psi_{L}^\e -  u_{LC}^\e\|_{L^2(\R^d)} \leq   \widehat{\mathcal R}(\e,L,\l),
	\end{aligned}
\end{equation*}
\begin{equation*}
	\begin{aligned}
		\frac{\|(\AA^\e-\l) \psi_{L}^\e\|_{L^2(\R^d)}}{\| \psi_{L}^\e\|_{L^2(\R^d)}}\leq |\l+1| \frac{\widehat{\mathcal R}(\e,L,\l)}{1- \widehat{\mathcal R}(\e,L,\l)} .
	\end{aligned}
\end{equation*}
It follows that $\l \in {\mathcal R}\mbox{-}\lim \Sp(\AA^\e)$  as per Remark \ref{r6.3}.

\medskip

\noindent\textbf{Case $\l \in \Sp(-\Delta_\OO)$.}

\noindent The argument we utilise here is similar to the one we used above but is more straightforward technically. Since ${\mathcal R}\mbox{-}\lim \Sp(\AA^\e)$ is closed, cf. Remark \ref{r6.2} d., and due to the characterisation   \eqref{20a} of the spectrum of $-\Delta_\OO$, we can assume without loss of generality that $\l \in \Sp(-\Delta_{\OO_\o^k})$ for some $k$. Let $\varphi\in W^{1,2}_0(\OO_\o^k), \|\varphi\|_{L^2(\OO_\o^k)} =1$, be the corresponding eigenfunction, 
\begin{equation}\label{148c}
- \Delta_{\OO_\o^k}\varphi = \l \varphi.
\end{equation}
We define functions $\varphi^\e \in W^{1,2}_0(\e\OO_\o^k)$ by the formula
\begin{equation*}
	\varphi^\e(x) : =\e^{-d/2} \varphi(x/\e).
\end{equation*} 

Testing $\varphi^\e$ in the bilinear form $a^\e(\cdot,\cdot)$ with $v = \widetilde v^\e + v_0^\e \in W^{1,2}(\R^d)$ (recall the decomposition \eqref{89c}), and taking into account  \eqref{148c}, we have
\begin{multline}\label{143}
	a^\e(\varphi^\e, v) = \int_{\e\OO_\o^k}\e^2 \nabla \varphi^\e \nabla v_0^\e   + \int_{\e\OO_\o^k} \varphi^\e v_0^\e + \int_{\e\OO_\o^k}\e^2 \nabla \varphi^\e \nabla \widetilde v^\e   +  \int_{\e\OO_\o^k}\varphi^\e  \widetilde v^\e 
	\\
	= (\l+1)\int_{\e\OO_\o^k} \varphi^\e v_0^\e + \int_{\e\OO_\o^k}\e^2 \nabla \varphi^\e \nabla \widetilde v^\e   +  \int_{\e\OO_\o^k}\varphi^\e  \widetilde v^\e 
	\\
	= (\l+1)\int_{\e\OO_\o^k} \varphi^\e v + \int_{\e\OO_\o^k}\e^2 \nabla \varphi^\e \nabla \widetilde v^\e   - \l \int_{\e\OO_\o^k}\varphi^\e  \widetilde v^\e.
\end{multline}
Using the obvious identity $\|\e\nabla \varphi^\e\|_{L^2(\e\OO_\o^k)}^2 = \l$ we  estimate, cf. also \eqref{a_vv_6}, the second term in \eqref{143}:
\begin{equation*}
	\left| \int_{\e\OO_\o^k}\e^2 \nabla \varphi^\e \nabla \widetilde v^\e\right| \leq \e \|\e\nabla \varphi^\e\|_{L^2(\e\OO_\o^k)} \|\nabla \widetilde v^\e\|_{L^2(\e\OO_\o^k)} \leq C \e \sqrt{a^\e(v, v) }.
\end{equation*}
In order to estimate the last term in \eqref{143}, we use the Sobolev embedding for a suitable $p>2$:
\begin{multline}\label{145}
		\left|\int_{\e\OO_\o^k}\varphi^\e  \widetilde v^\e \right| \leq \| \varphi^\e\|_{L^2(\e\OO_\o^k)} \| \widetilde v^\e\|_{L^2(\e\OO_\o^k)} \leq \| \widetilde v^\e\|_{L^p(\e\OO_\o^k)} \left|\e\OO_\o^k\right|^{(p-2)/2p}  
		\\
		\leq C \e^{d(p-2)/2p}\| \widetilde v^\e\|_{W^{1,2}(\e\OO_\o^k)} \leq C \e^{d(p-2)/2p} \sqrt{a^\e(v, v) }.
\end{multline}

Combining \eqref{143}--\eqref{145} we obtain
\begin{equation*}
	|a^\e(\varphi^\e, v) - (\l+1) (\varphi^\e, v) | \leq C \e^{d(p-2)/2p} \sqrt{a^\e(v, v) }.
\end{equation*}
Arguing as at the end of the previous case, we conclude that $\l \in {\mathcal R}\mbox{-}\lim \Sp(\AA^\e)$.

\subsection{Irrelevant limiting spectrum and semi-group convergence}

The next theorem has direct implications for the parabolic and hyperbolic evolution problems, see Corollary \ref{cor: 6.6} below. It follows from  Definition \ref{d:srel} that $\l\in {\mathcal Irr}\mbox{-}\lim \Sp(\AA^\e)$ if and only if  $\l\in \lim \Sp(\AA^\e)$ and the quantity in \eqref{111} vanishes --- this property is local in the frequency domain. The property \eqref{equi13} below, on the other hand, is non-local in the frequency domain and cannot be obtained directly from the definition.

\begin{theorem} \label{ponmain} For any compact $K\subset{\R}\setminus{\rm Sp}({\mathcal A}^{\rm hom})$ one has almost surely
	\begin{equation} \label{equi13} 
		  \lim_{\eps \to 0}	 \|P_{\square^{L}} E^\e_{K}(\omega)\|_{L^2 \to L^2}=0 \quad  \forall L>0.
	\end{equation} 
\end{theorem} 	
\begin{proof} 
We assume that 
\begin{equation*}
	\limsup_{\eps \to 0}	  \|E^\e_{K}(\omega)\|_{L^2 \to L^2}>0, 
\end{equation*}
otherwise there is nothing to prove. Arguing by contradiction, suppose that \eqref{equi13} does not hold. Then 	there exists $L$ and a positive constant $M$ such that 
\begin{equation}\label{148b}
	\limsup_{\eps \to 0}	  \|P_{\square^{L}} E^\e_{K}(\omega)\|_{L^2 \to L^2} =M >0. 
\end{equation}

Without loss of generality, we assume that $K\subset (a,b] \subset [a,b] \subset (\Sp(\AA^\hom))^c$. Let us fix an sufficiently small positive value
\begin{equation*}
	\delta < \dist((a,b], \Sp(\AA^\hom)),
\end{equation*} 
and consider a partition $a_0=a<a_1<\dots<a_n=b$ of the interval $(a,b]$ with $\max_i |a_{i+1} - a_i|<\delta$. Then \eqref{148b} implies that (up to extracting a subsequence) there exist $i \in \{0,\ldots, n-1\}$ and a sequence $u^\e$ such that  $\|u^\e \|_{L^2(\R^d)} = 1$,    $E^\e_{(a_i, a_{i+1}]}(\omega) u^\e = u^\e$, i.e. for $\l\in(a_i, a_{i+1}]$ we have 
\begin{equation}\label{151}
\|f^\e \|_{L^2(\R^d)} < \delta,	\quad  f^\e:= (\AA^\e - \l) u^\e, 
\end{equation}
 and 
\begin{equation*}
\liminf_{\eps \to 0}	  \| u^\e\|_{L^2(\square^{L})}=	\liminf_{\eps \to 0}	  \|P_{\square^{L}} E^\e_{(a_i, a_{i+1}]}(\omega) u^\e\|_{L^2(\R^d)}\geq C(K) \delta^{1/2} M.
\end{equation*}
The last bound easily follows from the orthogonality of the spectral subspaces of $\AA^\e$ corresponding to the intervals $(a_i, a_{i+1}], i = 0,\ldots,n-1.$

By the standard energy estimate, the inequality \eqref{151} implies that 
\begin{equation*}
	\|\e \nabla u^\e\|_{L^2(S_0^\e)} + \| \nabla u^\e\|_{L^2(S_1^\e)} \leq C.
\end{equation*}
Consider the decomposition $u^\e = \widetilde u^\e + v^\e$, cf. \eqref{51}, where  $\widetilde u^\e$ is the harmonic extension of $u^\e|_{ S_1^\e}$ to the whole $\R^d$ according to Theorem \ref{th:extension} and $v^\e\in W^{1,2}_0( S_0^\e)$.  Note that  $\| \nabla \widetilde u^\e\|_{W^{1,2}(\R^d)} \leq C$. Arguing as in the proof of Theorem \ref{th4.1}, cf. \eqref{54}, we obtain
\begin{equation*}
	\|v^\e\|_{L^2(\square^L)} \leq \frac{\| \l \widetilde u^\e + f^\e\|_{L^2(\square^{L+1})}}{\dist \{\l,\Sp(-\Delta_{\OO})\}}.
\end{equation*}
Clearly, the latter implies that, for sufficiently small $\delta>0$,
\begin{equation*}
	\liminf_{\eps \to 0}	  \| \widetilde u^\e\|_{L^2(\square^{L+1})} \geq C(\l) \liminf_{\eps \to 0}	  \| u^\e\|_{L^2(\square^{L})}\geq C(K,\l) \delta^{1/2} M,
\end{equation*}
where the constant $C(\l)$ only depends on $\l$ and $\dist \{\l,\Sp(-\Delta_{\OO})\}$. We conclude that up to a subsequence $\widetilde u^\e$ converges weakly in  $W^{1,2}(\R^d)$ and strongly in $L^2(\square^{L+1})$ to some $u_0 \in W^{1,2}(\R^d)$ satisfying
\begin{equation}\label{157}
 \| u_0\|_{L^2(\R^d)}  \geq \| u_0\|_{L^2(\square^{L+1})} \geq C(K,\l) \delta^{1/2} M.
\end{equation}

Since $v^\e, f^\e$ are bounded in $L^2(\R^d)$ they converge up to a subsequence weakly stochastically two-scale (see Proposition \ref{properties}) to $v, f \in L^2(\R^d\times \Omega) $  respectively. Note that, since $v$ vanishes outside $\R^d \times\OO$, one has
\begin{equation}\label{158}
		\|u_0 + v\|_{L^2(\R^d\times \Omega)}^2  \geq 	\|u_0 \|_{L^2(\R^d\times \Omega\setminus \OO)}^2 \geq (1-P(\OO)){\| u_0\|_{L^2(\R^d)}^2} 
\end{equation}
  It remains to pass to the limit in the equation \eqref{151} via Theorem \ref{misha10}:
\begin{equation*}
	(\AA^\hom - \l) (u_0+v) = \mathcal P f,
\end{equation*}
where the right-hand side satisfies the bound  
\begin{equation*}
	\|\mathcal P f\|_{L^2(\R^d\times \Omega)}  < {\delta}.
\end{equation*}
Then \eqref{157}, \eqref{158} imply that 
\begin{equation*}
	\dist((a,b], \Sp(\AA^\hom)) \leq	\dist(\l, \Sp(\AA^\hom)) \leq \delta^{1/2} \frac{1}{C(K,\l) M (1-P(\OO))^{1/2}},
\end{equation*}
which is a contradiction, since we can choose $\delta$ to be arbitrary small.
\end{proof}

Next we will discuss the consequences of Theorem \ref{ponmain} on parabolic and hyperbolic evolution. The first claim gives the estimate on $L^2$-norm of the solution of parabolic problem for every time $t$, while the second one gives the estimate of $L^2$-norm of hyperbolic problem in arbitrary time $t$ with given initial velocity or initial position with additional regularity. Both of the claims imply that part of the initial condition supported in the irrelevant spectrum can be neglected on any finite domain. 
\begin{corollary} \label{cor: 6.6}
	Suppose that $U \subset \left(\Sp (\mathcal{A}^\hom)\right)^c$ is closed. Then for all $L, t>0,$ one has almost surely 
	\begin{enumerate} 
		\item 
		\begin{eqnarray*} 
 \limsup_{\eps \to 0}\left\|P_{\square^L}\int_{U}e^{-t\lambda} dE^\e_{(-\infty,\lambda]}(\omega)  \right\|_{L^2 \to L^2}&=& 0;
		\end{eqnarray*} 
		\item
		\begin{eqnarray*} 
		 \limsup_{\eps \to 0}\left\|P_{\square^L}\int_{U} \frac{\sin(t\sqrt{\lambda})}{\sqrt{\lambda}} dE^\e_{(-\infty,\lambda]}(\omega)  \right\|_{L^2 \to L^2}&=&0;
		\end{eqnarray*} 	
		\item If a sequence $(f^{\eps})_{\eps>0} \subset L^2(\R^d)$, $\|f^\e\|_{L^2(\R^d)}=1$, is such that 
		$$\lim_{\Lambda\to \infty} \sup_{\eps>0} \int_{\lambda>\Lambda}  \langle dE^\e_{(-\infty,\lambda]}(\omega)f^{\eps}, f^{\eps}\rangle =0, $$
		then 
		\begin{equation*} 
		\limsup_{\eps \to 0}\left\|P_{\square^L}\int_{U} \cos(t\sqrt{\lambda}) dE^\e_{(-\infty,\lambda]}(\omega)f^{\eps} \right\|_{L^2(\R^d)}
			=0.
		\end{equation*} 
	\end{enumerate} 	
\end{corollary} 	
\begin{proof} 
	The proof follows directly from Theorem \ref{ponmain} and from the fact that the functions  $e^{-t \lambda}$, ${\sin(t\sqrt{\lambda})}/{\sqrt{\lambda}} $, $\cos(t\sqrt{\lambda})$ are bounded on $[0,+\infty)$ and the first two vanish at infinity. 
\end{proof} 	

\appendix
\appendixpage

\section{Probability framework and  stochastic two-scale convergence}\label{ap:probability}
	
Let $U(x):  L^2(\Omega)\to  L^2(\Omega), x\in \R^d,$ be the unitary group defined by
\[
(U(x)f)(\o) =f(T_x \o),\quad f\in  L^2(\Omega).
\]
The unitarity follows from the measure preserving property of the dynamical system.  For each $j=1,\dots,d,$ we denote by  ${\mathcal D}_j$ the infinitesimal generator of the unitary group 
\begin{equation*}
	U(0,\dots,0, x_j, 0, \dots, 0), \quad j=1,\dots,d.
\end{equation*}
Its domain ${\rm dom}(\mathcal{D}_j)$ 
is a dense linear subset of $ L^2(\Omega)$ and consists of $f\in L^2(\Omega)$ for which the limit
\begin{equation*}
	\mathcal{D}_j f(\omega):=\lim_{x_j \to 0}  \frac{f(T_{(0, \dots, 0, x_j, 0, \dots, 0)}\omega)-f(\omega) }{x_j}
\end{equation*}
exists in $L^2(\Omega).$  Note that ${\rm i}\mathcal{D}_j,$ $j=1,\dots,d,$ are self-adjoint, pairwise commuting linear operators on  $L^2 (\Omega).$ We denote	$\nabla_{\omega}:=(\mathcal D_{1},\dots, \mathcal D_d).$
Furthermore, we define 
\begin{equation*}
	\begin{gathered}
		W^{1,2}(\Omega):=\bigcap_{j=1}^d{\rm dom}(\mathcal{D}_j),
		\\
		W^{k,2} (\Omega):=\bigl\{f \in L^2(\Omega):\mathcal D_1^{\alpha_1}\dots \mathcal D_d^{\alpha_d} f \in L^2(\Omega),\; \alpha_1+\cdots +\alpha_d=k\bigr\},\qquad
		\\[0.5em]
		W^{\infty,2} (\Omega):= \bigcap_{k \in \mathbf{N}} W^{k,2}(\Omega),
		\\
		\mathcal{C}^{\infty} (\Omega)= \set[\big] {f \in W^{\infty,2} (\Omega): \forall (\alpha_1,\dots, \alpha_d) \in \mathbf{N}_0^d \quad \mathcal D_1^{\alpha_1} \dots \mathcal D_d^{\alpha_d} f \in L^{\infty} (\Omega) }.
	\end{gathered}
\end{equation*}
It is known that $W^{\infty,2}(\Omega)$ is dense in $L^2(\Omega)$, the set $\mathcal{C}^{\infty} (\Omega)$ is dense in $L^p(\Omega)$ for all $p \in [1,\infty)$ as well as in $W^{k,2}(\Omega)$ for all $k$,  and $W^{1,2} (\Omega)$ is separable (using the fact that $\mathcal F$ is countably generated). 

The following subspace of $W^{1,2}(\Omega)$  plays an essential role in our analysis:
$$  W_{0}^{1,2}(\mathcal{O}):=\bigl\{v \in W^{1,2}(\Omega): 
v(T_x \omega)=0 \textrm{ on }  \R^d \backslash \mathcal{O}_{\omega}\ \ \forall\omega\in\O\bigr\}. $$
Note that as a consequence of the ergodic theorem (Theorem \ref{thmergodic}) one has
$$  W_0^{1,2}(\mathcal{O})=\bigl\{ v \in W^{1,2}(\Omega): {\mathbf 1}_{\mathcal{O}} v=v\bigr\},$$
i.e. $W_0^{1,2}(\mathcal{O})$ consists of $W^{1,2}$-functions that vanish on $\Omega \backslash \mathcal{O}.$ We also introduce the space 
$$ \mathcal{C}_0^{\infty} (\mathcal{O}):=\bigl\{v \in \mathcal{C}^{\infty} (\Omega): v= 0 \textrm{ on } \Omega \backslash \mathcal{O}\bigr\}.$$
It is not difficult to see that Assumption \ref{kirill100} implies that the assumptions of Lemmata 3.1 and 3.2 in \cite{ChChV} are satisfied, and hence the following statement holds.
\begin{lemma} \label{lemmagloria} The space $\mathcal{C}_0^{\infty} (\mathcal{O})$ is dense in  $L^2(\mathcal{O})$ and in $W_0^{1,2} (\mathcal{O})$.
\end{lemma}

One can equivalently define $\mathcal{D}_jf$ as the function with the property 
\begin{equation*}
	\int_{\Omega} g\mathcal{D}_jf=-\int_{\Omega } f\mathcal{D}_jg\qquad \forall g \in \mathcal{C}^{\infty}(\Omega).
\end{equation*}
Note that for \ae $\omega \in  \Omega$ one has 
\[
\mathcal{D}_j f (T_x\omega) = \frac{\partial}{\partial x_j}f(x,\o),
\]
 where the expression on the right-hand side  is the distributional derivative of $f(\cdot,\omega ) \in L^2_{\rm loc} (\R^d)$. For a random variable $f \in L^2(\Omega)$, its realisation $f \in L^2_{\rm loc}(\R^d, L^2(\Omega))$ is a $T$-stationary random field, i.e.  $f(x+y,\omega)= f(x, T_y \omega)$. Moreover, there is a bijection between random variables from $L^2(\Omega)$ and $T$-stationary random fields from $L^2_{\rm loc}(\R^d, L^2(\Omega))$. This can be extended to Sobolev spaces, where one has a higher regularity in $x$ for realisations. Namely, the following identity holds (see \cite{gloria1} for details):
\begin{equation*} 
	\begin{aligned}
		W^{1,2} (\Omega) &=\bigl\{ f \in W^{1,2}_{\rm loc}\bigl(\R^d, L^2(\Omega)\bigr): f(x+y,\omega)= f(x, T_y \omega) \quad \forall x,y,\ \text{\ae}  \omega \bigr\}
		\\[0.4em]
		&=\bigl \{ f \in C^1\bigl(\R^d, L^2(\Omega)\bigr): f(x+y,\omega)= f(x, T_y \omega) \quad  \forall x,y,\ \text{\ae}  \omega\bigr\}. 
\end{aligned}\end{equation*}

Following \cite{zhikov2}, we define the spaces $\lpot$ and $\lsol$ of potential and solenoidal vector fields. Namely, a vector field $f\in L^2_\loc(\R^d)$ is called potential if it admits a representation $f= \nabla u, \, u \in W^{1,2}_\loc(\R^d)$. A vector field $f\in L^2_\loc(\R^d)$ is called solenoidal if 
\[
\int_{\R^d} f_j \frac{\partial \varphi}{\partial x_j} = 0 \qquad \forall\varphi\in C^\infty_0(\R^d).
\]
A vector field $f\in L^2(\Omega)$ is called potential (respectively, solenoidal), if almost all its realisations $f(T_x \o)$ are potential (respectively, solenoidal) in $\R^d$. The spaces $\lpot$ and $\lsol$ are closed in $L^2(\O)$. Setting 
\begin{equation*}
	\begin{aligned}
		\vpot:= &\{f\in \lpot, \langle f \rangle = 0\},\qquad		\vsol:= &\{f\in \lsol, \langle f \rangle = 0\},
	\end{aligned}
\end{equation*}
we have the following orthogonal decomposition (``Weyl's decomposition'')
\begin{equation*}
	L^2(\O) = \vpot \oplus \vsol \oplus \R^d.
\end{equation*} 
Note that $\vpot$ is the closure of the space $\{\nabla_\o u,\, u\in W^{1,2}(\O)\}$ in $L^2(\O)$.

We define the following notion of stochastic two-scale convergence, which is a slight variation of the definition given in \cite{zhikov1}. We shall stay in the Hilbert setting ($p=2$), as it suffices for our analysis. 

Let $S$ be an open Lipschitz set in $\R^d.$
Suppose that $\mathcal{C} \subset \C^\infty(\Omega)$ is a countable dense family of vector-functions in $L^2(\Omega)$ (recall that the latter is separable)  and  $\Omega_t=\Omega_t(\mathcal{C})\subset \Omega$ is a set of probability one such that the claim of Theorem~\ref{thmergodic} holds for all $\omega \in \Omega_t$ and $g\in \mathcal{C}$. Elements of $\Omega_t$ are often referred to as \emph{typical}, hence the subscript ``t''.

\begin{definition}\label{definicija1}
	Let $\{T_{x}\randomelement\}_{x \in \R^d}$ be a typical trajectory and $(u^\e)$ a bounded sequence
	in $L^2(S)$. We say that $(u^\e)$ {weakly stochastically two-scale converges} to 
	$u \in L^2(\pspace \times \randomspace )$ 
	and write
	$u^\e \wtwoscale u,$ if 
	\begin{equation*}
		\lim_{\e \downarrow 0} \int_{\pspace} u^\e(x) \varphi(x) g\bigl(T_{\e^{-1} x} \randomelement\bigr) \td x
		= \int_{\randomspace} \int_\pspace u(x,\randomelement) \varphi(x) g(\randomelement) \td x\, \drandommeasure\quad\quad \forall \varphi\in C^\infty_0(S),\, g\in \mathcal{C}.
	\end{equation*}
	
	If additionally 
	$ \norm  { u^\e}_{L^2(S)} \to \norm { u}_{L^2(S \times \Omega)},$
	we say that $(u^\e)$ strongly stochastically two-scale converges to $u$ and write $u^\e \stwoscale u$.
\end{definition}

In the next proposition we collect some properties of stochastic two-scale convergence, see \cite{ChChV} for the proof. 

\begin{proposition} \label{properties}  Stochastic two-scale convergence has the following properties.
	\begin{enumerate}
		\item
		Let $(u^\e)$ be a bounded sequence in $L^2(S)$.
		Then there exists a subsequence (not relabelled) and $u \in L^2(S\times\Omega)$ 
		such that $u^{\e} \wtwoscale u$.
		\item If $u^{\e} \wtwoscale u$ then 
		$ \|u\|_{L^2(S \times \Omega)} \leq \liminf_{\e \to 0} \|u^{\e}\|_{L^2 (S)}.$
		\item If $(u^\e) \subseteq L^2(S)$ is a bounded sequence with $u^\e \to u$ in $L^2(S)$ for some $u \in L^2(S)$, then $u^\e \stwoscale u$.
		\item If $(v^\e) \subseteq L^{\infty} (S)$ is uniformly bounded by a constant and $v^\e \to v$ strongly in $L^1(S)$ for some $v \in L^\infty(S)$, and  $(u^\e )$ is bounded in $L^2(S)$ with  $u^\e \wtwoscale u$ for some $u \in L^2(S \times \Omega),$ then $v^\e u^\e \wtwoscale vu$.  
		\item Let $(u^\e)$ be a bounded sequence in $W^{1,2}(S)$. Then on a subsequence (not relabelled) 
		$u^\e \weakly u^0$ in $W^{1,2}(S),$ and there exists $w \in L^2( S, \vpot)$ 
		such that $\nabla u^\e \wtwoscale \nabla u^0 + w(\cdot, \omega)\,.$
		\item Let  $(u^\e)$ be a bounded sequence in $L^2(S)$ such that $\e \nabla u^\e$ is bounded in $L^2(S, \R^d)$. Then there exists $u \in L^2( S, W^{1,2}(\randomspace))$ such that on a subsequence $			u^{\e}  \wtwoscale u$ and  $
			\e \nabla u^{\e}  \wtwoscale \nabla_{\omega} u(\cdot,\omega). $
	\end{enumerate} 
\end{proposition}

\section{Measurability properties}\label{s:auxiliary}

In this part of the appendix we collect a number of technical results on the measurability of various quantities being used throughout the paper.

For $q=(q_1,\dots,q_d) \in \mathbf{Q}^d$ we define the set 
$$ \mathcal{O}_q:=\bigl\{\omega \in \mathcal{O}:\textrm{there exists } k_0 \in \mathbf{N} \textrm{ such that } \{0,q\} \subset \mathcal{O}^{k_0}_{\omega}\bigr\}. $$
Recall that we reserve the index $k_0$ for the inclusion $\OO_\o^{k_0}$ containing the origin (assuming that $\o\in\OO$).
Note that for a fixed $\o\in\OO$ the set of points $q\in \Q^d$ such that $\o\in\OO_q$ is exactly  $\OO_\o^{k_0} \cap \Q^d$.
\begin{lemma} \label{lucia1} 
	For every $q \in \mathbf{Q}^d$, the set $\mathcal{O}_{q} \subset \Omega$ is measurable. 
\end{lemma} 
\begin{proof}  
	Note that 
	\begin{eqnarray}\label{connection1}
		\omega \in  \mathcal{O}_q & \iff &\textrm{There exists a polygonal line $L$ that connects $0$ and $q$ and } \\ \nonumber & & \textrm{consists of a finite set of straight segments with rational endpoints } \\  \nonumber & & \textrm{such that for all 
			$l \in \mathbf{Q}^d$ on this line one has }
		\textrm{$T_l \omega \in \mathcal{O}$.} 
	\end{eqnarray} 
	Since for each fixed $q \in \mathbf{Q}^d$ there is a countable number of lines satisfying the property \eqref{connection1}, the set $\mathcal{O}_q$ is measurable. 	
\end{proof} 	
 
We define the random variables
$$ 
\tilde D_i(\omega):=\inf\{q_i: \ \omega \in \mathcal{O}_q\},\quad \omega\in\Omega,\quad i=1, \dots,d. 
$$
Note that  $\tilde D_i=+\infty$ whenever $\omega \notin \mathcal{O},$ and for $\o\in\OO$ one has $\tilde D_i = (D_\o^{k_0})_i$. Furthermore, we denote by $D = D(\o)$ the random vector $D:=-(\tilde D_1, \dots, \tilde D_d)^T-d_{1/4}.$
Finally, for $\o\in\OO$ we define  
$P_{\omega}:= \OO_\o^{k_0} + D,$
which coincides with the definition in Section \ref{lim_eq}. Note that by Assumption \ref{kirill100} we have $P_\o \subset \square$.

\begin{lemma} 
	\label{ante1}
	The mapping $(x,\omega) \mapsto \dist(x, P_{\omega})$ is measurable from $\R^d \times \mathcal{O}$ to $\mathbf{R}$ (on $\R^d \times \mathcal{O}$ we take the product of Borel $\sigma$-algebra with $\mathcal{F}$). 
\end{lemma} 
\begin{proof}
	The statement follows from the representation 	
	$$ \dist(x, P_{\omega})= \inf_{q_1 \in \mathbf{Q}^d } \Big( |x-D(\o)-q_1|+\inf_{q_2 \in \mathbf{Q}^d}\bigl\{\dist(q_1,q_2): \omega \in \mathcal{O}_{q_2}\bigr\} \Big), \, (x,\o) \in \R^d\times \OO.$$
\end{proof}

\begin{lemma}\label{measp}
	The set-valued mapping $\mathcal{H}:\omega \mapsto \overline{P_\omega}$ is measurable, where on the closed subsets of $\square$ we take the $\sigma$-algebra generated by the Hausdorff distance (topology) $d_{\rm H}$.  	
\end{lemma}	
\begin{proof} 
	The Hausdorff topology on the closed subsets of a compact set is compact and thus separable. 
	We recall that, in line with the notation introduced in Section \ref{ss:4.4}, 
	\[
	\overline{B_{H,\square}(K,r)}= \{U\subset\square: U \mbox{ is compact}, d_{\rm H}(K,U)\leq r\}.
	\] 
	It is sufficient to prove that the set
	$$\mathcal{H}^{-1} (\overline{B_{H,\square}(K,r)}^{\,\rm c})=\{\omega \in \Omega: d_{\rm H}(\overline{P_\omega},K)>r\},$$
	is measurable. But this can be easily seen by using Lemma \ref{ante1} and the representation
	\begin{equation*}
		\begin{aligned}
			\mathcal{H}^{-1} (\overline{B_{H,\square}(K,r)}^{\,\rm c})=\{\omega \in \Omega: \exists q \in \Q^d  \textrm{ such that } \dist(q,P_\o)=0 \mbox{ and }& \dist(q,K)>r, 
			\\
			\textrm{ or } & q \in  K \textrm{ and } \dist(q,P_{\omega}) >r  \}.     
		\end{aligned}
	\end{equation*}    
\end{proof}

Let  $\{\widetilde \varphi^l\}_{l \in \mathbf{N}} \subset C_0^{\infty} (\square) $ be a family of functions dense in $W_0^{1,2}  (\square)$. Further, let $\rho \in C_0^{\infty} (\mathbf{R}^d)$ be non-negative  with $\int_{\mathbf{R}^d} \rho=1,$ $\supp\rho \subset B_1$, and denote $\rho_{\delta}(x)=\delta^{-d} \rho(x/\delta)$. We denote the characteristic function  of the set $\OO_\o^{k_0,m}+D$ by 
\begin{equation*}
	\begin{gathered}
			\chi^{m}(x,\omega ):={\mathbf 1}_{\OO_\o^{k_0,m}}(x -D),
			\\
			\OO_\o^{k_0,m}:=  \{x: \, x\in \OO_\o^{k_0}, \, \dist(x, \partial\OO_\o^{k_0})>{1}/{m}\},
	\end{gathered}
\end{equation*}
 and define a mapping from $\Omega$ taking values in $W_0^{1,2}(\square)$:
\begin{equation*}
	\varphi^{l,m}(x,\omega):=\rho_{1/2m} *\big(\chi^m(x,\omega )\widetilde \varphi^l(x)\big),
\end{equation*}
where $*$ stands for the usual convolution of functions. Note that for \ae $\omega \in \mathcal{O}$ one has $\supp \varphi^{l,m}(\cdot,\omega) \subset P_{\omega}$.

\begin{lemma} \label{ante10}
	For every $l,m \in \mathbf{N}$, the mapping $\omega \mapsto \varphi^{l,m}(\cdot,\omega)$ taking values in $W_0^{1,2}(\square)$ is measurable with respect to the Borel $\sigma$-algebra on $W_0^{1,2}(\square)$. 
\end{lemma}

\begin{proof}
	First note that 
	\begin{equation} \label{ante2} 
		\omega \mapsto \chi^{m}(\cdot, \omega)\widetilde \varphi^l(\cdot), 
	\end{equation} 
	is a measurable mapping taking values in the set $L^2(\square)$, with Borel $\sigma$-algebra. To check this note that for each $q \in \mathbf{Q}^d$ the set $ B_q:=\{\omega \in \Omega: q \in \OO_\o^{k_0,m}\}$ 	
	is measurable (the proof is similar to that of Lemma \ref{lucia1}). 
	Further, for $\psi \in C_0^{\infty} (\R^d)$ we have that 
	$\| \psi- \chi^{m} \widetilde \varphi^l\|_{L ^2}$ can be written as a limit of Riemann sums and each Riemann sum can be written in terms of finite number of ${\mathbf 1}_{B_q}$ and values of  function $\widetilde \varphi^l$. Thus $\omega \mapsto\| \psi- \chi^{m} \widetilde \varphi^l\|_{L ^2}$ is measurable.  
	Since the topology in $L^2(\R^d)$ is generated by the balls $B(\psi,r)$, where $\psi \in C_0^{\infty} (\R^d)$ and $r \in \mathbf{Q}$, we have that the mapping given by \eqref{ante2} is measurable. The final claim follows by using the fact that the convolution is a continuous (and thus measurable) operator from $L^2$ to $W^{1,2}$. 
\end{proof}
Note that by the construction we have that for \ae $\omega \in \Omega$ the family  $\{\varphi^{l,m}(\cdot,\omega)\}_{l,m \in\mathbf{N}}$ $\subset  C_0^{\infty} (P_{\omega})$ is a dense subset of $W_0^{1,2} (P_{\omega})$  (cf. Lemma \ref{lemmagloria}). Let's introduce the sets $Y_{[L,U]} \subset\mathcal{O}$  for some random variables $L,U: \Omega \to \mathbf{R}^+_{0}$:
\begin{equation*}
	Y_{[L,U]}:=\bigl\{\omega \in\mathcal{O}: 	-\Delta_{\mathcal{O}^{k_0}_{\omega}}\ \textrm{has an eigenvalue in the random interval\ }  [ L, U ] \bigr\}.
\end{equation*}
We also define the set $S_{[L,U]} \subset W_0^{1,2} (P_{\omega})$ as
\begin{equation*}
	S_{[L,U]}:=\bigl\{\psi \in W_0^{1,2}(P_{\omega}):  \psi\ \textrm{is an eigenfunction of } -\Delta_{P_{\omega}} \textrm{ whose eigenvalue is in } 
	[ L, U ]  \bigr\}. 
\end{equation*} 	

For every $r \in \mathbf{R}$ and $l,m \in \mathbf{N}$ we define the random variable 
$$ X_r^{l,m}:= \left\{\begin{array}{lr}
	\dfrac{\|-\Delta \varphi^{l,m}(\cdot,\omega)-r\varphi^{l,m}(\cdot,\omega)\|_{W^{-1,2 }(P_{\omega})}}{\|\varphi^{l,m}(\cdot,\omega)\|_{L^2(P_{\omega})}} & \textrm{if } \varphi^{l,m}(\cdot,\omega)\neq 0, \\[0.9em] +\infty & \textrm{otherwise.} \end{array} \right. $$
\begin{lemma}
	$X_r^{l,m}$ is a measurable function for every $r \in \mathbf{R}$ and $l,m \in \mathbf{N}$.  
\end{lemma}
\begin{proof}
	We use Lemma \ref{ante10} and the fact that
	$-\Delta$ is a continuous map from $W^{1,2}(\R^d)$ to $W^{-1,2}(\R^d)$ (we make use of the natural embedding of $W_0^{1,2}(P_{\omega})$ into  $W^{1,2}(\R^d)$)  and $\|\cdot\|_{W^{-1,2}(P_{\omega})}$ is a measurable mapping  from $W^{-1,2}(\R^d)$ to $\R$, since 
	$$\|\psi(\cdot,\omega)\|_{W^{-1,2}(P_{\omega})}=\sup_{l,m\in \mathbf{N}}\left\{\frac{_{W^{-1,2}(P_{\omega})}\big\langle \psi(\cdot,\omega), \varphi^{l,m}(\cdot,\omega) \big\rangle_{W^{1,2}_0(P_{\omega})} }{\|\varphi^{l,m}(\cdot,\omega)\|_{W^{1,2}(P_\omega)}}: \varphi^{l,m}(\cdot,\omega) \neq 0 \right\},$$	
	where $_{W^{-1,2}(P_{\omega})}\langle\cdot, \cdot\rangle_{W^{1,2}_0(P_{\omega})}$  is the duality between $W^{-1,2}(P_{\omega})$ and $W^{1,2}_0(P_{\omega})$.
\end{proof}
\begin{lemma}
	For measurable $L,U$, the set $Y_{[L,U]}$ is measurable. 
\end{lemma}
\begin{proof}
	The claim follows by observing that
	$$ Y_{[L,U]}=\cap_{n \in \mathbb{N}}\Bigl\{\o: \inf_{l,m \in \mathbf{N}, r \in \mathbf{Q} \cap [ L-\nicefrac{1}{n},U+\nicefrac{1}{n}]}X^{l,m}_r=0\Bigr\}.$$
	Note that we use the intervals $[ L-\nicefrac{1}{n},U+\nicefrac{1}{n}]$ in the above in order to address the case $L=U$.
\end{proof}

\begin{lemma}
	Let $\Phi: \Omega \to L^2(\R^d)$ and $L,U:\Omega \to \mathbf{R}^{+}_0$ be random variables.  Then the mapping
	\begin{equation*}
		\o \mapsto \left\{\begin{array}{ll}
			\dist_{L^2(\R^d)}\bigl(\Phi, S_{[L,U]}\bigr) & \mbox{ \rm if }\o\in\OO,
			\\
			+\infty & \mbox{ \rm otherwise}.
		\end{array} \right.
	\end{equation*}
 	is  measurable. 
\end{lemma}
\begin{proof}
	The claim follows from the formula
	\begin{equation*}
		\begin{aligned}
&	\dist_{L^2(\R^d)}\bigl(\Phi, S_{[L,U]}\bigr) 
	\\
& \quad=\liminf_{k\to \infty}\limsup_{n\to \infty} \inf_{l,m \in \mathbf{N}}\biggl\{\bigl\|\varphi^{l,m}(\cdot,\omega)-\Phi\bigr\|_{L^2(\R^d)}:X_r^{l,m} < \nicefrac{1}{n}
\textrm{ for some } r \in \mathbf{Q} \cap [L-\nicefrac{1}{k},U+\nicefrac{1}{k}]\biggr\}. 
		\end{aligned}	
	\end{equation*}
\end{proof}

	Recall that for $\omega \in \mathcal{O}$ we denote by $\Lambda_1(\omega)<\Lambda_2(\omega)<\dots<\Lambda_s(\omega)< \dots $ the eigenvalues of $-\Delta_{P_\o} $.
\begin{lemma} \label{l:B.9}
 For every $s \in \mathbf{N}$ the mapping $\Lambda_s$ is measurable.  
\end{lemma} 	
\begin{proof}
	For $s=1$	we use the equality $ \{\o:\Lambda_1\leq x\}=Y_{[0,x]}, \, x \in \mathbf{R},   $
	while for $s>1$ one has
	$$ \{\Lambda_s\leq x\}=\bigcup_{n \in \mathbf{N}}Y_{[\Lambda_{s-1}+\nicefrac{1}{n},x]}, \, x \in \mathbf{R}.     $$ 
	In each case, the claim follows immediately.
\end{proof}
\begin{lemma} 
	For each $\varphi \in L^2(\mathbf{R}^d)$ and $s\in \mathbf{N}$  the mapping $\omega \mapsto P_{S_{[\Lambda_s,\Lambda_s]}} \varphi$, taking values in  $L^2(\mathbf{R}^d)$, where $P_{S_{[\Lambda_s,\Lambda_s]}}$ is the $L^2$-orthogonal projection on $S_{[\Lambda_s,\Lambda_s]}$,  is measurable. 	
\end{lemma} 
\begin{proof} 
	For every $n \in \mathbf{N}$ we define the random variable
	$$ H_s^n:=  \inf_{l,m \in \mathbf{N}}\Big\{\bigl\|\varphi^{l,m}(\cdot,\omega)-\varphi\bigr\|_{L^2(\mathbf{R^d})}: \dist_{L^2} (\varphi^{l,m}(\cdot,\omega),S_{[\Lambda_s,\Lambda_s]} )<\nicefrac{1}{n} \B	ig\}.   $$
	We also define the  random variables $P^n_s$ as follows: 
	\begin{eqnarray*}
		P^n_s(\omega) = \varphi^{l_n(\omega),m_n(\omega)} (\omega), \qquad
		l_n(\omega) &:=& \min_{l \in \mathbb{N}} \left\{ \exists m \in\mathbf{N}: \bigl\|\varphi^{l,m}(\cdot,\omega)-\varphi\bigr\|_{L^2(\mathbf{R^d})} < H^n_s+\nicefrac{1}{n}   \right\},\\
		m_n(\omega)&:=& \min_{m \in \mathbb{N}} \left\{\bigl\|\varphi^{l_n(\o),m}(\cdot,\omega)-\varphi\bigr\|_{L^2(\mathbf{R^d})}< H^n_s+\nicefrac{1}{n}  \right\}.
	\end{eqnarray*}
	It is easy to see that for a fixed $\o\in \OO$ one has $  P_{S_{[\Lambda_s,\Lambda_s]}} \varphi=\lim_{n \to \infty} P^n_s,$
	where the convergence on the right-hand side is in $L^2(\mathbf{R^d})$.
\end{proof} 		
\begin{lemma}\label{l:b.11}
	For every $\omega \in \mathcal{O}$ and $s \in \mathbf{N}$ we define $N_s(\omega)$ as the dimension of $S_{[\Lambda_s, \Lambda_s]}$. Then $N_s$ is measurable. Moreover, there exist $\Psi^1_s(\omega), \Psi^2_s(\omega),\dots, \Psi^n_s(\omega), \dots$, measurable, taking values in $L^2(\R^d)$, such that $\Psi^1_s,\dots, \Psi_s^{N_s}$ is an orthonormal basis in $S_{[\Lambda_s, \Lambda_s]}$  (in the sense of $L^2(\R^d)$), and $\Psi_s^{N_s+1}=\Psi_s^{N_s+2}=\cdots=0$.   
\end{lemma}
\begin{proof} 
	We take the sequence $\tilde \varphi^k$ and define the measurable functions $ C^k_s= P_{S_{[\Lambda_s, \Lambda_s]}} \tilde \varphi^k$ taking values in $L^2(\R^d)$. Then we construct the sequence $\tilde \Psi^k_s$ by applying the Gramm-Schmidt orthogonalization process to the sequence $C^k_s$.  Note that for every $s \in \mathbf{N}$ and $\omega \in \mathcal{O}$ we have that there is at most finite number of $\tilde \Psi^k_s$ that are different from zero and they form the orthonormal basis in $S_{[\Lambda_s, \Lambda_s]}$. The claim follows by rearranging the sequence $\tilde \Psi^k_s$, taking into account that the set  of all finite subsets of $\mathbf{N}$  is countable. 
\end{proof} 	

Clearly, the sequence $\{\Psi_s^p\}_{s\in \N, p = 1,\ldots, N_s}$ in an orthonormal basis in $L^2(P_\o)$.

	We conclude this part of the Appendix by providing a streamlined  proof of \cite[Corollary 4.9 and Theorem 5.6]{ChChV}
	\begin{proposition} \label{pppppp1}
		Under  Assumption \ref{kirill100} we have 
		\begin{equation*} 
			\Sp(-\Delta_\OO)= \overline{ \bigcup_{x \in \R^d} \bigcup_{s \in \mathbf{N}} \{ \Lambda_s (T_x \omega)\}}.
		\end{equation*} 	
	\end{proposition} 	
	\begin{proof} 
		We denote 
		$$ \textrm{Eig}(\omega):=\overline{ \bigcup_{x \in \R^d} \bigcup_{s \in \mathbf{N}} \{ \Lambda_s (T_x \omega)\}}.$$
		It is easy to see that the set $\textrm{Eig}(\omega)$ is almost surely independent of $\omega \in \Omega$, i.e. deterministic. Indeed, it is clear that for $q \in \mathbf{Q}$, $n \in \mathbf{N}$ the set $ A_{q,n}:=\left\{\omega \in \Omega: (q-\nicefrac{1}{n},q+\nicefrac{1}{n}) \subset    (\textrm{Eig}(\omega))^c \right\} $
		is translation invariant and thus is of probability zero or one. Then the set 
		$$
		\Omega_1:= \Bigg(\bigcap_{q \in \mathbf{Q}, n \in \mathbf{N}: P(A_{q,n})=1} A_{q,n}\Bigg) \bigcap \Bigg(\bigcap_{q \in \mathbf{Q}, n \in \mathbf{N}: P(A_{q,n})=0} (A_{q,n})^c \Bigg)
		$$
		has probability one. First, we have a trivial identity valid for every $\o\in \O$,
		$$
     	\bigcup_{q \in \mathbf{Q}, n \in \mathbf{N}: \o\in A_{q,n}} (q-\nicefrac{1}{n},q+\nicefrac{1}{n}) = (\textrm{Eig}(\omega))^c.
    	$$
    	Next, it is not difficult to see that the set on the left-hand side is constant for all $\o\in \O_1$, which proves the claim. In the rest of the proof we omit the dependence of the set $\textrm{Eig}(\omega)$ on $\omega$.
    	
		The inclusion  $(\textrm{Eig})^c \subset (\Sp(-\Delta_\OO))^c$ is a direct consequence of \eqref{sasha110} and Proposition \ref{josip3}. From this we conclude that $\Sp(-\Delta_\OO) \subset \textrm{Eig}$. 
		
		To prove the opposite inclusion, take $\lambda \in \textrm{Eig}$. 
		It is easy to see that for all $\delta>0$ the set  $I_{\delta}:=\{\omega \in \mathcal{O}: [\lambda-\delta,\lambda+\delta] \cap \cup_{s \in \mathbf{N}} \{\Lambda_s (\omega)\} \} $ 
		has  positive probability, since $\rm Eig$ is deterministic.  
		For each $l \in \mathbf{N}$ we use the notation
		$$P^{l}_{[\lambda-\delta,\lambda+\delta]}(\omega):=
		\begin{cases}P_{S_{[\lambda-\delta,\lambda+\delta]}}(\tilde \varphi^l)(D(\omega)) & \textrm{if } \omega \in \mathcal{O}, \\
			0 & \textrm{otherwise}. 
		\end{cases} 
		$$
		Clearly,  for each $\delta>0$  there exists
		$l(\delta)\in \mathbf{N}$ such that $P^{l(\delta)}_{[\lambda-\delta,\lambda+\delta]} \neq 0$, since
		
		$$ I_{\delta} \subset \bigcup_{l \in \mathbf{N}} \{P^{l}_{[\lambda-\delta,\lambda+\delta]} \neq 0\}. $$ Also from the definition it follows that almost surely we have
		\begin{equation} \label{1208} 
			\int_{B_{\rho}} |(-\Delta_{\mathcal{O}}P^{l(\delta)}_{[\lambda-\delta,\lambda+\delta]})(T_x \omega)-\lambda P^{l(\delta)}_{[\lambda-\delta,\lambda+\delta]}(T_x \omega) |^2 dx \leq \delta \int_{\square^4} |P^{l(\delta)}_{[\lambda-\delta,\lambda+\delta]}(T_x \omega)|^2 dx, 
		\end{equation} 
		where we used the fact that $B_{\rho}$ can intersect at most one inclusion in $\mathcal{O}_{\omega}$, which is then necessarily contained in $\square^4$. Integrating  \eqref{1208} over $\Omega$ and using Fubini's theorem, we infer that there exists $C>0$ such that 
		$$ \|-\Delta_{\mathcal{O}}P^{l(\delta)}_{[\lambda-\delta,\lambda+\delta]}-   \lambda P^{l(\delta)}_{[\lambda-\delta,\lambda+\delta]}       \|_{L^2(\O)}\leq C \delta \|P^{l(\delta)}_{[\lambda-\delta,\lambda+\delta]}\|_{L^2(\O)}. $$
		Using the fact that $\delta>0$ is arbitrary, we conclude that $\lambda \in \Sp(-\Delta_\OO)$.
	\end{proof}

 \section{Higher regularity of the corrector}\label{ap:regul}

In order to prove higher regularity of the corrector we need to recall special versions of two well-known results, namely the Poincar\'e-Sobolev inequality for perforated domains and the reverse H\"older's inequality. We begin with the former. 

The Poincar\'e-Sobolev inequality provided in Lemma \ref{l:poincare-sobolev} is valid for a more general family of perforated domains $\R^d\setminus \OO_\o$ than those considered in this paper. Namely, in Lemma \ref{l:poincare-sobolev} we assume that the set $\OO_\o$ satisfies Assumption \ref{kirill100}, but we do not make the assumption that it is generated by a dynamical system on a probability space. We still keep the notation $\OO_\o$ and $\OO_\o^k, k\in \N,$ for the set of inclusions and the individual inclusions, respectively.  In addition, we admit sets ${\mathcal O}_\omega$ comprising a finite number of inclusions.

\begin{lemma}\label{l:poincare-sobolev}
	There exist $C>0$ and $m>1$ such that for any  $f\in W^{1,p}_{loc}(\R^d\setminus \OO_\o)$, $R>0$,  one has
	\begin{equation}\label{74}
		\|f - c_R(x) \|_{L^q(B_{R}(x) \setminus  \OO_\o)} \leq C R^{d({1}/{q} - {1}/{p})+1} \|\nabla f\|_{L^p(B_{mR}(x) 	\setminus \OO_\o)}, \quad x\in \R^d,
	\end{equation}
	where $p\leq q\leq dp/(d-p)$ and $c_R(x)$ is (for a given $x$) a constant that only depends on the values of $f$ in $B_{mR}(x) 	\setminus \OO_\o$.
\end{lemma}
\begin{proof}
	For $R>0$ we introduce a rectangular box $Q_{R,\gamma}:= (-R,R)^{d-1}\times (-\gamma R, \gamma R)$. 	 Part 1 of Assumption \ref{kirill100} implies that there exists a number $R_0>0$ (the largest $R$ such that the box $Q_{R,\gamma}$ is contained in the ball $B_\rho$) such that for any inclusion $\OO_\o^k$ and any $x_0 \in \partial \OO_\o^k$ there exist local coordinates $y$ congruent to $x$ with the origin at $x_0$ such that the intersection of the set  $\R^d\setminus \OO_\o$ with  $Q_{R_0,\gamma}$ is the subgraph of a Lipschitz continuous function  $g$ with $g(0)=0$ and Lipschitz constant $\gamma$:
	\begin{equation}\label{84a}
		Q_{R_0,\gamma}\setminus \overline{\OO_\o} = \{y\in Q_{R_0,\gamma} \,|\, y_d< g(y_1,\dots, y_{d-1})\}.
	\end{equation}	
	\begin{remark}
		Obviously, the above property  also holds for  boxes $Q_{R,\gamma}$ with $0<R<R_0$. 
	\end{remark}
	
	Consider the class of sets of the form (representing a scaled version of \eqref{84a})
	\begin{equation*}
		U:= \{y\in Q_{1,\gamma} |\, y_d< g(y_1,\dots, y_{d-1})\},
	\end{equation*}
	where $g$ is as above. Clearly, the sets $U$ satisfy the  cone condition uniformly (i.e. with the cone parameters depending only on $\gamma$). It is well known (see \cite{BC} and e.g. \cite{Adams}) that in this case the Poincar\'e inequality 
	\begin{equation}\label{85}
	\Big\|f - \fint_U f \Big\|_{L^p( U)}\leq C  \|\nabla f\|_{L^p( U)},
	\end{equation}
	and the Sobolev inequality
	\begin{equation}\label{86}
		\|f \|_{L^q( U)}\leq C  \| f\|_{W^{1,p}( U)}, \,\, p\leq q\leq dp/(d-p),
	\end{equation}
	hold for any $f\in W^{1,p}(U)$ with a constant $C$ that depends only on the cone parameters (i.e. on $\gamma$) and $p$, but not on the domain $U$. 
	\begin{remark} 	
		It is not difficult to prove the uniform Poincar\'e inequality \eqref{85} directly. It follows by making a simple change of variables transforming $U$ into the fixed box $(-1,1)^{d-1}\times (-\gamma, 0)$, applying the Poincar\'e inequality on the box and transforming back to $U$. 
	\end{remark}  
	Combining \eqref{85} and \eqref{86}, we obtain the Sobolev-Poincar\'e inequality 
	\begin{equation*}
		\Big\|f - \fint_U f \Big\|_{L^q(U)} \leq C  \|\nabla f\|_{L^p(U)},\qquad p\leq q\leq dp/(d-p),
	\end{equation*}
	with $C$ independent of $U$. Observing that $B_1 \subset Q_{1,\gamma} \subset B_m$, $m=m_0:=\sqrt{d-1+\gamma^2}$, we obtain 
	\begin{equation*}
		\Big\|f - \fint_U f \Big\|_{L^q(B_1)} \leq \Big\|f - \fint_U f \Big\|_{L^q(U)} \leq C  \|\nabla f\|_{L^p(B_m)},\qquad p\leq q\leq dp/(d-p).
	\end{equation*}
	Now, taking $x\in \partial \OO_\o^k$ and $R\leq R_0$, and using a scaling argument yields \eqref{74}  with $m=m_0$, $c_R(x) = \fint_{Q_{R,\gamma}\setminus {\OO_\o}}f $, where $Q_{R,\gamma}$ is now understood in the global coordinates, i.e. is an appropriately oriented box centred at $x$.
	
	Next we argue that \eqref{74} holds (with different $m$ and $c_R(x)$) for a ball $B_R(x)$ with arbitrary $R$ and $x$. Assume first that  $B_R(x) \subset \R^d\setminus \OO_\o$, then \eqref{74}  holds with $m=1$ and $c_R(x) = \fint_{B_R(x)} f$  by the standard Poincar\'e-Sobolev inequality for a ball. For $R\leq R_0/2$ and $\emptyset \neq B_R\cap \OO_\o \neq B_R$  there exists   $x_0 \in \partial \OO_\o^k$ such that $B_R(x)\subset B_{2R}(x_0)$. We conclude that \eqref{74} holds with $m=2 m_0 +1$ and $c_R(x) = \fint_{Q_{2R,\gamma}\setminus {\OO_\o}} f$, where $Q_{2R,\gamma}$ is centred at $x_0$ and appropriately oriented. Finally, consider the case $R>R_0/2$ and, again, $\emptyset \neq B_R(x)\cap \OO_\o \neq B_R(x)$. Let $\widetilde f\in W^{1,p}_{loc}(\R^d)$ be the extension of $f$ as per Theorem \ref{th:extension}. We have 
	\begin{equation*}
		\begin{aligned}
				\Big\|f - \fint_{B_{R}(x)}\widetilde f  \Big\|_{L^q(B_{R}\setminus  \OO_\o)} \leq &\Big\|\widetilde f - \fint_{B_{R}(x)}\widetilde f  \Big\|_{L^q(B_{R}(x) )} 
				\\
				\leq& C R^{d(1/q-1/p)+1} \|\nabla \widetilde f\|_{L^p(B_{R}(x))} \leq C R^{d(1/q-1/p)+1} \|\nabla f\|_{L^p(B_{R+\sqrt d}(x)\setminus \OO_\o)},
		\end{aligned}
	\end{equation*}
	where we use the fact that for any $\OO_\o^k$ with $\OO_\o^k\cap B_R(x)\neq \emptyset$ the corresponding ``extension set'' $\mathcal{B}_\o^k$ is contained in  $B_{R+\sqrt d}(x)$ (so $c_R(x) = \fint_{B_{R}(x)} \widetilde f $). Thus \eqref{74} holds for any ball $B_R(x)$ with $m = \max\{2 m_0+1, 1 + \sqrt d/R_0\}$ and   $c_R(x)$ as specified above for each case.
\end{proof}

We will need the following simplified version of the reverse H\"older's inequality, see e.g. Theorem 1.10 and Remark 1.14 in \cite{Bensoussan}.
\begin{lemma}\label{l:reverseHolder}
	Let $k\geq 2$,  $f\in L^q(\square_{x_0}^L), q>1$, for some cube $\square_{x_0}^L$ and assume that 
	\begin{equation*}
		\fint_{\square_x^{a/k}} |f|^q   \leq c\Big( 1 +   \Big(\fint_{\square_x^{a}} |f|\Big)^q \Big)
	\end{equation*}
	for some $c>0$ and every cube $\square_x^{a}$  contained in $\square_{x_0}^L$. Then there exists a constant $\epsilon = \epsilon(q,d,c,k)>0$ such that $\forall p$ with $q\leq p < q+\epsilon$ one has
	\begin{equation*}
			\Big(\fint_{\square_x^{a/k}} |f|^p \Big)^{1/p} \leq c_p\Big( 1 +   \Big(\fint_{\square_x^{a}} |f|^q\Big)^{1/q} \Big)
	\end{equation*}
	for $\square_x^{a}\subset\square_{x_0}^L$, where $c_p$ depends on $p,q,d,c$ and $k$.
\end{lemma}

We are ready to proceed with the proof of Theorem \ref{th:6.18}. Recall that $\hat N_j \in W^{1,2}_{\rm per}(\hat\square^{M+\kappa,1}_{\overline{x}_M}), j=1,\dots,d,$ are the solutions to \eqref{88}. Without loss of generality, we may assume $\overline{x}_M =0$. To simplify the notation, in what follows we extend $\hat N_j$ and $\nabla \hat N_j$  by zero on the inclusions, and extended by periodicity to the whole of $\R^d$.   (Note that the following argument does not involve integration by parts and is therefore not affected by the possible discontinuity of the gradient.) Note that \eqref{88} implies  the  identity 
	\begin{equation*}
	\int\limits_{\R^d}  A_1  (e_j + \nabla \hat N_j) \cdot \nabla\varphi = 0  
\end{equation*}
 for all compactly supported test functions  $\varphi \in W^{1,2}(\R^d)$. 

We redefine the cut-off function as follows: let $\eta \in C^\infty_0(B_2(0))$ be  such that $0\leq \eta\leq 1$, $\eta = 1$ in $B_1$, and  set $\eta_R(x):= \eta ((x-x_0)/{R})$, for some $x_0\in \R^d$ and $R>0$, which are assumed to be fixed for the time being. We apply Lemma \ref{l:poincare-sobolev} with the perforated domain being given by the periodic extension of the set $\square^{M+\kappa,1}$ to the whole of $\R^d$.  Let  $c_R(x_0)$ be as in \eqref{74} with $f$ replaced by $\hat N_j$. 

Using  $(\hat N_j - c_R(x_0)) \eta_R^2$ as  the test function in \eqref{88} and rearranging the terms, we obtain 
	\begin{equation*}
		\int\limits_{B_{2R}(x_0)} A_1 \nabla \hat N_j\cdot \nabla \hat N_j \eta_R^2 = - \int\limits_{B_{2R}(x_0)} (A_1 e_j \cdot \nabla \hat N_j \eta_R^2 +  2 A_1  (e_j + \nabla \hat N_j) \cdot \nabla \eta_R (\hat N_j - c_R(x_0)) \eta_R),
	\end{equation*}
and hence
	\begin{equation}\label{90a}
		\int\limits_{B_{2R}(x_0)} |\nabla \hat N_j \eta_R|^2 \leq C \int\limits_{B_{2R}(x_0)}\left(|\nabla \hat N_j|\eta_R^2 + R^{-1}|\hat N_j - c_R(x_0)|(1 + |\nabla \hat N_j|)\right).
	\end{equation}
	We proceed by estimating the right-hand side of \eqref{90a} term by term. For the first term, we use the Cauchy inequality:
	\begin{equation*}
		\int\limits_{B_{2R}(x_0)}|\nabla \hat N_j|\eta_R^2 \leq \int\limits_{B_{2R}(x_0)}\left(\alpha^{-1}|\nabla \hat N_j \eta_R|^2+ \alpha \eta_R^2\right) \leq \int\limits_{B_{2R}(x_0)} \alpha^{-1}|\nabla \hat N_j \eta_R|^2 +  \alpha  |B_{2R}|,
	\end{equation*}
	for any $\alpha>0$. For the second term, applying the Cauchy inequality, \eqref{74}, and then H\"older's inequality, we obtain
	\begin{equation*}
		\begin{aligned}
			&\int\limits_{B_{2R}(x_0)} R^{-1}|\hat N_j - c_R(x_0)| \leq \int\limits_{B_{2R}(x_0)} \left( 1+R^{-2}|\hat N_j - c_R(x_0)|^2\right)
			\\
			&\leq |B_{2R}| + C R^{-2}\bigg(\int\limits_{B_{2mR}(x_0)}|\nabla \hat N_j |^{\frac{2d}{d+2}}\bigg)^{\frac{d+2}{d}}
			\leq |B_{2R}| + C R^{-1}\bigg(\int\limits_{B_{2mR}(x_0) }|\nabla \hat N_j |^{\frac{2d}{d+1}}\bigg)^{\frac{d+1}{d}}.
		\end{aligned}
	\end{equation*}
	For the final term in \eqref{90a}, we first apply H\"older's inequality with the exponents  $q= {2d}/{(d-1)}$ and $p = {2d}/{(d+1)}$ to the integral of the product of $|\hat N_j - c_R(x_0)|$ and $|\nabla \hat N_j|$ and then the inequality \eqref{74} to the term containing $|\hat N_j - c_R(x_0)|$ with the same exponents $p, q,$ wich yields
	\begin{equation*}
		\int\limits_{B_{2R}(x_0) }R^{-1}|\hat N_j - c_R(x_0)| |\nabla \hat N_j| \leq C R^{-1} \bigg(\int\limits_{B_{2mR}(x_0) }|\nabla \hat N_j |^{\frac{2d}{d+1}}\bigg)^{\frac{d+1}{d}}.
	\end{equation*}
	Combining the above estimates with \eqref{90a}, moving the term containing $\alpha^{-1}$ to the left-hand side,  choosing $\alpha$ large enough, and observing that $\square_{x_0}^{R/\sqrt{d}} \subset B_{R}(x_0)$ and $B_{2mR}(x_0) \subset \square_{x_0}^{2mR} $, we obtain
	\begin{equation}\label{197}
	\int\limits_{\square_{x_0}^{R/\sqrt{d}} } |\nabla \hat N_j|^2 \leq C |B_{2R}| + C R^{-1} \bigg(\int\limits_{\square_{x_0}^{2mR} }|\nabla \hat N_j |^{\frac{2d}{d+1}}\bigg)^{\frac{d+1}{d}}.
	\end{equation}
Denoting $z:= |\nabla \hat N_j |^{{2d}/(d+1)},$
	and dividing both sides of  \eqref{197} by $|\square^{R/\sqrt{d}}|$, we obtain
	\begin{equation*}
		\fint\limits_{\square_{x_0}^{R/\sqrt{d}}} z^\frac{d+1}{d} \leq C  + C  \bigg(\fint\limits_{\square_{x_0}^{2mR}}z \bigg)^{\frac{d+1}{d}},
	\end{equation*}
where $x_0$ and $R>0$ are arbitrary.	We can now apply the reverse H\"older inequality (Lemma \ref{l:reverseHolder}) to infer that there exists $\epsilon >0$ depending only on $C$ and $d$ such that for any $\mu$ satisfying $ (d+1)/d \leq \mu < (d+1)/d + \epsilon$ one has
	\begin{equation*}
		\bigg(\fint\limits_{\square_{x_0}^{R/\sqrt{d}}} z^\mu\bigg)^{1/\mu} \leq C  + C  \bigg(\fint\limits_{\square_{x_0}^{2mR}}z^\frac{d+1}{d} \bigg)^{\frac{d}{d+1}}.
	\end{equation*}
	Rewriting this for $\nabla \hat N_j$ yields  	
	\begin{equation*}
		\bigg(\fint\limits_{\square_{x_0}^{R/\sqrt{d}}} |\nabla \hat N_j |^p \bigg)^{1/p} \leq C  + C  \bigg(\fint\limits_{\square_{x_0}^{2mR}}|\nabla \hat N_j |^2 \bigg)^{1/2},
	\end{equation*}
	for all $2\leq p < 2 + 2d\epsilon /(d+1)$. Choosing $\square_{x_0}^{R/\sqrt{d}} = \square^{M+\kappa}$,  observing that the corresponding $\square^{2mR}$ can be covered by finitely many periodically shifted copies of $\square^{M+\kappa}$, and utilising the periodicity of the corrector, we finally obtain \eqref{89a}.
	
\section{Auxiliary results for Theorem \ref{th:9.2}}\label{ss:9.1}

Denote by $\mathcal{H}_{\rho, \mathcal N, \gamma}(\square)$ the family of all $(\rho, \mathcal N, \gamma)$ minimally smooth closed sets $P \subset \square^{1/2}$ whose interiors are connected and $\square^{1/2}\setminus {P}$  also connected. Furthermore, let $L^2_{\pot}(\square)$ be the space of potential vector fields on $\square$. The direct product $L^2_{\pot}(\square)\times \mathcal{H}_{\rho, \mathcal N, \gamma}(\square)$ is a measurable space with the product $\sigma$-algebra generated by the standard distance on $L^2(\square;\R^d)$  and the Hausdorff  distance on $\mathcal{H}_{\rho, \mathcal N, \gamma}(\square)$. For a pair $(g, P)\in L^2_{\pot}(\square)\times \mathcal{H}_{\rho, \mathcal N, \gamma}(\square)$, let $\varphi\in W^{1,2}(\square)$ be a potential of $g$, i.e. $g = \nabla \varphi$, let $\widetilde \varphi  \in W^{1,2}(\square)$ be the harmonic extension of $\varphi|_{\square\setminus P}$ to the whole of $\square$, and denote $\widetilde{g}:=\nabla\widetilde\varphi$. This construction defines a mapping $\widetilde{E}:L^2_{\pot}(\square)\times \mathcal{H}_{\rho, \mathcal N, \gamma}(\square)\to L^2_{\pot}(\square)$ by setting $\widetilde{E}(g, P)= \widetilde g$.

\begin{lemma} \label{lemmameasm}
	The mapping $\widetilde{E}$ is measurable.
\end{lemma}
\begin{proof} 
Consider a convergent sequence $(g_n,P_n) \to (g,P)$ in  $ L^2_{\pot}(\square) \times \mathcal{H}_{\rho, \mathcal N, \gamma}(\square)$, i.e.  $g_n \to g$  in $L^2(\square; \R^d)$ and $P_n \to P$ in the Hausdorff distance. The minimal smoothness assumption implies that (see the proof of Lemma \ref{l:5.21}))
\begin{equation}\label{174}
	\begin{aligned}
		|(P_n\backslash P) \cup (P\backslash P_n)|\to 0, \qquad
            d_{\rm H}(\partial P,\partial P_n) \to 0			\mbox{ as } n\to\infty.
	\end{aligned}
\end{equation} 
 We take $\varphi, \varphi_n \subset W^{1,2}(\square), n \in \mathbf{N}$, such that $g = \nabla \varphi$, $g_n= \nabla \varphi_n$ and $\int_{\square} \varphi =  \int_{\square} \varphi_n =0$. Clearly, $\varphi_n \to \varphi$  in $W^{1,2}(\square)$. By Theorem \ref{th:extension} and Remark \ref{r:extension}, the harmonic extensions $\widetilde \varphi_n$ of $\varphi_n|_{\square\setminus P_n}$ to the whole of $\square$ satisfy the uniform bound 
		$$ \|\widetilde{\varphi}_n\|_{W^{1,2}(\square)} \leq C\| \varphi_n\|_{W^{1,2}(\square)} \quad \forall n \in \N. $$ 
		It follows that $\widetilde \varphi_n$ converges to some $\hat{\varphi}\in W^{1,2}(\square)$ strongly in $L^2(\square)$ and weakly in $ W^{1,2}(\square)$. It is not difficult to see that $\hat \varphi|_{\square\setminus P} = \varphi|_{\square\setminus P}$. Multiplying the equation 
		\begin{equation}\label{176b}
			\Delta \widetilde{\varphi}_n = 0 \mbox{ on } P_n
		\end{equation}
	 by an arbitrary $f\in C_0^\infty(P')$, where $P'$ is  contained in the interior of $P$ (and thus, by \eqref{174}, $P'$ is  contained in the interior of $P_n$ for large enough $n$), integrating by parts and passing to the limit, we see that 
		$$
		 0= \lim_{n\to\infty} \int_{P'} \nabla \widetilde{\varphi}_n\cdot\nabla f = \int_{P'} \nabla \hat{\varphi}\cdot\nabla f.
		$$
	We infer that $\hat \varphi$ is harmonic in $P$, and, hence, $\hat \varphi = \widetilde \varphi$, where the latter denotes the harmonic extensions of $\varphi|_{\square\setminus P}$ onto $\square$. 
	
	It remains to prove the strong convergence of $\nabla \widetilde{\varphi}_n$ to $\nabla \widetilde{\varphi}$.  To this end, we multiply \eqref{176b} by the test function $\widetilde \varphi_n - \varphi_n$ and integrate by parts:
\begin{equation}\label{177}
		\int_{P_n} \nabla \widetilde{\varphi}_n\cdot \nabla(\widetilde \varphi_n - \varphi_n )=0.
\end{equation}
Passing to the limit as $n\to\infty$ yields
	$$
	\lim_{n\to\infty} \int_{P_n}| \nabla \widetilde{\varphi}_n|^2 = \lim_{n\to\infty} \int_{P_n} \nabla \widetilde{\varphi}_n\cdot \nabla\varphi_n = \int_{P} \nabla \widetilde{\varphi}\cdot \nabla\varphi.
$$
	The last equality follows from a simple observation that $\nabla\varphi_n {\mathbf 1}_{P_n}$ converges to $\nabla\varphi {\mathbf 1}_{P}$ in $L^2(\square)$ (utilising \eqref{174}). Similarly to  \eqref{177}, we have
	$$
	\int_{P}| \nabla \widetilde{\varphi}|^2 =  \int_{P} \nabla \widetilde{\varphi}\cdot \nabla\varphi,
	$$
	and thus $\nabla \widetilde{\varphi}_n\to\nabla \widetilde{\varphi}$ in $L^2(\square)$. Therefore, the mapping $\widetilde E$ is continuous and hence measurable.
\end{proof} 

\begin{proposition}\label{p:extensionm}
	For every $f \in \vpot\, (\lpot)$ there exists an extension $\widetilde f \in \vpot\, (\lpot)$ of $f|_{\Omega \backslash \mathcal{O}}$ onto $\Omega$ such that for a.e.  realisation $\widetilde f(T_x\o)$ its potential $\widetilde \varphi \in W^{1,2}_\loc(\R^d)$ (so $\nabla \widetilde\varphi(x) =  \widetilde f(T_{x} \omega)$) is harmonic on the set of inclusions $\OO_\o$ and
	\begin{equation}\label{118a}
		\| \widetilde f\|_{L^2( \Omega)} \leq C \|  f\|_{L^2( \O\setminus\OO)}.
	\end{equation}
\end{proposition}
\begin{proof} 
	Let $f \in \lpot$. 	The mapping from $\OO$ to $L^2_{\pot}(\square)$ that to every $\o\in\OO$ assigns the vector field $\varphi_\o(x) := f(T_{x-D} \omega) \in L^2_{\pot}(\square)$
  is measurable. (This can be easily checked for an arbitrary $f\in L^2(\Omega;\R^d)$  first by checking it for simple functions and then  using the density argument.)  using Lemmata \ref{measp} and  \ref{lemmameasm}, we conclude that the mapping $\omega \mapsto  \widetilde E(\varphi_\o,P_\o) \in L^2_{\pot}(\square)$ is measurable.
	We define
	$$ \widetilde f(\omega)=\begin{cases}   \widetilde E(\varphi_\o, P_\o)(D) & \mbox{ if } \omega \in \mathcal{O},
		\\ f(\o) & \textrm{ otherwise}.    \end{cases}. $$
	Note that thus defined $\widetilde f$ is measurable, since $\widetilde E(\varphi_\o, P_\o)$ is harmonic and hence infinitely smooth in $P_\o$ as well as an element of $\lpot$ by construction.
	
	We next prove \eqref{118a}. By  Assumption \ref{kirill100}, for a.e. $\o$ the ball $B_\rho$ intersects at most one inclusion $\OO_\o^k$. At the same time, the cube $\square^2$ contains this inclusion together with its extension domain $\mathcal B_\o^k$. By Theorem \ref{th:extension}, one has
	\begin{equation*} 
		\int_{B_\rho} |\widetilde f(T_x \omega)|^2 \leq C_{\rm ext}\int_{{\square}^2 \backslash \mathcal{O}_{\omega}} |f(T_x \omega)|^2 =C_{\rm ext}\int_{{\square}^2} |(f {\mathbf 1}_{\Omega \backslash \mathcal{O}})(T_x \omega)|^2. 
	\end{equation*} 
	Integrating over $\Omega$ and applying  the Fubini theorem yields
	$$ |B_\rho| \int_{\Omega}  | \widetilde f|^2 \leq  2^d C_{\rm ext}\int_{\Omega\setminus \OO} |f|^2. $$

	Note that, by the construction of $\widetilde{f},$ if  $\varphi \in W^{1,2}_\loc(\R^d)$ is such that $\nabla \varphi(x) =  f(T_{x} \omega)$, then  $\widetilde f(T_x \o) = \nabla\widetilde \varphi(x)$, where $\widetilde \varphi$ is the harmonic  extension of $\varphi|_{\R^d\setminus \OO_\o}$ into the set of inclusions.  
	
	It remains to show that  $\widetilde f \in \vpot$ if $ f \in \vpot$. Let $\varphi^\e, \widetilde\varphi^\e  \in W^{1,2}(\square)$ be  sequences such that $\nabla \varphi^\e(x) = f(T_{x/\e}\omega)$ and $\nabla \widetilde \varphi^\e(x) = \widetilde f(T_{x/\e}\omega)$ in $\square$. By the ergodic theorem, one has
	\begin{equation}\label{119a}
		\int_\square \nabla \varphi^\e \to \int_\O f = 0  \mbox{ as } \e \to 0.
	\end{equation}
	Using the fact that by construction $\widetilde \varphi^\e = \varphi^\e$ on $\square\setminus S_0^\e$ and integrating by parts, we obtain
	\begin{equation*}
		\int_{\OO_\o^k} \nabla \widetilde  \varphi^\e= \int_{\partial \OO_\o^k} \widetilde \varphi^\e {\bf n} = \int_{\partial \OO_\o^k}  \varphi^\e {\bf n} = \int_{\OO_\o^k} \nabla \varphi^\e \qquad \forall \OO_\o^k \subset \square,
	\end{equation*}
	where ${\bf n}$ is the outward unit  normal vector. Therefore,
	\begin{equation}\label{161}
		\int_{\square\setminus K^\e} \nabla \widetilde  \varphi^\e = \int_{\square\setminus K^\e} \nabla  \varphi^\e,
	\end{equation}
	where $K^\e : = \bigcup\limits_{\OO_\o^k \cap \partial \square \neq \emptyset} \OO_\o^k$. It is then not difficult to see that  
	\[
 \lim_{\e\to 0}	\int_{\square\cap K^\e} \nabla \varphi^\e = \lim_{\e\to 0}	\int_{\square\cap K^\e}\nabla \widetilde  \varphi^\e = \lim_{\e\to 0}	\int_{\square\cap K^\e} (\nabla \varphi^\e -\nabla \widetilde  \varphi^\e)=0.
	\]  
	Combining this with \eqref{161} and \eqref{119a}, we conclude that
	\begin{equation*}
		\int_\O \widetilde f  = \lim_{\e\to 0} \int_\square \nabla \widetilde  \varphi^\e =0.
	\end{equation*}
\end{proof}

\begin{lemma}\label{l7.6}
	For every $f\in {\mathcal X}$ there exists the extension $\widetilde{f} \in \vpot$, $\widetilde f=f$ on $\Omega \backslash \mathcal{O}$, such that for a.e. realisation $\widetilde f(T_x\o)$ its potential $\widetilde \varphi \in W^{1,2}_\loc(\R^d)$, $\nabla \widetilde\varphi(x) =  \widetilde f(T_{x} \omega)$, is harmonic on the set of inclusions $\OO_\o$, and
	\begin{equation*}
		\| \widetilde f\|_{L^2( \Omega)} \leq C \|  f\|_{L^2( \O\setminus\OO)}.
	\end{equation*}
\end{lemma}
\begin{proof}
	Let $(f_k)$ be a sequence in $\vpot$ such that $\|f - f_k\|_{L^2( \O\setminus\OO)}$ converges to zero. For each $k$ let  $\widetilde f_k \in \vpot$ be the extension of $f_k$ as in Proposition \ref{p:extensionm}. By virtue of \eqref{118a},  $(\widetilde f_k)$ is a Cauchy sequence and therefore it  converges in $L^2(\O)$ to some $\widetilde f \in\vpot$. Clearly, one then has  $\widetilde f=f$ on $\Omega \backslash \mathcal{O}$. 
	
	The remaining part of the statement can be proven similarly to the argument used to show that $\hat \varphi$ is harmonic in $P$ in the proof of Lemma \ref{lemmameasm}.
\end{proof}

In the next lemma  we  work with spaces of  complex valued functions.

\begin{lemma}\label{l12.2}
	Let $w$ be a zero-mean solenoidal random vector field, $w \in \mathcal V^2_{\rm sol}(\Omega)$. Then there exists a random tensor field $W_{ikl},\,i,k,l=1,\dots,d,$ such that for all $i$ and $k$ the random field $W_{ik\cdot}:=(W_{ik1}, \dots,W_{ikn})^T$ is an element of $ \mathcal V^2_{\rm pot}$, for fixed $l$ the field   $W_{ikl}$ is antisymmetric, i.e. $W_{ikl} = - W_{kil}$, and $	w_k =  W_{iki},\, k=1,\ldots,d.$ (We use the usual summation convention.)
\end{lemma}
\begin{proof}
	 Since the differentiation operators ${\rm i} D_k,\, k = 1,\dots,d,$ are self-adjoint and commuting, by the spectral theorem there exists a measure space $(M,\mu)$ with  a finite measure $\mu$, a unitary operator $U: L^2(\Omega) \to L^2(M)$, and a.e. finite real-valued functions $a_k,$ $k=1,\dots,d,$ on $M$ such that $U{\rm i} D_k \varphi = a_k U \varphi$ for all $\varphi\in L^2(\Omega)$. 
	
	It is easy to see that $\ker \nabla_\o$ consists of constants, so $w \perp \ker \nabla_\o$ since $w$ has mean zero. Note that $f\in \ker \nabla_\o$ if and only if $a_k Uf=0,\,k=1,\ldots,d$. Therefore, the linear set $U(\ker \nabla_\o) = \cap_k \ker a_k$ consists of all $L^2(M)$ functions that vanish on $M\setminus  \{a = 0\}$, where $a:=(a_1,\dots,a_d)^T$. Since $U w \perp  U(\ker \nabla_\o)$, we conclude that 
	\begin{equation}\label{122}
		Uw = 0 \mbox{ on }  \{a = 0\}.
	\end{equation}	
	Furthermore, since $w$ is solenoidal, one has 
	\begin{equation*}
		\int_{\Omega} w \cdot {\nabla_{\omega} \varphi} = 0 \quad \forall\varphi\in W^{1,2}(\Omega),
	\end{equation*}
	and hence
	\begin{equation*}
		\int_{M} U w \cdot \psi a=0 \quad\forall\psi\in L^2(M) \mbox{ such that } \psi a\in (L^2(M))^d.
	\end{equation*}
	Since $W^{1,2}(\Omega)$ is dense in $L^2(\Omega)$, the set $\{U\varphi:\, \varphi\in W^{1,2}(\Omega)\}=\{\psi\in L^2(M): \psi a\in (L^2(M))^d\}$ is dense in $L^2(M)$, and hence 
	\begin{equation}\label{118}
		U w\cdot a = 0 \mbox{ \ae}
	\end{equation} 
	
	Define an antisymmetric in $ik$ tensor field $	\tilde W_{ikl}$ with entries in $ L^2(M)$ as follows ($i,k,l=1,\dots,d$):
	\begin{equation*}
		\tilde W_{ikl} := \left\{\begin{array}{cc}
			a_l\,\dfrac{a_i U w_k - a_k U w_i}{|a|^2} & \quad \mbox{ if }\ |a|\neq0,
			\\
			\\
			0 & \quad \mbox{ otherwise. }
		\end{array} \right.
	\end{equation*}
	Note that $\tilde W_{iki} = U w_k$ by \eqref{122}, \eqref{118}. We set $W_{ikl} : = U^{-1} \tilde W_{ikl}$. It only remains to show that  $W_{ik\cdot}\in \mathcal V^2_{\rm pot}$, which follows from 	the fact that for all $\varphi\in L^2_{\rm sol}(\Omega)$ one has $a\cdot U\varphi=0$ a.e. (cf. \eqref{118} with $w$ replaced by $\varphi$): 
	\begin{equation*}
		\int_{\Omega}  W_{ikl} \varphi_l =  \int_M   a_l\,	 \frac{a_i U w_k - a_k U w_i}{|a|^2}\,U\varphi_l = \int_M  \frac{a_i U w_k - a_k U w_i}{|a|^2}\, a\cdot U\varphi =0\quad\forall i,k=1,\dots,d.
	\end{equation*}
\end{proof}

\begin{corollary}\label{c12.3}
	Let $g_j^\e$, $j=1,\dots,d,$ be as in  \eqref{128a}. There exist skew-symmetric tensor fields $G^\e_j\in W^{1,2}(\square^L;\R^{d\times d})$, $(G^\e_j)_{ik} = -(G^\e_j)_{ki}$, such that 
	\begin{equation}\label{119}
		g^\e_j = \nabla\cdot G_j^\e ,\quad j=1,\dots,d,
	\end{equation}
	i.e. $(g_j^\varepsilon)_k=\partial_i(G_j^\varepsilon)_{ik}$ for all $k=1,\dots, d,$ and  
	\begin{equation}\label{89}
		\|G^\e_j\|_{L^2(\square^L)} \to 0 \quad \mbox{as } \e\to 0.
	\end{equation}
\end{corollary}
\begin{proof}
	Applying Lemma \ref{l12.2} to $g_j$ yields $g_j = W_{iki}$, where we drop the index $j$ on the right-hand side for simplicity. Since   $W_{ik\cdot}\in \mathcal V^2_{\rm pot}$ for all $i,k,$ there exists a zero-mean function $(G^\e_j)_{ik}\in W^{1,2}(\square^L)$ such that $W_{ik\cdot}(T_{x/\e} \o) = \nabla (G^\e_j)_{ik}(x)$ (i.e. $W_{ikl}(T_{x/\e} \o) = \partial_l (G^\e_j)_{ik}$). Obviously, $G^\e_j$ is skew-symmetric and \eqref{119} holds. The convergence \eqref{89} follows from two observations: a) since $\langle W_{ikl}\rangle = 0$, the sequence $\partial_l (G^\e_j)_{ik}$ converges to zero weakly in $L^2(\square^L)$ by the ergodic theorem; b) functions  $(G^\e_j)_{ik}$ have zero mean over $\square^L$.
\end{proof}

\begin{lemma}\label{l13.3}Let $K$ be a bounded domain, and suppose that a sequence $(f^\varepsilon)\subset L^2(K)$ is bounded and converges weakly to zero as $\varepsilon\to0.$ Then there exists a sequence of zero-mean solutions $B^\e \in W^{1,2}(K;\R^d)$ of 
	\begin{equation}\label{128}
		\nabla\cdot B^\e = f^\e \mbox{ in } K
	\end{equation}
	such that 
	\begin{equation}\label{101}
		\|B^\e\|_{L^2(K)} \to 0.
	\end{equation}
	Moreover, here exists $C>0$ such that
	\begin{equation}\label{134}
		\|B^\e\|_{W^{1,2}(K)} \leq C\diam(K)	\|f^\e\|_{L^2(K)}.
	\end{equation}
\end{lemma}
\begin{proof}
	By extending $f^\e$ by zero into a cube containing $K$ and using  appropriate scaling and translation arguments, it is sufficient to prove the statement of the lemma for the case $K=\square^{2\pi}$.	To that end, consider the Fourier series for $f^\e$:
	\begin{equation*}
		f^\e = \sum_{k\in \mathbb Z^d} f^\e_k \exp(ik\cdot x).
	\end{equation*}
	Since  $ f^\e $ converges weakly to zero, its Fourier coefficients   also converge to zero:
	\begin{equation}\label{90}
		f^\e_k \to 0 \mbox{ as } \e\to 0, \quad k\in \mathbb Z^d.
	\end{equation} 
	Define $B^\e$ as follows:
	\begin{equation*}
		B^\e (x): = \frac{f^\e_0 \,x}{d} - \sum_{k\in \mathbb Z^d, k\neq 0}  \frac{ik \, f^\e_k}{|k|^2} \,\exp(ik\cdot x) \qquad x\in\square^{2\pi}.
	\end{equation*}
	By direct inspection, $B^\e$ has zero mean, satisfies the equation \eqref{128} and  bound \eqref{134}. Since $\sum_{k} |f^\e_k|^2$ is  bounded uniformly in $\e$,  for all $\delta>0$ there exists $k_0\in \N$ such that 
	\begin{equation*}
		\sum_{|k|\geq k_0} \frac{|f^\e_k|^2}{|k|^2} \leq \frac{1}{k_0^2} \sum_{|k|\geq k_0} {|f^\e_k|^2}< \frac{\delta^2}{4}
	\end{equation*} 
	for all $\e$. Furthermore,   by virtue of \eqref{90} one has, for sufficiently small $\varepsilon,$
	\begin{equation*}
		\frac{2^d \pi^{d+2}}{3d}|f_0^\e|^2 + \sum_{|k|< k_0,k\neq 0} \frac{|f^\e_k|^2}{|k|^2} < \frac{\delta^2}{4}.
	\end{equation*} 
	It follows that, for sufficiently small $\varepsilon,$
	\begin{equation*}
		\|B^\e\|_{L^2(K)} \leq \sqrt{2} \bigg( \frac{2^d \pi^{d+2}}{3d}|f_0^\e|^2 + \sum_{k\in \mathbb Z^d\setminus\{0\}} \frac{|f^\e_k|^2}{|k|^2} \bigg)^{1/2} < \delta,	
	\end{equation*}
	which implies  \eqref{101}. 
\end{proof}

\begin{remark}
		In the case when $f^\e$ is the $\e$-realisation of a zero-mean function from $L^2(\Omega)$, one can prove the first part of the statement of Lemma \ref{l13.3} (without the bound \eqref{134})  via an argument similar to that used in Lemma \ref{l12.2} and Corollary \ref{c12.3}. Namely, for a zero-mean $f\in L^2(\Omega)$, the potential fields $F_{i\cdot}:=(F_{i1}, \dots,F_{in})^T\in \mathcal V^2_{\rm pot}, i = 1,\dots,d$, defined via 
	\begin{equation*}
		U F_{ik} := \left\{\begin{array}{cc}
			\dfrac{a_i a_k U f }{|a|^2} \quad &   {\rm if}\  |a|\neq0,
			\\
			\\
			0 & \mbox{ otherwise. }
		\end{array} \right.
	\end{equation*}
	satisfies $f =  F_{ii}$. Then the zero-mean field $B^\e = (B^\e_1,\dots,B^\e_n)^T \in W^{1,2}(K;\R^d)$ is defined by  $\nabla B^\e_i = F^\e_{i\cdot}$. Note that thus defined $B^\e$ does not necessary coincide with the one in the above lemma. We prefer the method used in the lemma since, first, it applies to a sequence which is not necessary the $\e$-realisation of a zero-mean function from $L^2(\Omega)$, and, second, it provides the bound \eqref{134}. 
\end{remark}

\begin{lemma}\label{l9.9}
	The following bounds hold for the $\e$-realisations of $b$:
	\begin{equation*}
		\|b^\e\|_{L^2(\square^L)}\leq \frac{C L^{d/2}}{d_\l},\qquad
		\|\e \nabla b^\e\|_{L^2(\square^L)}\leq C L^{d/2} \left(\frac{ |\l| }{(d_\l)^2} + \frac{ 1}{d_\l}\right)^{1/2}.
	\end{equation*} 
\end{lemma}
\begin{proof}
	The rescaled function $b_{\l}^\e(\e y)$  is the solution to
	\begin{equation}\label{b_e}
		(-\Delta_{\mathcal{O}^k_{\omega}} - \l) b_{\l}^\e(\e\cdot) = {\mathbf 1}_{\OO^k_\o}(\cdot)
	\end{equation}
	on every $\OO^k_\o$. By virtue of   \eqref{20a}, we have  
	\begin{equation*}
		\|b_{\l}^\e(\e\cdot) \|_{L^2(\OO^k_\o)} \leq d_\l^{-1}{ |\OO^k_\o|^{1/2}}.
	\end{equation*}
	Multiplying \eqref{b_e} by $b_{\l}^\e(\e\cdot)$ and integrating by parts yields
	\begin{equation*}
		\|\nabla b_{\l}^\e(\e\cdot) \|^2_{L^2(\OO^k_\o)} = \int_{\OO^k_\o} (\l |b_{\l}^\e(\e\cdot)|^2 + b_{\l}^\e(\e\cdot))  \leq d_\l^{-2}{ |\l| |\OO^k_\o|} + d_\l^{-1}{ |\OO^k_\o|}.
	\end{equation*}
	The statement now follows by a rescaling argument and taking the norm over $\square^L$.
\end{proof}

 \section{Other auxiliary results}\label{ap:bconv}
 
 \begin{lemma}\label{l:E.1}
 Let $\AA$ be a   non-negative self-adjoint operator in a Hilbert space $H$. Let $a(\cdot,\cdot)$ be the sesquilinear form associated with the operator $\AA+1$. Assume that for some $u\in {\rm Dom}(a)$, $\|u\|=1$,  $\l\in \R$, and $0<\epsilon<1$ we have
 \[
  |a(u,v) - (\l +1) (u,v)| \leq \epsilon \sqrt{a(v,v)} \quad \forall v\in {\rm Dom}(a).
 \] 
 Then the exists $\hat u \in {\rm Dom}(\AA)$ such that 
 \begin{equation}\label{195}
 	\|\hat u -u\|\leq \epsilon,
 \qquad
 \frac{	\| (\AA - \l)\hat  u\|}{\|\hat u\|} \leq |\l +1| \frac{\epsilon}{1-\epsilon}. 
 \end{equation}
In particular, one has 
 \begin{equation*}
	\dist(\l,\Sp(\AA)) \leq |\l +1| (1-\epsilon)^{-1} \epsilon. 
\end{equation*}
 \end{lemma}
  \begin{proof}
 	Defining $\hat u: =  (\l+1)(\AA + 1)^{-1} u$, we have $a(u - \hat u,v) = a(u,v) - (\l +1) (u,v).$
Substituting $v= u - \hat u$, we obtain $a(u - \hat u, u - \hat u)  \leq \epsilon \sqrt{a(u - \hat u, u - \hat u)}$, and hence
 \begin{equation*}
	\|\hat u -u\| \leq \sqrt{a(u - \hat u, u - \hat u)} \leq \epsilon.
\end{equation*}
Finally, the triangle inequality yields $	\|\hat u \| \geq 1 - \epsilon,$
which, together with the obvious equality $(\AA - \l)\hat u = (\l +1)(u - \hat u)$
implies the second inequality in \eqref{195}.
 \end{proof}

 \begin{lemma}\label{solta10} 
 	Let $1 \leq p < \infty$ and $X \subset \R^d$ be such that continuous functions with compact support are dense in $L^p(X)$. Let $\{X^\e_k\}_{k \in \mathbf{N},\e>0}$ be a family of subsets of $X$ that are mutually disjoint for every $\varepsilon$ and are such that  $\lim_{\varepsilon \to 0} \sup_{k \in \mathbf{N}} \diam X^\e_k=0$. Consider the ``local averaging'' operators $P^\varepsilon$ mapping $f\in L^p(X)$ to the function 
 	$$ 
 	P^{\varepsilon}f(x):=\begin{cases}|X^\e_k|^{-1}\int_{X^\e_k} f &\ \  \textrm{ if } x \in X^\e_k, \\[0.35em] 
 		f &\ \  \textrm{ otherwise}.   \end{cases}  
 	$$
 	Then $P^{\varepsilon}f\to f$ strongly in $L^p(X)$ as $\varepsilon\to0$ for every $f \in L^p(X).$
 \end{lemma}
 \begin{proof}
 	For $f\in L^p(X)$ we  take a sequence $f_n \in C_c (\R^d)$ such that $f_n \to f$ strongly in $L^p(X).$ 
 	Because of the uniform continuity of $f_n$ we obviously have $P^{\varepsilon} f_n \to f_n $ strongly in $L^p(X)$ as $\varepsilon \to 0$ for a fixed $n$. Fix $\delta>0$ and $n \in \mathbf{N}$ such that 
 	$
 	\|f-f_n\|_{L^p(X)} \leq \delta/3. $
 	Take also $\varepsilon_0>0$ such that
 	$ \|f_n-P^{\varepsilon} f_n\|_{L^p(X)} <\delta/3, \quad 0<\varepsilon< \varepsilon_0. $
 	Using H\"older's inequality on each $X^\e_k$, we obtain
 	\begin{equation*}
\begin{aligned}
	 		\|P^{\varepsilon} f_n-P^{\varepsilon}f \|^p_{L^p(X^\e_k)}& =\frac{1}{|X^\e_k|^{p-1}}\bigg|\int_{X^\e_k}f_n- \int_{X^\e_k} f \bigg|^p
 		\\
 	&	\leq  \frac{1}{|X^\e_k|^{p-1}}|X^\e_k|^{p-1} \int_{X^\e_k}|f_n-f|^p=	\|f_n-f \|^p_{L^p(X^\e_k)},
\end{aligned}
 	\end{equation*} 
 	and hence
 	$
 	\|P^{\varepsilon} f_n-P^{\varepsilon}f \|_{L^p(X)} \leq \delta/3.
 	$
Now, the triangle inequality yields
 	$$\| f-P^{\varepsilon} f\|_{L^p(X)} \leq \| f-f_n \|_{L^p(X)}+\| P^{\varepsilon} f_n-P^{\varepsilon} f \|_{L^p(X)}+\| f_n-P^{\varepsilon} f_n\|_{L^p(X)} \leq \delta, \quad \forall \varepsilon< \varepsilon_0.       
 	$$  
 	The claim follows by the fact that $\delta$ is arbitrary.
 \end{proof}

\begin{lemma}\label{l:bconv}
	Let $U$ be a bounded open set in $\R^d$ and let $U_n, \, n\in \N$ be a sequence of open sets converging to $U$ as $n\to \infty$ in the Hausdorff metric, such that $U$ and $U_n, \, n\in \N$, are $(\rho, \mathcal N, \gamma)$ minimally smooth, and $U,U_n,\R^d\setminus \overline{U}, \R^d\setminus \overline{U_n}, n\in \N,$ are connected.  Assume that $\l$ is uniformly away from the spectra of  the Dirichlet Laplacian operators on each of the sets $-\Delta_U$, $U_n, \, n\in \N$. Let $f\in L^2(V)$, where the open set $V$ is large enough to contain $U$ and $U_n, \, n\in \N$. Then the solutions $\phi_n\in W_0^{1,2}(U_n)$ to
	\begin{equation*}
		- \Delta \phi_n - \l\phi_n = f \mbox{ in } U_n
	\end{equation*}
converges to the solution $\phi\in W_0^{1,2}(U)$ to
	\begin{equation}\label{225}
	- \Delta \phi - \l\phi = f \mbox{ in } U
\end{equation}
weakly in $W^{1,2}(V)$ and strongly in $L^2(V)$, after extending $\phi_n$, $\phi$ by zero outside $U_n$, $U$, respectively.
\end{lemma}
\begin{proof}
	Clearly, the solutions $\phi_n$ are uniformly bounded in $W^{1,2}(V)$ and hence they converge (up to a subsequence) weakly in $W^{1,2}(V)$ and strongly in $L^2(V)$ to some $\hat \phi \in W^{1,2}(V)$. It is not difficult to see that $\hat \phi \in W^{1,2}_0(U)$.	Let $U'$ be an arbitrary open set such that $\overline{U'}\subset U$. Then, as a consequence of the minimal smoothness assumption, for sufficiently large $n$ we have $U'\subset U_n$ (see the proof of Lemma \ref{l:5.21})). Let $\psi$ be an arbitrary test function from $W^{1,2}_0(U')$ extended by zero outside $U'$. We have 
		\begin{equation*}
		\int\limits_{U} (\nabla\phi_n\cdot \nabla \psi - \l\phi_n\psi) = \int\limits_{U} f \psi
	\end{equation*}
for sufficiently large $n$. Passing to the limit as $n\to\infty,$ we obtain
\begin{equation*}
	\int\limits_{U} (\nabla\hat\phi\cdot \nabla \psi - \l\hat\phi\psi) = \int\limits_{U} f \psi.
\end{equation*}
It follows from the density of smooth compactly supported in $U$ functions in  $W^{1,2}_0(U)$  that $\hat\phi$ is the solution of \eqref{225}, i.e. $\hat\phi = \phi$.

Since we can apply the above argument to an arbitrary subsequence,  the convergence holds for the entire sequence $\phi_n$.
\end{proof}


\end{document}